\pdfoutput=1
\documentclass{scrartcl}
\usepackage{amsmath,amsthm,amssymb}

\usepackage{txfonts}
\usepackage[T1]{fontenc}
\usepackage{lmodern}
\usepackage{array,breqn,bm,ifpdf,authblk,relsize,extpfeil,accents,pdflscape}
\usepackage[scr=boondoxo]{mathalpha}
\usepackage{eucal}
\usepackage{longtable}
\usepackage[abbrev,nobysame,lite]{amsrefs}
\usepackage[inline,shortlabels]{enumitem}
\usepackage{microtype}
\usepackage{tkz-graph}
\usepackage{hyperref}

\usepackage{tocbasic}
\addtocontents{toc}{\protect\renewcommand*\protect\addvspace[1]{}}

\addtokomafont{sectionentry}{\normalfont\small}
\DeclareTOCStyleEntry[
  linefill=\dotfill,
  beforeskip=200pt
]{tocline}{section}

\allowdisplaybreaks[1]

\definecolor{brightmaroon}{rgb}{0.76, 0.13, 0.28}
\definecolor{darkolivegreen}{rgb}{0.33, 0.42, 0.18}

\makeatletter
\DeclareOldFontCommand{\rm}{\normalfont\rmfamily}{\mathrm}
\DeclareOldFontCommand{\sf}{\normalfont\sffamily}{\mathsf}
\DeclareOldFontCommand{\bf}{\normalfont\bfseries}{\mathbf}
\DeclareOldFontCommand{\it}{\normalfont\itshape}{\mathit}

\DeclareFontFamily{U}{BOONDOX-calo}{\skewchar\font=45 }
\DeclareFontShape{U}{BOONDOX-calo}{m}{n}{
  <-> s*[1.05] BOONDOX-r-calo}{}
\DeclareFontShape{U}{BOONDOX-calo}{b}{n}{
  <-> s*[1.05] BOONDOX-b-calo}{}
\DeclareMathAlphabet{\mathcalboondox}{U}{BOONDOX-calo}{m}{n}
\SetMathAlphabet{\mathcalboondox}{bold}{U}{BOONDOX-calo}{b}{n}
\DeclareMathAlphabet{\mathbcalboondox}{U}{BOONDOX-calo}{b}{n}

\makeatletter
\newcommand*{\da@rightarrow}{\mathchar"0\hexnumber@\symAMSa 4B }
\newcommand*{\da@leftarrow}{\mathchar"0\hexnumber@\symAMSa 4C }
\newcommand*{\xdashrightarrow}[2][]{%
  \mathrel{%
    \mathpalette{\da@xarrow{#1}{#2}{}\da@rightarrow{\,}{}}{\!\!}%
  }%
}
\newcommand{\xdashleftarrow}[2][]{%
  \mathrel{%
    \mathpalette{\da@xarrow{#1}{#2}\da@leftarrow{}{}{\,}}{}%
  }%
}
\newcommand*{\da@xarrow}[7]{%
  \sbox0{$\ifx#7\scriptstyle\scriptscriptstyle\else\scriptstyle\fi#5#1#6\m@th$}%
  \sbox2{$\ifx#7\scriptstyle\scriptscriptstyle\else\scriptstyle\fi#5#2#6\m@th$}%
  \sbox4{$#7\dabar@\m@th$}%
  \dimen@=\wd0 %
  \ifdim\wd2 >\dimen@
    \dimen@=\wd2 %
  \fi
  \count@=2 %
  \def\da@bars{\dabar@\dabar@}%
  \@whiledim\count@\wd4<\dimen@\do{%
    \advance\count@\@ne
    \expandafter\def\expandafter\da@bars\expandafter{%
      \da@bars
      \dabar@ 
    }%
  }%
  \mathrel{#3}%
  \mathrel{%
    \mathop{\da@bars}\limits
    \ifx\\#1\\%
    \else
      _{\copy0}%
    \fi
    \ifx\\#2\\%
    \else
      ^{\copy2}%
    \fi
  }%
  \mathrel{#4}%
}
\makeatother

\newtheorem{thm}{Theorem}[section]
\newtheorem{lemma}[thm]{Lemma}
\newtheorem{claim}[thm]{Claim}
\newtheorem{prop}[thm]{Proposition}
\newtheorem{cor}[thm]{Corollary}

\newtheorem{thmabc}{Theorem}
\newtheorem{propabc}[thmabc]{Proposition}

\theoremstyle{definition}
\newtheorem{defn}[thm]{Definition}
\newtheorem{ex}[thm]{Example}
\newtheorem{rem}[thm]{Remark}

\numberwithin{equation}{section}
\allowdisplaybreaks[1]

\renewcommand{\le}{\leqslant}
\renewcommand{\ge}{\geqslant}

\def\emph{}
\DeclareTextFontCommand{\bfemph}{\bfseries}
\DeclareTextFontCommand{\itemph}{\itshape}
\def\emph{\bfemph}

\newcommand{\mon}[1]{\operatorname{mon}({#1})}

\newcommand{\DVR}{\textsf{\textup{DVR}}}

\newcommand{\PID}{\textsf{\textup{PID}}}

\newextarrow{\xbigtoto}{{20}{20}{20}{20}}
{\bigRelbar\bigRelbar{\bigtwoarrowsleft\rightarrow\rightarrow}}

\makeatletter
\def\blankfootnote{\xdef\@thefnmark{}\@footnotetext}

\newcommand*{\textlabel}[2]{%
  \edef\@currentlabel{#1}
  \phantomsection
  #1\label{#2}
}
\makeatother

\newcommand{\parto}{\dashrightarrow}
\newcommand{\xto}{\xrightarrow}
\newcommand{\xparto}{\xdashrightarrow}

\newcommand{\precdot}{\prec\mathrel{\mkern-5mu}\mathrel{\cdot}}

\newcommand{\lte}{\preccurlyeq}
\newcommand{\notlte}{\npreccurlyeq}
\newcommand{\lt}{\prec}
\newcommand{\cover}{\precdot}

\newcommand{\last}{\mathsf{last}}

\newcommand{\olte}{\sqsubseteq}

\newcommand{\join}{\vee}

\DeclareMathOperator{\dist}{d}
\DeclareMathOperator{\ask}{ask}
\DeclareMathOperator{\cc}{cc}

\DeclareMathOperator{\diag}{diag}
\DeclareMathOperator{\GL}{GL}
\DeclareMathOperator{\Coker}{Coker}
\DeclareMathOperator{\Ker}{Ker}

\DeclareMathOperator{\Mat}{M}
\DeclareMathOperator{\rank}{rk}
\DeclareMathOperator{\trace}{trace}

\DeclareMathOperator{\Nbh}{N}
\DeclareMathOperator{\nbhcnt}{d}

\DeclareMathOperator{\CG}{K}
\DeclareMathOperator{\DG}{\Delta}

\newcommand{\WOhat}{\widehat{\textup{WO}}}

\newcommand{\ord}{\nu}

\newcommand{\Cycle}{\mathrm{C}}

\DeclareMathOperator{\Real}{Re}

\DeclareMathOperator{\outdeg}{outdeg}
\DeclareMathOperator{\source}{s}
\DeclareMathOperator{\target}{t}

\newcommand{\card}[1]{\lvert#1\rvert}
\DeclarePairedDelimiter{\abs}{\lvert}{\rvert}
\DeclarePairedDelimiter{\norm}{\lVert}{\rVert}
\newcommand{\pth}[1]{\check{#1}}

\newcommand{\Ideal}{\mathfrak{I}}

\newcommand{\llb}{\ensuremath{[\![ }}
\newcommand{\rrb}{\ensuremath{]\!] }}

\newcommand{\QQ}{\mathbf{Q}}
\newcommand{\FF}{\mathbf{F}}
\newcommand{\GG}{\mathbf{G}}
\newcommand{\NN}{\mathbf{N}}
\newcommand{\RR}{\mathbf{R}}
\newcommand{\ZZ}{\mathbf{Z}}
\newcommand{\CC}{\mathbf{C}}

\newcommand{\sA}{\mathsf{A}}
\newcommand{\sC}{\mathsf{C}}

\DeclareMathOperator{\dd}{d}

\newcommand{\fO}{\ensuremath{\mathfrak{O}}}
\newcommand{\fP}{\ensuremath{\mathfrak{P}}}

\newcommand{\udef}{\ensuremath{\bot}}
\DeclareMathOperator{\Dom}{\mathcal{D}}

\newcommand{\per}{\mathrm{per}}
\newcommand{\tra}{\mathrm{tra}}

\DeclareMathOperator{\odlen}{\mathscr{o\!l}}

\newcommand{\Cone}{\mathscr{C}}

\DeclareMathOperator{\concnt}{k}

\newcommand{\Eta}{\ensuremath{\mathsf H}}
\newcommand{\Tau}{\ensuremath{\mathsf T}}
\newcommand{\Zeta}{\ensuremath{\mathsf Z}}

\newcommand{\increl}{\mathrel{\imath}}

\DeclareMathOperator{\isolated}{iso}
\DeclareMathOperator{\empties}{emp}

\DeclareMathOperator{\Selectors}{Sel}
\DeclareMathOperator{\Animations}{Ani}
\DeclareMathOperator{\Nilpotents}{Nil}
\DeclareMathOperator{\Oddperiods}{Odd}
\DeclareMathOperator{\Fixed}{Fix}
\DeclareMathOperator{\Parfun}{Par}

\DeclareMathOperator{\Adj}{\mathscr{A\!d\!j}}
\DeclareMathOperator{\Inc}{\mathscr{I\!n\!c}}

\newcommand{\pred}{\mathsf{pred}}
  
\newcommand{\Minors}{\ensuremath{\mathfrak{I}}}

\newcommand{\std}[1]{\ensuremath{\mathsf{b}_{{#1}}}}

\newcommand{\Iverson}[1]{\ensuremath{\bm{[}{#1}\bm{]}}}

\newcommand{\subs}{\gets}

\newcommand{\Flags}{\mathcal{F}}

\ifpdf
  \hypersetup{pdftitle={Ask zeta functions of joins of graphs},
    pdfauthor={Tobias Rossmann, Christopher Voll}
  }
\fi

\title{Ask zeta functions of joins of graphs}

\author{Tobias Rossmann and Christopher Voll}

\date{}

\begin{document}
\thispagestyle{empty}

\maketitle

\thispagestyle{empty}
\vspace*{-2em}
\begin{abstract}
  \small
  In previous work~\cite{cico}, we studied rational generating functions
  (``ask zeta functions'') associated with graphs and hypergraphs.
  These functions encode average sizes of kernels of generic matrices with
  support constraints determined by the graph or hypergraph in question,
  with applications to the enumeration of linear orbits and conjugacy classes
  of unipotent groups.

  In the present article,
  we turn to the effect of a natural
  graph-theoretic operation on associated ask zeta functions.
  Specifically, we show that two instances of rational functions,
  $W^-_\Gamma(X,T)$ and $W^\sharp_\Gamma(X,T)$, associated with a
  graph $\Gamma$ are both well-behaved under taking
  joins of graphs.  In the former case, this has
  applications to zeta functions enumerating conjugacy classes
  associated with so-called graphical groups.
\end{abstract}

\blankfootnote{\noindent{\itshape 2020 Mathematics Subject Classification.}
  11M41, 
  05C50, 
  15B33, 
  05A15, 
  11S80  
  20D15, 
  20E45. 

}

\RedeclareSectionCommand[tocbeforeskip=1ex plus 1pt minus 1pt]{section}

\setcounter{tocdepth}{1}
\tableofcontents

\section{Introduction}
\label{s:intro}

Ask zeta functions encode average kernel sizes within families of
linear maps over finite quotients of infinite rings.  Matrices of
linear forms lead, via specialisation of variables, to associated ask
zeta functions. Such zeta functions first arose by linearising the
enumeration of linear orbits and conjugacy classes of groups. Their
rich algebraic and combinatorial structure established them as objects
of independent interest.

Combinatorial incidence structures such as graphs and hypergraphs lead
to matrices of linear forms by ``linearising'' their adjacency and
incidence matrices. In \cite{cico}, we initiated the study of such
(hyper)graphical ask zeta functions over compact discrete valuation
rings.
In particular, for each graph $\Gamma$, we showed that there exist
\itemph{rational} generating functions
$W^+_\Gamma(X,T)$ and $W^-_\Gamma(X,T)$ that capture average kernel
sizes associated with the symmetric and  antisymmetric linearised
adjacency matrix of~$\Gamma$, respectively.

In the present paper, we develop new tools for studying the functions
$W^\pm_\Gamma$, allowing us to contribute to two contemporary research
themes in the area: rigidity of zeta functions and effects of
operations on zeta functions. Rigidity is an umbrella term for the
phenomenon that certain algebraic or combinatorial operations on
matrices of linear forms leave associated ask zeta functions
unchanged.
Here we show specifically that for reflexive graphs
$\Gamma$, three a priori quite different ask zeta functions
$W^+_\Gamma$, $W^-_\Gamma$, and $W_{\Adj(\Gamma)}$ all coincide. The
effects of natural— and seemingly innocuous—operations on zeta
functions associated with algebraic structures are generally
mysterious and poorly understand.
We show that $\Gamma\leadsto W^-_\Gamma$ is well-behaved under joins for
\itemph{all} loopless graphs (answering a question from \cite{cico}) and that
$\Gamma \leadsto W^+_\Gamma$ is well-behaved under joins for
arbitrary \itemph{reflexive} graphs: in both cases we give concise
formulae for the zeta functions of the joins in terms of the zeta
functions of the graphs being joined, vastly generalising results
from~\cite{cico}. Our results have applications to class-counting zeta
functions of so-called graphical groups.

\subsection{Ask zeta functions derived from matrices of linear forms}
\label{ss:intro/ask}

We briefly recall selected concepts and results pertaining to ask zeta
functions.  For further details and references, see \S\ref{s:ask_bg}.

\paragraph{From matrices of linear forms to ask zeta functions.}
Let $\fO$ be a compact discrete valuation ring (\DVR) with maximal ideal
$\fP$.
Examples include the ring $\ZZ_p$ of $p$-adic integers for a prime $p$ and the
ring $\FF_q\llb z \rrb$ of formal power series over a finite field $\FF_q$.
Let $A = A(X_1,\dotsc,X_\ell) \in \Mat_{d\times e}(\fO[X_1,\dotsc,X_\ell])$ be
a matrix of linear forms.
The \emph{(algebraic) ask zeta function} of $A$ over $\fO$ is the formal power
series $\Zeta^{\ask}_{A/\fO}(T) = \sum_{k=0}^\infty \alpha_k T^k \in \QQ\llb
T\rrb$ defined as follows.
Given $k\ge 0$, by specialising variables, each $x \in (\fO/\fP^k)^\ell$ gives
rise to a module homomorphism $A(x)\colon (\fO/\fP^k)^n \to (\fO/\fP^k)^m$
(acting by right multiplication on rows).
The coefficient $\alpha_k \in \QQ$ is the \underline{a}verage \underline{s}ize
of the \underline{k}ernel among the maps $A(x)$ as $x$ ranges over
$(\fO/\fP^k)^\ell$.
If $\fO$ has characteristic zero, then $\Zeta^{\ask}_{A/\fO}(T) \in \QQ(T)$
by \cite[Thm~1.4]{ask}.
Ask zeta functions were introduced in \cite{ask} as linearisations of zeta
functions enumerating linear orbits and conjugacy classes of suitable groups.

\paragraph{Global templates for ask zeta functions and arithmetic questions.}
Given a matrix of linear forms $A$ over $\ZZ$, we may consider the ask
zeta function $\Zeta^{\ask}_{A/\fO}(T)$ for each compact \DVR{} $\fO$.
It is then natural to ask how these ask zeta functions vary with
$\fO$.  Let $q$ be the size of the residue field of $\fO$.  Excluding
finitely many exceptional residue characteristics, general machinery
from $p$-adic integration \cite[Thm~4.11--4.12]{ask} provides a
``Denef formula'' for $\Zeta^{\ask}_{A/\fO}(T)$ in terms of
rational functions in $q$ and $T$ and the numbers of rational points
of certain schemes over $\ZZ$ over the residue field of $\fO$.  The
presence of numbers of rational points on schemes turns out to be
unavoidable in a precise sense, see~\cite{density}.  In general, one
can therefore expect the study of ask zeta functions
$\Zeta^{\ask}_{A/\fO}(T)$ for a fixed $\ZZ$-defined $A$ and varying
$\fO$ to involve difficult arithmetic problems.

\subsection{Matrices of linear forms from graphs and hypergraphs}
\label{ss:intro/matrices}

Two types of matrices of linear forms are of particular interest in the present
paper.

\paragraph{Matrices of linear forms from hypergraphs.}
By a \emph{hypergraph}, we mean a triple $\Eta = (V,E,\increl)$ where $V$ and
$E$ are finite sets of \emph{vertices} and \emph{hyperedges}, respectively,
and $\increl\,\, \subset V\times E$ is the \emph{incidence relation} of~$\Eta$.
Hypergraphs are also referred to as \emph{incidence structures}
in the literature.
The relation $\increl$ can be equivalently described by the \emph{support
  function} $\norm{\cdot} = \lVert\cdot\rVert_\Eta \colon E \to 2^V$ given by
$\lVert e\rVert = \{ v \in V : v \increl e\}$
as in \cite{cico}.

Let $\Eta = (V,E,\increl)$ be a hypergraph with distinct vertices
$v_1,\dotsc,v_n$ and distinct hyperedges $e_1,\dotsc,e_m$.
Up to isomorphism (suitably defined, see \S\ref{ss:hypergraphs}), $\Eta$ is
completely determined by the $n\times m$ \emph{incidence matrix} whose
$(i,j)$-entry is $1$ if $v_i \increl e_j$ and $0$ otherwise.
Let $\sA_\Eta = [X_{ij}]$ be the $n\times m$ matrix of linear forms
over $\ZZ$ such that $X_{ij} = 0$ if and only if $v_i \!\not\increl e_j$ and
such that the nonzero $X_{ij}$ are algebraically independent over $\ZZ$.
The matrix $\sA_\Eta$ is well-defined up to suitable equivalence, see
\S\ref{s:ask_bg}.
Note that $\sA_\Eta(1,\dotsc,1)$ is the incidence matrix of $\Eta$ as defined
above.

\begin{ex}
  \label{ex:looped_P4_model}
  Consider the hypergraph $\Eta$ with vertices $v_1,\dotsc,v_4$, hyperedges
  $e_1,\dotsc,e_4$, and associated incidence matrix
  \[
    \begin{bmatrix} 1 & 1 & 0 & 0 \\ 1 & 1& 1& 0 \\ 0 & 1 & 1 & 1 \\ 0 & 0 &
      1 & 1
    \end{bmatrix}.
  \]
  Hence, $\norm{e_1} = \{v_1,v_2\}$, $\norm{e_2} = \{v_1,v_2,v_3\}$,
  $\norm{e_3} = \{v_2,v_3,v_4\}$, $\norm{e_4} = \{v_3,v_4\}$, and
  \[
    \sA_\Eta = \begin{bmatrix} X_{11} & X_{12} & 0 & 0 \\ X_{21} & X_{22} &
      X_{23} & 0 \\ 0 & X_{32} & X_{33} & X_{34} \\ 0 & 0 & X_{43} & X_{44}
      \end{bmatrix}.
  \]
\end{ex}

\paragraph{Matrices of linear forms from graphs.}
Let $\Gamma$ be a graph with distinct vertices $v_1,\dotsc,v_n$ and adjacency
relation $\sim$.
In this article, unless otherwise indicated, graphs are allowed to contain
loops (i.e.\ $v_i \sim v_i$ is possible) but no parallel edges.
Let $\Gamma$ have $m$ edges.
We define matrices of linear forms $\sA_\Gamma^+$ and $\sA_\Gamma^-$ over
$\ZZ$ defined as follows.
Let $X_{ij}$ denote the $(i,j)$-entry of $\sA_\Gamma^\pm$.
We impose the following conditions:
\begin{itemize}
\item $X_{ij} = 0$ if and only if $v_i \not\sim v_j$,
\item $X_{ji} = \pm X_{ij}$ whenever $i\not= j$, and
\item The nonzero $X_{ij}$ with $i \le j$ are algebraically independent over $\ZZ$.
\end{itemize}

\begin{ex}
  \label{ex:looped_P4}
  For the graph $\Gamma$ given by
  
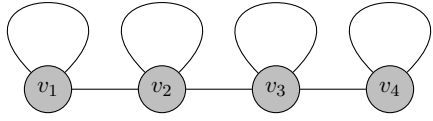
\begin{center}
    \begin{tikzpicture}
      \tikzstyle{Grey Vertex}=[fill=lightgray, draw=black, shape=circle, scale=0.8]
      \node [style=Grey Vertex] (1) at (0,0) {$v_1$};
      \node [style=Grey Vertex] (2) at (1.5,0) {$v_2$};
      \node [style=Grey Vertex] (3) at (3,0) {$v_3$};
      \node [style=Grey Vertex] (4) at (4.5,0) {$v_4$};

      \draw (1) to (2) to (3) to (4);
      \draw (1) to[out=45, in=135, looseness=10] (1);
      \draw (2) to[out=45, in=135, looseness=10] (2);
      \draw (3) to[out=45, in=135, looseness=10] (3);
      \draw (4) to[out=45, in=135, looseness=10] (4);
      \end{tikzpicture}
    \end{center}
    we obtain
  \[
    \sA^\pm_\Gamma = \begin{bmatrix} X_{11} & X_{12} & 0 & 0\\
      \pm X_{12} & X_{22} & X_{23} & 0 \\
      0 & \pm X_{23} & X_{33} & X_{34} \\
      0 & 0 & \pm X_{34} & X_{44}
    \end{bmatrix}.
   \]
\end{ex}

As in the case of hypergraphs, the matrix $\sA^\pm_\Gamma$ does not merely
depend on $\Gamma$ but also on our choice of a total order on its vertices.
Again, different choices yield equivalent matrices of linear forms
(see \S\ref{s:ask_bg}).
The matrix $\sA^+_\Gamma(1,\dotsc,1)$ is the usual adjacency matrix of
$\Gamma$ relative to the given order of the vertices.

\subsection{Background: the Uniformity Theorem}
\label{ss:intro/uniformity}

Belkale and Brosnan~\cite{BB03} showed that counting $\FF_q$-rational points
on the degeneracy loci of the matrices $\sA^+_\Gamma$ is, in a precise sense,
as hard as counting $\FF_q$-rational points on arbitrary schemes.
Average sizes of kernels are expressible in terms of the numbers of
matrices of given rank (see \cite[\S 2.1]{ask}).
This notwithstanding, the following result shows that at least for the
matrices $\sA_\Eta$ and $\sA^\pm_\Gamma$, such hard geometric problems average
out when passing to ask zeta functions.

\begin{thm}[{Uniformity Theorem \cite[Thm~A]{cico}}]
  \label{thm:uniformity}
  \quad
  \begin{enumerate}[(i)]
  \item\label{thm:uniformity1}
    Let $\Eta$ be a hypergraph.
    Then there exists $W_\Eta(X,T)\in \QQ(X,T)$ such that for each compact
    \DVR{} $\fO$ with residue field size $q$, we have
    $\Zeta^{\ask}_{\sA_\Eta/\fO}(T) = W_\Eta(q,T)$.
  \item\label{thm:uniformity2}
    Let $\Gamma$ be a graph.
    Then there exists $W^+_\Gamma(X,T)\in \QQ(X,T)$ such that for each compact
    \DVR{} $\fO$ with \underline{odd} residue field size $q$, we have 
    $\Zeta^{\ask}_{\sA^+_\Gamma/\fO}(T) = W^+_\Gamma(q,T)$.
  \item\label{thm:uniformity3} 
    Let $\Gamma$ be a graph.
    Then there exists $W^-_\Gamma(X,T)\in \QQ(X,T)$ such that for each compact
    \DVR{} $\fO$ with (arbitrary) residue field size $q$, we have 
    $\Zeta^{\ask}_{\sA^-_\Gamma/\fO}(T) = W^-_\Gamma(q,T)$.
  \end{enumerate}
\end{thm}

\begin{rem}
  \quad
  \begin{enumerate}
  \item
    In \cite{cico},
    Theorem~\ref{thm:uniformity}\ref{thm:uniformity3} was only spelled out in
    case $\Gamma$ is loopless.
    However, in \cite[\S 6]{cico}, both parts
    \ref{thm:uniformity2} and \ref{thm:uniformity3} of
    Theorem~\ref{thm:uniformity} were proved simultaneously and, in particular,
    the proof via \cite[Thm~6.4(iii)]{cico} also applies when $\Gamma$ has
    a loop.
  \item
    The first author's package \textsf{Zeta}~\cite{Zeta} for SageMath~\cite{SageMath}
    implements algorithms for computing $W^\pm_\Gamma$ and $W_\Eta$.
    These algorithms are practical for small (hyper)graphs, say on at most $7$ vertices.
  \item
    The present article provides us with a new and self-contained proof of
    Theorem~\ref{thm:uniformity}; see \S\ref{s:selectors} and
    \S\ref{ss:minors_yield_animations}.
  \end{enumerate}
\end{rem}

The present article contributes to our understanding of the operation
$\Gamma \leadsto W^-_\Gamma$.
In particular, we will see that for two very natural classes of graphs, namely
loopless and reflexive ones, taking joins has a remarkably tame algebraic
effect on $W^-_\Gamma$.

\subsection{Result: joins of loopless graphs}
\label{ss:intro/loopless_join}

Let $\Gamma_1$ and $\Gamma_2$ be graphs and
let $\Gamma_1\oplus \Gamma_2$ denote their disjoint union.
The \emph{join} $\Gamma_1 \join \Gamma_2$ of $\Gamma_1$ and $\Gamma_2$ is
obtained from $\Gamma_1\oplus \Gamma_2$ by adding edges connecting each vertex
of $\Gamma_1$ to each vertex of $\Gamma_2$.
Note that a join of loopless graphs is again loopless.

\begin{thmabc}[Loopless joins]
  \label{thm:join}
    Let $\Gamma_1$ and $\Gamma_2$ be loopless graphs on $n_1$ and $n_2$
    vertices, respectively.
    Write $z_i = X^{-n_i}$.
    Then
    {\small
    \begin{multline}
      W^-_{\Gamma_1\join \Gamma_2}(X,T)\\
      =
      \frac{z_1 z_2 XT \!-\! 1
      + W^-_{\Gamma_1}(X,\!z_2T)(1\!-\!z_2T)(1\!-\!z_2XT)
      + W^-_{\Gamma_2}(X,\!z_1T)(1\!-\!z_1T)(1\!-\!z_1XT)}
      {(1-T)(1-XT)}.
      \label{eq:join}
    \end{multline}}
\end{thmabc}

\begin{rem}[Precursors to Theorem~\ref{thm:join}]
  \quad \begin{enumerate} \item In the special case that $\Gamma_1$
  and $\Gamma_2$ are \itemph{cographs}, i.e.\ loopless graphs which do
  not contain a path on four vertices as induced subgraphs,
  Theorem~\ref{thm:join} was first proved in \cite[Prop.~8.4]{cico}.
  That proof relied on the Cograph Modelling Theorem
  (Theorem~\ref{thm:cmt} below) from \cite{cico}.
  Theorem~\ref{thm:join} thus provides a positive answer
  to \cite[Question 10.1]{cico}.

    Removing the assumption that $\Gamma_1$ and $\Gamma_2$ be cographs
    constitutes a significant extension of the scope of this result.
    For instance, it is well known that asymptotically, the number of
    (isomorphism classes of) loopless graphs on $n$ vertices grows like
    $\gamma_n := 2^{n \choose 2} / n!$; see \cite[\S 9.1]{HP73}.
    Ravelomanana and Thimonier~\cite[Thm~4]{RT01} showed that asymptotically,
    the number of (isomorphism classes of) cographs on $n$ vertices grows like
    $\beta_n := C \alpha^n / n^{3/2}$, where $C = 0.206\dotso$ and $\alpha =
    3.560\dotso$.  In particular, $\beta_n$ grows at
    most exponentially while $\gamma_n$ grows super-exponentially as
    $n\to \infty$.  \item By \cite[Prop.\ 8.5]{csp},
    equation \eqref{eq:join} holds modulo $T^2$.  Given a graph
    $\Gamma$, the coefficient of $T$ of $W_\Gamma^-(X,T)$ encodes the
    average size of the kernel of the matrices $\sA^-_\Gamma$ over
    finite fields.  \end{enumerate}
\end{rem}

\begin{rem}
  \quad
  \begin{enumerate}
  \item
    The conclusion of Theorem~\ref{thm:join} does not generally hold
    unless $\Gamma_1$ and $\Gamma_2$ are \itemph{both} loopless.
    Taking $\Gamma_1$ to be an isolated vertex and $\Gamma_2$ to be a loop at
    one vertex provides a counterexample.
    On the other hand, when $\Gamma_1$ and $\Gamma_2$ contain loops at all
    vertices, then $W^-_{\Gamma_1 \join \Gamma_2}$ is again expressible in
    terms of $W^-_{\gamma_1}$ and $W^-_{\Gamma_2}$ albeit in terms of a
    formula other than \eqref{eq:join}; see Proposition~\ref{prop:Wsharp_join}.
  \item
    Taking $\Gamma_1$ to be a singleton and $\Gamma_2$ to be a path on $2$
    vertices (so that $\Gamma_1 \join \Gamma_2$ is a triangle) shows that the
    conclusion of Theorem~\ref{thm:join} does not carry over to
    $W^+_{\Gamma_1\join \Gamma_2}(X,T)$.
    (See \cite[Table~1]{cico}.)
  \end{enumerate}
\end{rem}

\paragraph[Renormalising.]{Enter $W^\flat_\Gamma$: renormalising $W^-_\Gamma$.}

Theorem~\ref{thm:join} may be rephrased as follows.  For a graph
$\Gamma$ on $n$ vertices, let $W^\flat_\Gamma = W^\flat_\Gamma(X,T) =
W^-_\Gamma(X,X^n T)$.  Of course, $W^-_\Gamma$ and $W^\flat_\Gamma$
determine one another for given $n$.  While the coefficients of
$W^-_\Gamma(X,T)$ as a series in $T$ encode average sizes of kernels
of specialisations of $\sA^-_\Gamma(X)$ over finite quotients of
\DVR{}s, the coefficients of $W^\flat_\Gamma(X,T)$ count pairs $(v,x)$ with $v
\sA^-_\Gamma(x) = 0$ over such rings.

Let $\Gamma_1$ and $\Gamma_2$ be graphs on $n_1$ and $n_2$ vertices, respectively.
If $\Gamma_1$ and $\Gamma_2$ are both loopless, then Theorem~\ref{thm:join}
takes the following form:
\begin{multline}
  W^\flat_{\Gamma_1\join \Gamma_2}(X,T) = \\ \frac{XT - 1 +
  W^\flat_{\Gamma_1} \cdot (1-X^{n_1}T)(1-X^{n_1+1}T) +
  W^\flat_{\Gamma_2} \cdot
  (1-X^{n_2}T)(1-X^{n_2+1}T)}{(1-X^{n_1+n_2}T)(1-X^{n_1+n_2+1}T)}.  \label{eq:flat_join}
\end{multline}

\paragraph{Join powers.}

Noting that taking joins of graphs is an associative operation (up to
isomorphism), let $\Gamma^{\join k}
= \Gamma \join \dotsb \join \Gamma$, with $k$ copies of~$\Gamma$. We
let $\Gamma^{\join 0}$ denote the graph without vertices, the identity
element with respect to $\join$.  Theorem~\ref{thm:join} provides us
with infinitely many \itemph{explicit} formulae for the functions
$W^\flat_\Gamma$ (hence also $W^-_\Gamma$):

\begin{cor}
  \label{cor:join_power}
  Let $\Gamma$ be a loopless graph on $n$ vertices.
  Let $k \ge 0$.
  Then
  \[
    W^\flat_{\Gamma^{\join k}}(X,T) = \frac{ (k-1)(XT-1) + k W^\flat_\Gamma\cdot
      (1-X^nT)(1-X^{n+1}T)}
    {(1-X^{kn}T)(1-X^{kn+1}T)}.
  \]
\end{cor}
\begin{proof}
  Induction on $k$ using $\Gamma^{\join (k+1)}\approx \Gamma \join
  \Gamma^{\join k}$ and \eqref{eq:flat_join}.
\end{proof}

\begin{ex}
  Let $\Gamma$ be the following graph:
  \begin{center}
    \begin{tikzpicture}
      \tikzstyle{Grey Vertex}=[fill=lightgray, draw=black, shape=circle, scale=0.4]

      \node [style=Grey Vertex] (1) at (0,1) {};
      \node [style=Grey Vertex] (3) at (-1,0) {};
      \node [style=Grey Vertex] (4) at (1,0) {};
      \node [style=Grey Vertex] (2) at (0,-1) {};

      \draw (1) to (2);
      \draw (1) to (3);
      \draw (1) to (4);
      \draw (2) to (3);
      \draw (2) to (4);
      \end{tikzpicture}
    \end{center}
    By \cite[Thm~8.18]{cico} or \cite[Table~1]{cico}, we have
    $W^-_\Gamma = \frac{(1-X^{-1}T)(1-X^{-2}T)}{(1-T)^2(1-XT)}$
    and thus
    $W^\flat_\Gamma = W^-_\Gamma(X,X^4T) = \frac{(1-X^2T)(1-X^3
      T)}{(1-X^4T)^2(1-X^5T)}$.
    Using Corollary~\ref{cor:join_power}, we find that
    \[
      W^\flat_{\Gamma^{\join k}}(X,T) =
      \frac
      {
        X^5 T^2
        + (k-1) (X^4T + XT)
        - k (X^3T + X^2T)
        + 1}
      {(1-X^4T)(1-X^{4k}T)(1-X^{4k+1}T)}
    \]
    and thus
    \begin{align*}
      W^-_{\Gamma^{\join k}}(X,T)
      & = W^\flat_\Gamma(X,X^{-4k}T) \\
      & =
            \frac
      {
        X^{5-8k} T^2
        + (k-1) (X^{4-4k}T + X^{1-4k}T)
        - k (X^{3-4k}T + X^{2-4k}T)
        + 1}
      {(1-X^{4-4k}T)(1-T)(1-XT)}.
    \end{align*}
\end{ex}

\subsection{Related work: disjoint unions and Hadamard products}
\label{ss:hadamard}

Theorem~\ref{thm:join} shows that one natural graph-theoretic operation,
taking joins, has a transparent effect on the rational functions $W^-_\Gamma$
attached to loopless graphs.
There are of course many other natural binary operations on (loopless) graphs
that one may consider.
To our knowledge, only one of these has been investigated so far: disjoint
unions.

Recall that given formal power series $F(T) = \sum_{k=0}^\infty a_k T^k$ and
$G(T) = \sum_{k=0}^\infty b_k T^k$ with coefficients in a field $k$, their
\emph{Hadamard product} is $F(T) *_T G(T) = \sum_{k=0}^\infty a_k b_kT^k$.
It is well known that if $F(T), G(T)\in k(T)$, then $F(T)*_T G(T) \in
k(T)$; see e.g.\ \cite[Ch.\ 1]{BR88}.
It is easy to see that for each graph $\Gamma$,
we have $W^\pm_{\Gamma_1\oplus\Gamma_2}(X,T) =
W^\pm_{\Gamma_1}(X,T) *_T W^\pm_{\Gamma_2}(X,T)$; cf.\ \cite[\S 8.2]{cico}.
Here, the Hadamard product is taken over $k = \QQ(X)$.
We note that
$W^\flat_{\Gamma_1 \oplus \Gamma_2} = W^\flat_{\Gamma_1} *_T W^\flat_{\Gamma_2}$
follows easily; cf.\ \cite[Lem.\ 5.9]{CMR24a}.

In general, explicitly expressing Hadamard products of rational generating
functions as (explicit sums of) rational functions appears to be difficult.
A growing body of research has provided algebro-combinatorial tools
for studying Hadamard products of ask zeta functions (in particular those
associated with graphs), orbit-counting, and class-counting zeta functions in
fortunate cases, see \cite[\S 2.3]{ask}, \cite[\S 5.2]{cico}, and
\cite{CMR24b} (see also \cite{CMR24a}).
We record some consequences of these results in \S\ref{s:Wsharp}.

\subsection{Result: the Reflexive Graph Modelling Theorem}
\label{ss:intro/RGMT}

Let $\Gamma$ be a graph with vertex set $V$ and adjacency relation $\sim$ on $V$.
Anticipating a definition from \S\ref{ss:graphs}, we write
$\Adj(\Gamma)$ for the \emph{adjacency hypergraph} $(V,V,\sim)$ of $\Gamma$.
Hence, $\Adj(\Gamma)$ is the hypergraph obtained by viewing an adjacency
matrix of $\Gamma$ as incidence matrix of a hypergraph.
Recall that $\Gamma$ is \emph{reflexive} if $v \sim v$ for each $v\in V$.
The following is our second main result.

\begin{thmabc}[Reflexive Graph Modelling Theorem]
  \label{thm:rgmt}
  Let $\Gamma$ be a reflexive graph.
  Then
  \[
    W^+_\Gamma = W^-_\Gamma = W_{\Adj(\Gamma)}.
  \]
\end{thmabc}

\begin{ex}
  Let $\Eta$ be as in Example~\ref{ex:looped_P4_model} and
  $\Gamma$ be as in Example~\ref{ex:looped_P4}.
  Then Theorem~\ref{thm:rgmt} shows that
  $W^\pm_\Gamma = W_{\Eta}$.
  Using \textsf{Zeta}~\cite{Zeta}, we find this common rational function
  to be
  \[
    \frac{
      1 + 2 X^{-1}T - X^{-2}T^2 - 6X^{-2}T + 6 X^{-4}T^2 + X^{-4}T - 2X^{-5}T^2 - X^{-6}T^3}
    {(1 -X^{-1}T)^2(1 - T)^2}.
  \]
\end{ex}

We offer three perspectives on Theorem~\ref{thm:rgmt}. The first
portrays it as a new ``modelling theorem'', the second as a new
rigidity result, and the third---developed in
Section~\ref{ss:intro/Wsharp}---views it in the context of reflexive
joins.

\paragraph{A new modelling theorem.}
First, Theorem~\ref{thm:rgmt} is a new ``modelling theorem'' that
belongs to the same genre as the Cograph Modelling
Theorem \cite[Thm~D]{cico} (Theorem~\ref{thm:cmt} below).
Indeed, it asserts that for each graph $\Gamma$ of a certain type (namely,
each reflexive graph), there exists a modelling hypergraph $\Eta =
\Adj(\Gamma)$ such that $W^\pm_\Gamma = W_\Eta$.
Such a result allows us to leverage what is known about the
rational functions $W_\Eta$; see \S\ref{ss:intro/master} below.

\paragraph{Rigidity of ask zeta functions.}
Second, we may view Theorem~\ref{thm:rgmt} as a new instance of
rigidity phenomena that have been previously explored in the study of
ask zeta functions.  The first such result
is \cite[Cor. 5.10]{ask}.  It asserts that for any $d > 1$ and each
compact \DVR{} $\fO$ with residue field size $q$, the ask zeta
functions associated with the generic $d \times d$ matrix $A_d$ and
the generic traceless $d\times d$ matrix $T_d$ over $\fO$ coincide;
this common ask zeta function is~$(1-q^{-d}T)/(1-T)^2$.  That is,
imposing the linear relation $\trace(A_d) = 0$ on the entries of $A_d$
has no effect on associated ask zeta functions.  This result was
significantly extended in \cite[Thm~A]{board}, which showed that
imposing suitably \itemph{admissible} systems of linear equations
involving the entries of generic rectangular, symmetric, or
antisymmetric matrices has no effects on ask zeta functions.

Theorem~\ref{thm:rgmt} is a new result of this form.
Indeed, given a reflexive graph $\Gamma$, the matrix $\sA^\pm_\Gamma =
[a_{ij}]$ is obtained from $\sA_{\Adj(\Gamma)}$ by imposing the linear
relations $a_{ij} = \pm a_{ji}$ for $i \not= j$.
Crucially, these relations are \itemph{not} among those covered by \cite[Thm~A]{board}.

Viewing $\sA_{\Gamma}^\pm$ as being obtained from $\sA_{\Adj(\Gamma)}$
by imposing off-diagonal (anti)symmetry relations, it is natural ask
whether the same conclusion holds for more general classes of matrices
of linear forms.
Indeed, even slight generalisations of the matrices $\sA_{\Adj(\Gamma)}$
yield examples for which natural analogues of the  conclusions of
Theorem~\ref{thm:rgmt} no longer hold.

\begin{ex}
  \label{ex:ABC}
  Consider the following matrices of linear forms:
  \[
    A = \begin{bmatrix} U_1 & X & 0 \\ X & U_2 & X \\ 0 & X & U_3 \end{bmatrix}, \quad
    B = \begin{bmatrix} U_1 & X & 0 \\ Y & U_2 & X \\ 0 & Y & U_3 \end{bmatrix}, \quad 
    C = \begin{bmatrix} U_1 & X & 0 \\ Y & U_2 & X \\ 0 & Z & U_3
    \end{bmatrix}.
  \]

  Using \textsf{Zeta}~\cite{Zeta}, we find that for almost all residue
  characteristics of $\fO$, the zeta functions $\Zeta^{\ask}_{A/\fO}$,
  $\Zeta^{\ask}_{B/\fO}$, and $\Zeta^{\ask}_{C/\fO}$ are all distinct.
  They also all differ from $\Zeta^{\ask}_{\sA_\Gamma^\pm/\fO}$,
  where~$\Gamma$ is the ``looped path''
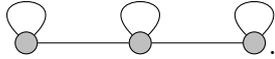
  \begin{center}
    \begin{tikzpicture}
      \tikzstyle{Grey Vertex}=[fill=lightgray, draw=black, shape=circle, scale=0.8]
      \node [style=Grey Vertex] (1) at (0,0) {};
      \node [style=Grey Vertex] (2) at (1.5,0) {};
      \node [style=Grey Vertex] (3) at (3,0) {};

      \draw (1) to (2) to (3);
      \draw (1) to[out=45, in=135, looseness=10] (1);
      \draw (2) to[out=45, in=135, looseness=10] (2);
      \draw (3) to[out=45, in=135, looseness=10] (3);
    \end{tikzpicture}
    \hspace*{-2em} .
    \end{center}
    Using the notation from the next subsection, a formula for the latter zeta
    function is recorded in Table~\ref{tab:graphs4} in the row corresponding
    to the (simple) path on three vertices.
\end{ex}

\subsection[A new graph invariant and reflexive joins]{A new graph invariant
  ($W^\sharp_\Gamma$) and reflexive joins}
\label{ss:intro/Wsharp}

Given a graph $\Gamma$, let $\hat\Gamma$ denote its reflexive closure,
obtained by adding all missing loops.
For our third perspective on Theorem~\ref{thm:rgmt}, observe that
$\Gamma \mapsto \hat\Gamma$ yields a bijection between loopless and reflexive graphs.
Theorem~\ref{thm:rgmt} therefore suggests the study of the following rational functions
attached to loopless graphs.

\begin{defn}
  \label{d:Wsharp}
  For a graph $\Gamma$,
  let $W^\sharp_\Gamma(X,T)$ denote the common value of
  \[
    W^+_{\hat\Gamma}(X,T) = W^-_{\hat\Gamma}(X,T) = W_{\Adj(\hat\Gamma)}(X,T).
  \]
\end{defn}

For a list of all $W^\sharp_\Gamma$ as $\Gamma$ ranges over graphs on at most
four vertices, see Table~\ref{tab:graphs4}.
The~$W^\sharp_\Gamma$ appear to be remarkably well-behaved.
In particular, using Theorem~\ref{thm:rgmt} and results from \cite{cico}, we
obtain the following reflexive counterpart of Theorem~\ref{thm:join}.

\begin{propabc}[Reflexive joins and disjoint unions]
  \label{prop:Wsharp_join}
  Let $\Gamma_1$ and $\Gamma_2$ be (loopless) graphs on $n_1$ and $n_2$
  vertices, respectively.
  Write $z_i = X^{-n_i}$.
  Then $W^\sharp_{\Gamma_1\oplus \Gamma_2} = W^\sharp_{\Gamma_1} *_T
  W^\sharp_{\Gamma_2}$
  and
  {\small
    \[
    W^\sharp_{\Gamma_1 \join \Gamma_2}(X,T) =
    \frac{z_1 z_2 T  - 1
      + W^\sharp_{\Gamma_1}\!(X,z_2T)(1-z_2 T)^2
      + W^\sharp_{\Gamma_2}\!(X,z_1T)(1-z_1 T)^2
    }
    {(1-T)^2}.
  \]}
\end{propabc}

\subsection{Group-theoretic context: graphical groups and their conjugacy classes}
\label{ss:intro/cc_and_baer}

As shown in \cite{ask,ask2}, ask zeta functions arise naturally in the
enumeration of linear orbits and conjugacy classes of unipotent
groups.  We recall the connection between ask and class-counting
zeta functions in the special case of Baer group schemes.

\paragraph{Class-counting zeta functions of Baer group schemes.}
Let $\concnt(H)$ denote the number of conjugacy classes of a group $H$.
Let $\GG$ be a group scheme of finite type over $\fO$.
Inspired by~\cite{dS05}, the \emph{class-counting zeta function} of $\GG$ is
$\Zeta^{\cc}_{\GG}(T) = \sum_{k=0}^\infty \concnt(\GG(\fO/\fP^k))T^k$.

Let $A = A(X_1,\dotsc,X_\ell) \in \Mat_d(\ZZ[X_1,\dotsc,X_\ell])$ be an
antisymmetric matrix of linear forms.
We identify linear forms in $\ZZ[X_1,\dotsc,X_\ell]$ and elements of
$\ZZ^\ell$ with $X_i$ corresponding to the $i$th standard basis vector of
$\ZZ^\ell$.
With this identification, $A$ is equivalently described by the alternating
bilinear product 
$\diamond\colon \ZZ^d\times \ZZ^d \to \ZZ^\ell$ given by $x\diamond y = x A y^\top$
($x,y\in \ZZ^d$).
Using a geometric variant from \cite[\S 2.4]{SV14} of the classical Baer
correspondence, the \emph{Baer group scheme}~$\GG_\diamond$ attached to
$\diamond$, and hence to $A$, was defined in \cite[\S 2.4]{cico}.
The following was proved in \cite{cico} in the special case that the map
$x\mapsto A(x)$ on $\ZZ^\ell$ is injective---the same arguments apply without
this assumption.

\begin{prop}[{Cf.\ \cite[Prop.\ 1.1]{cico}}]
  \label{prop:baer_cc_via_ask}
  Let $A\in\Mat_d(\ZZ[X_1,\dotsc,X_\ell])$ be an antisymmetric matrix of
  linear forms.
  Let $\diamond\colon \ZZ^d\times \ZZ^d\to \ZZ^\ell$ be the alternating
  bilinear product attached to $A$ as above.
  Let $\fO$ be a compact \DVR{} of arbitrary characteristic.
  Let $q$ be the size of the residue field of $\fO$.
  Then $\Zeta^{\cc}_{\GG_\diamond \otimes \fO}(T) =
  \Zeta^{\ask}_{A/\fO}(q^\ell T)$.
\end{prop}

\paragraph{Graphical groups and their class-counting zeta functions.}
Let $\Gamma$ be a loopless graph.
The \emph{graphical group scheme} $\GG_\Gamma$ (over $\ZZ$) associated
with~$\Gamma$ was defined in \cite[\S 3.4]{cico}.
For a group-theoretic description, see \cite[\S 1.1]{csp}.
Using our notation from \S\S\ref{ss:intro/ask}--\ref{ss:intro/matrices},
$\GG_\Gamma$ is the Baer group scheme associated with the alternating bilinear
product attached to the antisymmetric matrix of linear forms $\sA^-_\Gamma$.
(We require $\Gamma$ to be loopless for $\sA^-_\Gamma$ to be antisymmetric.)
The group $\GG_\Gamma(\ZZ)$ is the maximal nilpotent quotient of class at most
two of the right-angled Artin group associated with the (loopless) complement
of $\Gamma$.
For each odd prime $p$, we have $\GG_\Gamma(\FF_p) \approx
\GG_\Gamma(\ZZ)/\GG_\Gamma(\ZZ)^p$.
The following consequence of Proposition~\ref{prop:baer_cc_via_ask} provides a
group-theoretic motivation for studying the functions $W^-_\Gamma$.

\begin{cor}[{Cf.\ \cite[Prop.\ 3.9]{cico}}]
  Let $\Gamma$ be a loopless graph with $m$ edges.
  Let $\fO$ be a compact \DVR{} with residue field of size $q$.
  Then $\Zeta^{\cc}_{\GG_\Gamma \otimes \fO}(T) =
  W^-_\Gamma(q,q^mT)$.
  \qed
\end{cor}

In this way, our results from \S\ref{ss:intro/loopless_join} have immediate
group-theoretic consequences in that they provide us with formulae for
class-counting zeta functions of graphical groups.

\subsection{Related work: from hypergraphs to cographs and back again}
\label{ss:intro/master}

Theorem~\ref{thm:uniformity}\ref{thm:uniformity2}--\ref{thm:uniformity3}
notwithstanding, it appears to be very difficult to produce explicit examples 
of the rational functions $W^\pm_\Gamma(X,T)$ for interesting families of graphs.
On the other hand, we do have a very precise formula for $W_\Eta(X,T)$.
As in \cite{cico}, for a finite set $V$, let $\WOhat(V)$ denote the poset of
(possibly empty) flags of (possibly empty) subsets of $V$.

\begin{thm}[{Cf.\ \cite[Thm~C]{cico}}]
  \label{thm:hypergraph_master}
  Let $\Eta = (V,E,\increl)$ be a hypergraph.
  For $U\subset V$, define $\pth{U} = \{ e\in E : \exists u \in U, u \increl
  e\}$.
  Then
  \begin{equation}\label{eq:hypergraph_master}
    W_{\Eta}(X,T) = \sum_{y \in \WOhat(V)}(1-X^{-1})^{\card{\sup(y)}}
    \prod_{U\in y} \frac{X^{\card{U}-\card{\pth{U}}}T}{1-X^{\card{U}-\card{\pth{U}}}T}.
  \end{equation}
\end{thm}
\begin{proof}
  For $I \subset V$, define
  $\mu_I = \#\bigl\{ e \in E : I = \{ v\in V : v\increl e\}\bigr\}$.
  By \cite[Thm~C]{cico},
  \[
    W_{\Eta}(X,T) = \sum_{y \in \WOhat(V)}(1-X^{-1})^{\card{\sup(y)}}
    \prod_{U\in y} \frac{X^{\card{U}-\sum_{I\cap U \neq
          \emptyset}\mu_I}T}{1-X^{\card{U}-\sum_{I\cap U \neq
          \emptyset}\mu_I}T}.
  \]
  The claim follows since for each $U\subset V$, we have
  $\sum\limits_{I \cap U \not= \emptyset} \mu_I = \card{\pth{U}}$.
\end{proof}

The number of summands in \eqref{eq:hypergraph_master} grows
super-exponentially with $\card V$; see \cite[(1.4)]{cico}.  While
Theorem~\ref{thm:hypergraph_master} is therefore of limited use when
it comes to explicitly computing $W_\Eta(X,T)$, it turns out to be a
powerful theoretical tool; see \cite[Thms~E--F]{cico}.  The following
result from \cite{cico} constitutes a bridge between incidence
structures in hypergraphs and adjacency structures in graphs.

\begin{thm}[{Cograph Modelling Theorem \cite[Thm~D]{cico}}]
  \label{thm:cmt}
  Let $\Gamma$ be a cograph.
  Then there exists an explicit hypergraph $\Eta$ such that $W^-_\Gamma(X,T) =
  W_\Eta(X,T)$.
\end{thm}

Hence, if $\Gamma$ is a cograph, then
Theorem~\ref{thm:hypergraph_master} and its many consequences
in \cite{cico} apply to $W^-_\Gamma(X,T)$.  Following \cite{cico}, we
refer to $\Eta$ as in Theorem~\ref{thm:cmt} as a
\emph{model} or a \emph{modelling hypergraph} of $\Gamma$.
Theorem~\ref{thm:rgmt} shows that if $\Gamma$ is a reflexive graph,
then $W^-_\Gamma(X,T) = W_{\Eta}(X,T)$ for $\Eta = \Adj(\Gamma)$.  In
contrast, the explicit modelling hypergraph in the proof of
Theorem~\ref{thm:cmt} in \cite{cico} is constructed recursively in terms of
decompositions of a cograph $\Gamma$ into joins and disjoint unions of
subgraphs.  We note that by combining Theorem~\ref{thm:join} and
results from \cite[\S 5]{cico}, we obtain a new simple new proof of
Theorem~\ref{thm:cmt} in \S\ref{ss:proof_cmt}.

\subsection{Methodology}
\label{ss:intro/methodology}

All main results in the present paper rely on a new proof of
Theorem~\ref{thm:uniformity}.

\paragraph{The Uniformity Theorem: behind the scenes.}
Given a matrix of linear forms $A$ over a compact \DVR{} $\fO$ with maximal
ideal $\fP$,
\cite[\S 4]{ask} provides formulae for $\Zeta^{\ask}_{A/\fO}(T)$
in terms of $\fP$-adic integrals.
Using ideas from \cite{Vol10}, these integrals are expressible in terms of the
ideals of minors of $A$ itself or of one of its \itemph{Knuth duals} in the
sense of \cite{ask2}.
In the cases of the matrices $\sA_\Eta$ (resp.\ $\sA^\pm_\Gamma$) from
\S\ref{ss:intro/matrices}, we record descriptions of such $\fP$-adic
integrals representing the associated ask zeta functions
in Proposition~\ref{prop:ask_Eta_integral} (resp.\
Proposition~\ref{prop:ask_Gamma_integral}) below. 
For the integral notation that we use, see
\eqref{eq:Eta_integral}--\eqref{eq:Gamma_integral}.
In these integrals, $\Minors_i \Eta$ (resp.\ $\Minors_i \Gamma^\pm$) denotes the
ideal generated by the $i\times i$ minors of a certain matrix of linear forms
$\sC_\Eta$ (resp.\ $\sC_\Gamma^\pm$).
In the language of \cite{ask2}, $\sC_\Eta$ (resp.\ $\sC_\Gamma^\pm$) is the
$\circ$-dual of $\sA_\Eta$ (resp.\ $\sA_\Gamma^\pm$).

The proof of Theorem~\ref{thm:uniformity} (parts
\ref{thm:uniformity2}--\ref{thm:uniformity3}, in particular)
in \cite{cico} was based on an analysis of integrals as
in \eqref{eq:Eta_integral}--\eqref{eq:Gamma_integral}
using toric geometry and an elaborate recursion.
Working directly with ideals of minors of matrices of linear forms can quickly
become daunting and the recursive approach from \cite{cico} allowed us to
completely avoid investigating any minors.
In essence, for a given graph $\Gamma$, the proof of
Theorem~\ref{thm:uniformity}\ref{thm:uniformity2}--\ref{thm:uniformity3}
expressed each integral \eqref{eq:Gamma_integral} as an unspecified finite
sum of \itemph{monomial} $\fP$-adic integrals.
For each of the latter integrals, conclusions analogous to
Theorem~\ref{thm:uniformity} are well known to hold;
see Proposition~\ref{prop:general_monomial_integral} below.
The proof of Theorem~\ref{thm:cmt} relied on very similar core ingredients.

\paragraph{Key new tool: an explicit combinatorial parameterisation of
  minors}

The technical innovation of the present article is an
explicit combinatorial parameterisation of the nonzero minors in the integrals
\eqref{eq:Eta_integral}--\eqref{eq:Gamma_integral}.
Our parameterisation will be obtained in Proposition~\ref{prop:selectors} for
hypergraphs and in Theorem~\ref{thm:animations} for graphs, the latter case
being much more involved.
Our parameterisation shows that (assuming invertibility of~$2$
in the study of $\sA^+_\Gamma$) each of the ideals of minors in
\eqref{eq:Eta_integral}--\eqref{eq:Gamma_integral} is generated by monomials.
As a first application of our parameterisation, we obtain a new proof of
Theorem~\ref{thm:uniformity} using the well-known uniform rationality of
monomial $\fP$-adic integrals (Proposition~\ref{prop:general_monomial_integral}). 

Crucially, we prove significantly more than monomiality of the aforementioned
ideals of minors.
Namely, we show that up to signs (and multiplication by powers of $2$ in the case
of $\sA^+_\Gamma$), the nonzero minors in question
are explicit monomials derived from combinatorial
gadgets attached to hypergraphs and graphs that we call \itemph{selectors} and
\itemph{animations}, respectively.
Here, a \emph{selector} of a hypergraph $\Eta = (V,E,\increl)$ is a partial
function $\phi$ defined on some subset of $E$ such that, whenever it is defined, $\phi$
sends a hyperedge $e$ to one of its incident vertices.
Similarly, an \emph{animation} of a graph is a partial function on the vertex
set which, whenever defined, sends a vertex to one of its neighbours. For formal definitions, see \S\ref{s:selectors} and
\S\ref{ss:animations}.

By studying algebraic and combinatorial features of animations, we
obtain our two main results, Theorem~\ref{thm:join} and
Theorem~\ref{thm:rgmt}.  When studying the rational functions
$W^-_\Gamma(X,T)$ attached to loopless graph, our parameterisation of
minors leads us to consider \itemph{nilpotent} animations, within an
ambient monoid of partial functions.  These have a rich algebraic and
combinatorial structure which forms the basis of our proof of
Theorem~\ref{thm:join}.

\subsection{Overview}

In \S\ref{s:ask_bg}, we collect basic material on ask zeta functions.
In \S\ref{s:(hyper)graphs}, we review facts on ask zeta functions
attached to graphs and hypergraphs.
Ask zeta functions associated with hypergraphs are the subject of
\S\ref{s:selectors}, culminating in a new proof of
Theorem~\ref{thm:uniformity}\ref{thm:uniformity1} by means of selectors.
As indicated above, the case of graphs is considerably more complicated.
In \S\ref{s:towards_minors}, we lay the foundation for our analysis of the
rational functions $W^\pm_\Gamma$ by means of animations in
\S\ref{s:animations}.
As a by-product, we obtain a new proof of
Theorem~\ref{thm:uniformity}\ref{thm:uniformity2}--\ref{thm:uniformity3}.
Drawing upon the machinery that we developed, the Reflexive Graph
Modelling Theorem (Theorem~\ref{thm:rgmt}) follows quite easily
in \S\ref{s:RGMT}.
The next sections lay the groundwork for the proof of Theorem~\ref{thm:join}
in \S\ref{s:plumbing}.
Section~\ref{s:nilpotent} is devoted entirely to nilpotent animations.
These play a crucial role in our study of $W^-_\Gamma$ for loopless $\Gamma$.
In \S\ref{s:plumbing}, we relate the nilpotent animations of a join $\Gamma_1
\join \Gamma_2$ of two loopless graphs to those of the $\Gamma_i$.
At first glance, \S\ref{s:generic_rows} might be mistaken for a non sequitur:
in it, we investigate the effect of adding generic rows to matrices of linear
forms on associated ask zeta functions.
This investigation will play a small but pivotal role in our proof of
Theorem~\ref{thm:join} in \S\ref{s:proof_join}.
Finally, in \S\ref{s:Wsharp}, we have a closer look at the rational functions
$W^\sharp_\Gamma$ from Definition~\ref{d:Wsharp}.
In particular, we prove Proposition~\ref{prop:Wsharp_join}, we derive
an explicit formula in the spirit of Theorem~\ref{thm:hypergraph_master} for
$W^\sharp_\Gamma$ (Proposition~\ref{prop:Wsharp_formula}), we deduce key
analytic properties (Proposition~\ref{prop:Wsharp_poles}, and we collect
several examples of these rational functions.

\subsection{Notation}
\label{ss:notation}

\paragraph{Sets, functions, and logic.}
Maps usually act on the right.
We write $A \sqcup B$ for the disjoint union of the sets $A$ and $B$.
We write $A \subset B$ to indicate that $A$ is a not necessarily proper subset
of $B$.
For a property $P$, we write $\Iverson P$ for the \emph{Iverson bracket}
\[
  \Iverson{P} = \begin{cases} 1, & \text{if $P$ is true}, \\
    0, &\text{otherwise}.
    \end{cases}
\]

\paragraph{Rings and modules.}
All rings are assumed to be associative, commutative, and unital.
Let $R$ be a ring.
By an \emph{$R$-algebra}, we mean a ring $S$ endowed with a ring map $R\to
S$. Let $U$ be a set.
Let $RU$ denote the free $R$-module on $U$ with basis $(\std u)_{u\in U}$.
For $x\in R U$, we define $x_u \in R$ ($u\in U$) via
$x = \sum_{u\in U} x_u \std u$.
For a subset $A\subset R$ (not necessarily a subring or an ideal), we
occasionally write $A R = \{ x\in R U : x_u \in A \text{ for all } u\in U\}$.

We write $X_U = (X_u)_{u\in U}$ for a chosen set of algebraically independent
variables over~$R$.
Each $a \in \ZZ U$ gives rise to the Laurent monomial
$X_U^a = \prod_{u\in U} X_u^{a_u}$.
By a \emph{monomial ideal} $I$ of $R[X_U]$, we mean an ideal of the form
$I = \langle X_U^a : a \in A\rangle$ for some $A\subseteq \NN_0 U$.
It is well known that if $I$ is a monomial ideal as before, then the set $A$
can be chosen to be finite.
By a \emph{linear form} in $R[X_U]$, we mean a polynomial of the form
$\sum_{u\in U} c_u X_u$ ($c_u \in U$).
If $S$ is an $R$-algebra, then
by sending $x \in SU$ to the map
which evaluates polynomials in $X_U$ at $x$, we obtain a canonical bijection
between $SU$ and the set of $R$-algebra homomorphisms $R[X_U] \to S$.

\paragraph{Discrete valuation rings.}
Throughout, $\fO$ denotes a compact \DVR{} with
maximal ideal~$\fP$ and residue field $\fO/\fP$ of size $q$ and characteristic
$p$.
For a nonzero finitely generated $\fO$-module $M$, we write $M^\times = M
\setminus \fP M$;
we also set $\{0\}^\times = \{0\}$.
Let $\pi\in\fP\setminus \fP^2$ denote a fixed uniformiser.
Let $K$ be the field of fractions of $\fO$.
Let $\abs\cdot$ be the absolute value on $K$ with $\abs\pi = q^{-1}$ and
let $\norm\cdot$ denote the associated maximum norm.
We write $\mu$ for the Haar measure on $\fO$ with total volume $1$.
We use the same symbol for the product measure on a free $\fO$-module of
finite rank.
We let $\ord = \ord_K$ denote the normalised (additive) valuation on $K$ with
$\ord(\pi) = 1$.
For a finite set $U$ and $x\in K U$ with $\prod_{u\in U}x_u \not= 0$, we write
$\ord(x) = \sum_{u\in U} \ord(x_u) \std u\in \ZZ U$.

\paragraph{Further notation.}
\small
\vspace*{-.5em}
\begin{longtable}{r|l|c}
  Notation\phantom{$1_1$} & comment & reference \\
  \hline
  $\Minors_k(M)$ & ideal generated by $k\times k$ minors of $M$ &
                                                                     \S\ref{s:ask_bg}
  \\
  $\Zeta^{\ask}_{A/\fO}$, $\zeta^{\ask}_{A/\fO}$ & ask zeta functions &
  \S\ref{s:ask_bg}
  \\
  $\Eta = (V,E,\increl)$  & hypergraph & \S\ref{ss:hypergraphs} \\
  $\norm{e}, \norm{e}_\Eta$ & support of the hyperedge $e$ &
                                                             \S\ref{ss:intro/matrices} \\
  $\Eta[V'\mid E']$ & subhypergraph & \S\ref{ss:hypergraphs} \\
  $\oplus$, $\join$ & disjoint union, join / complete union
                                    & \S\ref{ss:hypergraphs},
                                      \S\ref{ss:graphs}, \S\ref{ss:proof_cmt}\\
  $\sA_\Eta, \sC_\Eta$ & linearised incidence matrix, its $\circ$-dual & \S\ref{ss:ask_hypergraph} \\
  $\Minors_k \Eta$ & $\Minors_k(\sC_\Eta)$ & \S\ref{ss:ask_hypergraph} \\
  $\int_W \Eta(s)$ & integral expression for $\zeta^{\ask}_{\sA_\Eta/\fO}(s)$
                                    & \eqref{eq:Eta_integral}, Proposition~\ref{prop:ask_Eta_integral} \\
  $\sA^\pm_\Gamma, \sC^\pm_\Gamma$ & linearised adjacency matrix, its
                                     $\circ$-dual & \S\ref{ss:ask_graph} \\
  $\Minors_k^\pm \Gamma$, $\frac 1 2 \Minors^+_\Gamma$ &
                                                         $\Minors_k(\sC^\pm_\Gamma)$
                                                         over $\ZZ$ or $\ZZ[1/2]$
                                                         &\S\ref{ss:ask_graph}
  \\
  $\int_W \Gamma^\pm(s)$ & integral expression for $\zeta^{\ask}_{\sA^\pm_\Gamma/\fO}(s)$
                                    & \eqref{eq:Gamma_integral}, Proposition~\ref{prop:ask_Gamma_integral} \\
  $U_\udef$ & $U \sqcup \{\udef\}$ & \S\ref{s:selectors} \\
  $\Dom(\phi)$ & domain of definition of $\phi$ & \S\ref{s:selectors} \\
  $Y^{\phi^*}$ & preimage of $Y$ under $\phi$ & \S\ref{s:selectors} \\
  $\deg(\phi)$ & $\card{\Dom(\phi)}$ & \S\ref{s:selectors} \\
  $\mon\phi$ & monomial associated with $\phi$ & \S\ref{s:selectors} \\
  $\phi\restriction U'$ & restriction of $\phi$ to $U'$ & \S\ref{s:selectors} \\
  $\phi[x\gets y]$ & redefining $\phi$ & \S\ref{s:selectors} \\
  $\sim$ & adjacency in a graph & \S\ref{ss:graphs}\\
  $\hat\Gamma$ & reflexive closure & \S\ref{ss:intro/Wsharp} \\
  $W^\pm_\Gamma$, $W_\Eta$ & ask zeta function associated with (hyper)graph
                                    & Theorem~\ref{thm:uniformity} \\
  $W^\sharp_\Gamma$ & common value of $W^+_{\hat \Gamma} = W^-_{\hat \Gamma} =
                      W_{\Adj(\hat\Gamma)}$ & Definition~\ref{d:Wsharp}\\
  $\Selectors(\Eta)$ & selectors & \S\ref{s:selectors} \\
  $\Adj(\Gamma)$ & adjacency hypergraph & \S\ref{ss:intro/RGMT},
                                          \S\ref{ss:graphs} \\
  $\Inc(\Gamma)$ & incidence hypergraph & \S\ref{ss:graphs} \\
  $m^\pm_\Gamma[V'\mid E']$ & minor of $\sC^\pm_\Gamma$ &
                                                          \S\ref{ss:hypergraph_decompositions}, \S\ref{ss:animations_yield_minors} \\
  $\odlen(\phi)$ & number of $\phi$-orbits of odd length $ > 1$ &
                                                                  \S\ref{ss:nilpotency} \\
  $\Animations(\Gamma)$ & animations & \S\ref{ss:animations} \\
  $\Nilpotents(\Gamma)$, $\Fixed(\Gamma)$, $\Oddperiods(\Gamma)$ & special sets of animations &
                                                                         \S\ref{ss:animations} \\
  $\lte_\alpha$ & preorder derived from a partial function &
                                                             \S\ref{ss:animations_order_vertices} \\
  $\last_\alpha(v)$ & unique $\lte_\alpha$-maximal element above $v$ &
                                                                       \S\ref{ss:animations_order_vertices} \\
  $\lte_u$ & partial order relative to distinguished vertex & \S\ref{ss:ordering_monomials}\\
  $\simeq$ & equivalence of matrices & \S\ref{ss:matrix_shenanigans} \\
  $\lambda(A)$, $\lambda_i(A)$ & elementary divisors & \S\ref{ss:edt}
\end{longtable}

\normalsize

\section{Background on ask zeta functions}
\label{s:ask_bg}

The following is a brief introduction to ask zeta function attached to
matrices of linear forms.
In terms of generality, this perspective lies between \cite{ask}, which
considers modules of matrices, and \cite{ask2}, which considers so-called
module representations.
Let $R$ be a ring.
Let $U$ be a finite set.
Recall that $RU$ denotes the free $R$-module with basis $(\std u)_{u\in U}$.

\paragraph{Equivalence.}
There is a natural action of $\GL(R U)\times \GL_n(R) \times \GL_m(R)$ on the
$R$-module of linear forms within $\Mat_{n\times m}(R[X_U])$: the second and
third factor act by matrix multiplication on the left and right, respectively,
and $\GL(R U)$ acts by changing coordinates of linear forms.
Two matrices of linear forms $A(X_U)$ and $B(X_U)$ in
$\Mat_{n \times m}(R[X_U])$ are \emph{equivalent} if they lie in
the same orbit under this action.
Equivalence in this sense corresponds to \itemph{isotopy} of module
representations in \cite{ask2}.

\paragraph{Ask zeta functions.}
Let $A(X_U)\in \Mat_{n\times m}(R[X_U])$ be a matrix of linear forms.
Let $S$ be an $R$-algebra.
Given $x\in S U$, we view the specialisation $A(x) \in \Mat_{n\times m}(S)$
as a linear map $S^n \to S^m$ acting by right multiplication on rows.
If $S$ is finite as a set, we write
\[\ask_S(A(X_U)) = \frac 1 {\card{S U}} \sum_{x\in S U}\card{\Ker(A(x))}
\]
for the \underline{a}verage \underline{s}ize of the \underline{k}ernel of
these maps.

Let $\fO$ be a compact \DVR{} endowed with an $R$-algebra structure.
Recall that $\fP$ denotes the maximal ideal of $\fO$.
The \emph{(algebraic) ask zeta function} of $A(X_U)$ over $\fO$ is the formal
power series
\[
  \Zeta^{\ask}_{A/\fO}(T) =
  \Zeta^{\ask}_{A(X_U)/\fO}(T) = \sum_{k=0}^\infty
  \ask_{\fO/\fP^k}(A(X_U)) T^k\in \QQ[\![T]\!].
\]
If $A(X_U)$ and $B(X_U)$ are equivalent matrices of linear forms over $R$,
then $\Zeta_{A(X_U)/\fO}(T) = \Zeta_{B(X_U)/\fO}(T)$ for each $\fO$ as above.
We note that if $\fO$ has characteristic zero, then $\Zeta_{A(X_U)/\fO}(T) \in
\QQ(T)$; see \cite[Thm~4.10]{ask}.
As explained in \cite{ask}, ask zeta functions arise in the enumeration of linear
orbits and conjugacy classes of unipotent groups.

Writing $q = \card{\fO/\fP}$ for the residue field size of $\fO$,
we write $\zeta^{\ask}_{A(X_U)/\fO}(s) = \Zeta^{\ask}_{A(X_U)/\fO}(q^{-s})$
for the \emph{(analytic) ask zeta function} of $A(X_U)$ over $\fO$.
The series $\zeta^{\ask}_{A(X_U)/\fO}(s)$ converges for $\Real(s) > n$.
Moreover, $\zeta^{\ask}_{A(X_U)/\fO}(s)$ and $\Zeta^{\ask}_{A(X_U)/\fO}(T)$
determine one another so referring to both as ``the'' ask zeta function of
$A(X_U)$ constitutes only a minor abuse of terminology.

\paragraph{Duals, minors, and integrals.}
We will study the zeta functions $\zeta^{\ask}_{A(X_U)/\fO}(s)$ of interest to
us by means of suitable $\fP$-adic integrals.
Ignoring the harmless effect of replacing $A(X_U)$ by its transpose (see
\cite[Lem.\ 2.4]{ask}), using the ``Knuth duality'' operations from
\cite{ask2} and the machinery from \cite{ask},
given $A(X_U)$, we obtain three (in general very different) formulae
for $\zeta^{\ask}_{A(X_U)/\fO}(s)$ by means of $\fP$-adic integrals.
In addition, each of these formulae admits an affine and a projective
version.
Our choice here is the projective form of the integral attached to the
$\circ$-dual of $A(X_U)$.
To describe this explicitly, let us first order the elements of $U$ and write
$U = \{u_1,\dotsc,u_\ell\}$ where $\ell = \card U$.
Let $V = \{v_1,\dotsc,v_n\}$ be a set of cardinality $n$.
We may characterise $A(X_U)$ and a matrix $C(X_V) \in \Mat_{\ell\times
  m}(R[X_V])$ via
\begin{equation}
  \label{eq:A_circ_C}
  A(X_U)_{ij} = \sum_{k=1}^\ell \alpha_{ijk} X_{u_k},
  \qquad
  C(X_V)_{kj} = \sum_{i=1}^n \alpha_{ijk} X_{v_i},
\end{equation}
where $\alpha_{ijk} \in R$ .
We refer to $C(X_V)$ as a \emph{$\circ$-dual} of $A(X_U)$.
Our use of the indefinite article reflects the fact that $C(X_V)$ depends on
the chosen total orders.
Note that by construction, $A(X_U)$ is a $\circ$-dual of $C(X_V)$.
(The abstract version of the $\circ$-operation in~\cite{ask2} is
genuinely idempotent.)

Given a matrix $M$ over a ring $S$ and $k\ge 0$, we write $\Minors_k(M)$ or
$\Minors_k(M;S)$ for the ideal of $S$ generated by the $k\times k$ minors of
$M$.
We record the following observation.
\begin{lemma}
  \label{lem:minor_specialisation}
  Let $A(X_U) \in \Mat_{n\times m}(R[X_U])$.
  Let $S$ be an $R$-algebra and $x\in S U$.
  Let $i \ge 0$.
  Then $\Minors_i(A(x); S)$ is the ideal of $S$ generated by the image
  of $\Minors_i(A(X_U);R[X_U])$ under the specialisation map
  $R[X_U] \to S$ determined by $x$.
  \qed
\end{lemma}

Let us return to our $\circ$-dual $C(X_V)$ of $A(X_U)$.
Let $I_k = \Ideal_k\bigl(C(X_V); R[X_V]\bigr)$.
Note that $I_k$ is generated by homogeneous elements of degree $k$
and that $I_0 = \langle 1\rangle = R[X_V]$.

\begin{prop}[{Cf.\ \cite[\S\S 4.3--4.4]{ask}}]
  \label{prop:proj_circ_int}
  Let $\fO$ be a compact \DVR{} endowed with an $R$-algebra
  structure.
  Let $r$ be the rank of the image of $C(X_V)$ in $\Mat_{\ell \times
    m}(K[X_V])$ over $K(X_V)$.
  Then for all $s\in \CC$ with $\Real(s) > d$, we have
  \[
    (1-q^{-s}) \zeta^{\ask}_{A(X_U)/\fO}(s) = 1 + (1-q^{-1})^{-1}
    \int\limits_{(\fO V)^\times \times\, \fP}
    \abs{z}^{s - n + r- 1} \prod_{i=1}^r
    \frac{\norm{I_{i-1}(x)}}{\norm{I_i(x) \cup z I_{i-1}(x)}}
    \dd\!\mu(x,z).
  \]
\end{prop}

Here and in the following, for an ideal $I$ of $R[X_V]$ and $x\in SV$, we
write $I(x)$ for the ideal of $S$ generated by all $f(x)$ as $f(X_V)$ ranges
over $I$.
In the context of Proposition~\ref{prop:proj_circ_int},
by Lemma~\ref{lem:minor_specialisation}, we have
$I_i(x) = \Minors_i(C(x);S)$.

In our applications of Proposition~\ref{prop:proj_circ_int}, the ring $R$ is
of the form $R = \ZZ[1/N]$.
In that case, $r$ is simply the rank of $A(X_U)$ over $\QQ(X_U)$;
in particular, $r$ does not depend on $\fO$.

It is a folklore result in $\fP$-adic integration that zeta functions 
given by $\fP$-adic integrals defined in terms of monomial ideals are
\itemph{uniform} in the sense that for some rational function $W(X,T)$, these
integrals are of the form $W(q,q^{-s})$ as $\fO$ ranges over (suitable)
compact \DVR{}s.
The following makes this precise for ask zeta functions.

\begin{prop}
  \label{prop:general_monomial_integral}
  Let the notation be as in Proposition~\ref{prop:proj_circ_int}.
  Suppose that each of $I_1,\dotsc,I_r$ is a monomial ideal, say $I_k =
  \langle X_V^a : a \in A_k\rangle$ for a finite set $A_k \subset \NN_0 V$.
  Then there exists a rational function $W(X,T)\in \QQ(X,T)$
  (explicitly expressible in terms of $n$, $r$, and $A_1,\dotsc,A_k$)
  such that for all compact \DVR{}s $\fO$ endowed with an $R$-algebra
  structure, we have $\Zeta^{\ask}_{A(X_U)/\fO}(T) = W(q,T)$.
\end{prop}
\begin{proof}
  Apply \cite[Prop.\ 3.9]{topzeta} to the affine version~\cite[Eqn~(4.6)]{ask}
  of the integral in Proposition~\ref{prop:proj_circ_int}.
\end{proof}

\section{(Hyper)graphs and their ask zeta functions}
\label{s:(hyper)graphs}

\subsection{Hypergraph basics}
\label{ss:hypergraphs}

Two hypergraphs $\Eta = (V,E,\increl)$ and $\Eta' = (V',E',\increl')$ are
\emph{isomorphic} if there exist bijections $V \xto\psi V'$ and
$E \xto\phi E'$ such that for all $v\in V$ and $e\in E$, we have $v \increl e$
if and only if $v^\phi \increl' e^\psi$.

\paragraph{Incidence matrices.}

Let $\Eta$ have $m$ hyperedges and $n$ vertices.
Write $E = \{e_1,\dotsc,e_m\}$ and $V =\{v_1,\dotsc,v_n\}$.
Equivalently, we endow $E$ and $V$ with (arbitrary) total orders $\lte$ and
$\olte$, respectively, given by $e_1 \lte \dotsb \lte e_m$ and $v_1 \olte
\dotsb \olte v_n$.
Having made these choices, the associated \emph{incidence matrix} of $\Eta$ is
the $(0,1)$-matrix $A_\Eta\in \Mat_{n\times m}(\ZZ)$ given by
$(A_\Eta)_{ij} = \Iverson{v_i \increl e_j}$
(Iverson bracket notation, see \S\ref{ss:notation}).

\paragraph{Disjoint unions.}

Given hypergraphs $\Eta =
(V,E,\increl)$ and $\Eta' = (V',E',\increl')$, their disjoint union is
$\Eta \oplus \Eta' = (V \sqcup V', E \sqcup E', \increl
\sqcup \increl')$.
Given total orders on $V$ and $V'$ (resp.\ $E$ and $E'$),
we obtain a total order on $V \sqcup V'$ (resp.\ $E \sqcup E'$) in which the
elements of $V$ (resp.\ $E$) precede those of $V'$ (resp.\ $E'$).
With respect to these orders, we then have
$\sA_{\Eta \oplus \Eta'} =
\left[\begin{smallmatrix}\sA_\Eta & 0 \\ 0 & \sA_{\Eta'}\end{smallmatrix}\right]$.

\paragraph{Subhypergraphs.}

Let $\Eta = (V,E,\increl)$ be a hypergraph.
A \emph{subhypergraph} of $\Eta$ is a hypergraph $\Eta' = (V',E',\increl')$
with $V\subset V'$, $E\subset E'$, and $\increl' \, \subset \, \increl$.
Given subsets $V' \subset V$ and $E'\subset E$, the associated \emph{induced
  subhypergraph} of $\Eta$ is
\[
  \Eta[V' \mid E'] = \bigl(V', E', \increl \cap (V'\times E')\bigr).
\]
Subhypergraphs of (incidence hypergraphs of) graphs will play an important
role throughout this article;
see \S\ref{ss:graphs}.

Order the vertices and hyperedges of $\Eta$ as above to define the incidence
matrix $A_\Eta$.
Then the incidence matrix of $\Eta[V'\mid E']$
relative to the induced total orders on $V'$ and $E'$ is the
submatrix of $\sA_\Eta$ obtained by selecting rows from $V'$ and columns from
$E'$, respectively.

\subsection[Ask zeta functions associated with hypergraphs]{Ask zeta functions
  associated with hypergraphs: $\zeta^{\ask}_{\sA_\Eta/\fO}$}
\label{ss:ask_hypergraph}

Let $\Eta = (V,E,\increl)$ be a hypergraph with $m$ hyperedges and $n$
vertices.
Write $E = \{e_1,\dotsc,e_m\}$ and $V =\{v_1,\dotsc,v_n\}$.
Let $F = \Flags(\Eta) = \{ (v,e) \in V\times E : v\increl e\}$ be the set of
\emph{flags} of $\Eta$.
Let $\sA_\Eta = \sA_\Eta(X_F) \in \Mat_{n\times m}(\ZZ[X_F])$ be the matrix of
linear forms with
\[
  \sA_\Eta(X_F)_{ij} = \begin{cases}
    X_{(v_i,e_j)},&\text{if } v_i \increl e_j,\\
    0, & \text{otherwise.}
    \end{cases}
\]
Up to equivalence (see \S\ref{s:ask_bg}), $\sA_\Eta(X_F)$ only depends on
$\Eta$ and not on the chosen total orders.
Note that $\sA_\Eta(X_F)$ is obtained from the incidence matrix $A_\Eta$ of
$\Eta$ (w.r.t.\ the given total orders on $V$ and $E$) by replacing nonzero
entries by distinct variables.
We refer to $\sA_\Eta(X_F)$ as the \emph{linearised incidence matrix} of
$\Eta$.

\begin{rem}
  The zeta function we denote by $\zeta^{\ask}_{\sA_\Eta/\fO}(s)$ here coincides
  with $\Zeta^{\ask}_{\eta^\fO}(q^{-s})$ from \cite[\S 3.2]{cico},
  where $\eta$ denotes the \itemph{incidence representation} of
  $\Eta$.
\end{rem}

Write $f = \card F$.
By ordering $F$ lexicographically relative to the chosen total orders on $V$
and $E$, we may identify $F$ and $\{1,\dotsc,f\}$.
Let $\sC_\Eta = \sC_\Eta(X_V)\in \Mat_{f\times m}(\ZZ[X_V])$ be the matrix
such that for $(v,e)\in F$ and $j\in \{1,\dotsc m\}$, we have
\[
  \sC_\Eta(X_V)_{(v,e)j} = \begin{cases}
    X_v,&\text{if } e = e_j,\\
    0, & \text{otherwise.}
    \end{cases}
\]

It is easy to see (cf.\ \cite[\S 3.2]{cico}) that $\sC_\Eta$ is a $\circ$-dual
of $\sA_\Eta(X_F)$.
Let $\Minors_k \Eta = \Minors_k(\sC_\Eta) \subset \ZZ[X_V]$
denote the ideal of $\ZZ[X_V]$ generated by the $k\times k$ minors of $\sC_\Eta$.
In contrast to $\sC_\Eta$, the ideal $\Ideal_k \Eta$ only depends on
$\Eta$ and not on the arbitrary choices of total orders on $V$, $E$, and $F$
used to define $\sC_\Eta$.
Write $r_\Eta = \rank_{\QQ(X_V)}(\sC_\Eta)$;
we will derive a simple description of this number in
Proposition~\ref{prop:rank_C_Eta} below. 
For $W \subset \fO V\times \fO$, we define
\begin{equation}
  \label{eq:Eta_integral}
  \int\limits_W \Eta(s) :=
  \int\limits_{W}
  \abs{z}^{s - n + r_\Eta - 1}
  \prod_{i=1}^{r_\Eta}
  \frac{\norm{\Minors _{i-1} \Eta(x)}}
  {\norm{\Minors_i \Eta(x) \cup z \Minors_{i-1} \Eta(x)}}
  \dd\!\mu(x,z),
\end{equation}
where $\mu$ denotes the normalised Haar measure on $\fO V\times \fO$.
Proposition~\ref{prop:proj_circ_int} (with $R = \ZZ$) then yields the following.

\begin{prop}
  \label{prop:ask_Eta_integral}
  For each compact \DVR{} $\fO$
  and $s \in \CC$ with $\Real(s) > n$,
  we have
  \[
    (1-q^{-s}) \zeta^{\ask}_{\sA_\Eta/\fO}(s) =
    1 + (1 - q^{-1})^{-1} \int\limits_{(\fO V)^\times\times \fP} \!\!\!\! \Eta(s).
    \pushQED{\qed}
    \qedhere
    \popQED
  \]
\end{prop}

\subsection{Graph basics}
\label{ss:graphs}

By a \emph{graph}, we mean a pair $\Gamma = (V,E)$ where $V$ is a finite set
and $E$ is a set of subsets of $V$, each of which has cardinality $1$ or $2$.
As usual, we refer to the elements of $V$ and $E$ as the \emph{vertices} and
\emph{edges} of $\Gamma$, respectively.
We explicitly allow loops (i.e.\ edges $e$ with $\card{e} = 1$) but no
parallel edges.
Let $\sim \,=\, \sim_\Gamma\, \subset V\times V$ be the (symmetric) adjacency
relation of $\Gamma$.
Hence, $v\sim v'$ if and only if $\{v,v'\} \in E$.
By a \emph{subgraph} of $\Gamma$, we mean a graph $\Gamma' = (V',E')$ with
$V'\subset V$ and $E'\subset E$.

\paragraph{Hypergraphs from graphs.}
Let $\increl \,=\, \increl_\Gamma\, \subset V\times E$ be the incidence
relation of $\Gamma$.
(Hence, $v\increl e$ if and only if $v\in e$.)
Every graph gives rise to two hypergraphs that will be of interest to us:
the \emph{incidence hypergraph} $\Inc(\Gamma) = (V,E, \increl_\Gamma)$ and
the \emph{adjacency hypergraph} $\Adj(\Gamma) = (V,V,\sim_\Gamma)$.
The former of these simply amounts to viewing a graph $\Gamma$ as a hypergraph
whose hyperedges are the edges of $\Gamma$ with the evident incidence
relation.

\paragraph{Disjoint unions and joins.}
Given graphs $\Gamma$ and $\Gamma'$, we let $\Gamma \oplus \Gamma'$
denote their disjoint union.
We have $\Inc(\Gamma \oplus \Gamma') = \Inc(\Gamma) \oplus \Inc(\Gamma')$ and
$\Adj(\Gamma \oplus \Gamma') = \Adj(\Gamma) \oplus \Adj(\Gamma')$.
The \emph{join} $\Gamma \join \Gamma'$ of $\Gamma$ and $\Gamma'$ is
obtained from $\Gamma \oplus \Gamma'$ by adding an edge connecting each
vertex of $\Gamma$ to each vertex of $\Gamma'$.

\paragraph{Adjacency and incidence matrices.}
Let $\Gamma$ have $n$ vertices, say $V = \{ v_1,\dotsc,v_n\}$.
As usual, the \emph{adjacency matrix} $A_\Gamma \in \Mat_n(\ZZ)$ of $\Gamma$
(relative to the chosen total order on $V$) is the $(0,1)$-matrix given by
$(A_\Gamma)_{ij} = \Iverson{v_i \sim v_j}$.
Note that, using the same total order on $V$ throughout, the adjacency matrix
$A_\Gamma$ of $\Gamma$ coincides with the incidence matrix $A_{\Adj(\Gamma)}$
of the adjacency hypergraph of $\Gamma$.

\subsection[Two ask zeta functions associated with graphs]{Two ask zeta
  functions associated with graphs: $\zeta^{\ask}_{\sA^+_\Gamma/\fO}$ and
  $\zeta^{\ask}_{\sA^-_\Gamma/\fO}$}
\label{ss:ask_graph}

Let $\Gamma = (V,E)$ be a graph with $m$ edges and $n$ vertices.
As in the previous section, we write $V = \{v_1,\dotsc,v_n\}$ which reflects a
choice of a total order $\lte$ on $V$ with $v_1\lte \dotsb \lte v_n$.
Each edge $e \in E$ is of the form $\{v_i,v_j\}$ with
$i\le j$ and $v_i\sim v_j$.
We obtain a total order on $E$, which we again denote by $\lte$, by mapping
$e$ to $(i,j)$ and by ordering the resulting pairs of numbers lexicographically.
Let $e_1 \lte \dotsb\lte e_m$ be the distinct edges of $\Gamma$.
Using this order, we identify $E$ and $\{1, \dotsc, m\}$.

Following (and, in fact, slightly generalising) \cite{cico}, we now define
matrices $\sA^+_\Gamma(X_E)$ and $\sA^-_\Gamma(X_E)$  in $\Mat_{n\times n}(\ZZ[X_E])$.
Namely, $\sA^\pm_\Gamma$ is the matrix given by the following conditions:
\begin{itemize}
\item
  For $1 \le i\le j\le n$, we have
  \[
    \sA^\pm_\Gamma(X_E)_{ij} = \begin{cases}
      X_e, & \text{if }e := \{v_i,v_j\} \in E, \\
      0, & \text{otherwise.}
      \end{cases}
  \]
\item
  For $1 \le i < j \le n$, we have
  $\sA^\pm_\Gamma(X_E)_{ij} = \pm \sA^\pm_\Gamma(X_E)_{ji}$.
\end{itemize}

We refer to $\sA^+_\Gamma$ and $\sA^-_\Gamma$ as the \emph{linearised
  adjacency matrices} of $\Gamma$.
Note that $\sA^+_\Gamma$ is symmetric and if $\Gamma$ is loopless, then
$\sA^-_\Gamma$ is antisymmetric.
The matrix $\sA^-_\Gamma + (\sA^-_\Gamma)^\top$ is diagonal and the 
$(i,i)$-entry of $\sA^\pm_\Gamma$ is $X_{\{v_i\}}$ if $v_i \sim v_i$ in
$\Gamma$ and zero otherwise.

\begin{rem}
  The zeta function we denote by $\zeta^{\ask}_{\sA^\pm_\Gamma/\fO}(s)$ here coincides
  with $\Zeta^{\ask}_{\gamma_\pm^\fO}(q^{-s})$ from \cite{cico}, where
  $\gamma_\pm$ denotes the \itemph{(positive or negative) adjacency representation} of
  $\Gamma$ \cite[\S 3.2]{cico}.
  We note that $\Gamma$ was assumed to be loopless in the definition of
  $\gamma_-$ in \cite{cico}.
\end{rem}

Let $\sC_\Gamma^\pm = \sC_\Gamma^\pm(X_V)\in \Mat_{m \times n}(\ZZ[X_V])$ be
the matrix defined as follows:
given an edge $e = \{v_i,v_j\}$ with $i\le j$ and $k\in \{1,\dotsc,n\}$,
we define
\[
  (\sC^\pm_\Gamma)_{ek} =
  \begin{cases}
    \phantom\pm X_i, & \text{if } k = j, \\
    \pm X_j, &\text{if } k = i \text{ and } i\not= j, \\
    \phantom\pm 0, &\text{otherwise}.
  \end{cases}
\]
It is easy to see (cf.\ \cite[\S 3.3]{cico}) that $\sC^\pm_\Gamma(X_V)$ is a
$\circ$-dual of $\sA^\pm_\Gamma(X_F)$.

Let  $\Minors^\pm_k \Gamma = \Minors_k(\sC_\Gamma^\pm)$, an ideal of
$\ZZ[X_V]$.
We also write $\frac 1 2 \Minors_k^+ \Gamma$ for the ideal of
$(\ZZ[\frac 1 2])[X_V]$ generated by $\Minors^+_k \Gamma$.
As for hypergraphs, all of these ideals only depend on $\Gamma$ and the
chosen ``type'', $+$ or $-$, not on the total order on the vertices of
$\Gamma$ used to define~$\sC^\pm_\Gamma$.

Write $r_\Gamma[\pm] = \rank_{\QQ(X_V)}(\sC^\pm_\Gamma)$;
we will obtain graph-theoretic interpretations of these numbers in
Proposition~\ref{prop:rank_C-} and Proposition~\ref{prop:rank_C+}.
For $W\subset \fO V\times \fO$, we define
\begin{equation}
  \label{eq:Gamma_integral}
  \int\limits_W \Gamma^\pm(s) :=
  \int\limits_W
  \abs{z}^{s - n + r_\Gamma[\pm] - 1}
  \prod_{i=1}^{r_\Gamma[\pm]}
  \frac{\norm{\Minors^\pm _{i-1} \Gamma (x)}}
  {\norm{\Minors^\pm_i \Gamma(x) \cup z \Minors^\pm_{i-1} \Gamma(x)}}
  \dd\!\mu(x,z),
\end{equation}
where $\mu$ is the normalised Haar measure on $\fO V\times \fO$.
Proposition~\ref{prop:proj_circ_int} then yields the following.

\begin{prop}
  \label{prop:ask_Gamma_integral}
  For each compact \DVR{} $\fO$ and $s \in \CC$ with $\Real(s) > n$, we have
  \[
    (1-q^{-s}) \zeta^{\ask}_{\sA^\pm_\Gamma/\fO}(s) =
    1 + (1 - q^{-1})^{-1} \int\limits_{(\fO V)^\times\times\fP} \!\!\!\! \Gamma^\pm(s).
    \pushQED{\qed}
    \qedhere
    \popQED
  \]
\end{prop}

\begin{rem}
  \label{rem:odd_integral}
  If $\fO$ has odd residue characteristic in the $+$-case of
  Proposition~\ref{prop:ask_Gamma_integral}, then we may replace each
  $\Minors_i^+\Gamma$ by $\frac 1 2 \Minors^+_i \Gamma$ in the definition of
  $\int_W \Gamma^+(s)$.
\end{rem}

\section{Selectors of hypergraphs}
\label{s:selectors}

Let $\Eta$ be a hypergraph.
We derive an explicit combinatorial parameterisation of the
minors of $\sC_\Eta$ (see \S\ref{ss:ask_hypergraph}) in terms of partial functions that
we call \itemph{selectors}.
Apart from immediately providing us with a new proof of
Theorem~\ref{thm:uniformity}\ref{thm:uniformity1}, the results of this section
will play a key role in our proof of Theorem~\ref{thm:rgmt} in \S\ref{s:RGMT}.

\paragraph{Partial functions.}

Given a set $U$, we let $U_\udef$ denote the set obtained from $U$ by
adjoining an additional element~$\udef$.
We write $\phi\colon U \parto V$ and $U\xparto{\phi}V$ to
indicate a partial function $\phi$ from $U$ to $V$.
The \emph{domain of definition} $\Dom(\phi)$ of $\phi$ consists of those
$u\in U$ for which $u^\phi$ is defined.
We tacitly identify partial functions $\phi\colon U\parto V$ and those total
functions $U_\udef \to V_\udef$ which send $\udef$ to $\udef$.
For $Y\subset V_\udef$, write
$Y^{\phi^*} = \{ u \in U_\udef : u^\alpha \in Y\}$.
In particular, $\Dom(\phi) = \{ u \in U : u^\phi \not= \udef\} = V^{\phi^*}$.
We write $\deg(\phi) = \card{\Dom(\phi)}$ for the \emph{degree} of $\phi$.
Given $U\xparto{\phi} V$, let
\[
  \mon\phi =
  \prod\limits_{u\in \Dom(\phi)} X_{u^\phi} \in \ZZ[X_V].
\]
Note that $\deg(\phi)$ is the (total) degree of $\mon\phi$ as a monomial in
$\ZZ[X_V]$.

We will often find the need to modify or extend partial functions.
Let $U\xparto\phi V$ be a partial function. Let $u_1,\dotsc,u_r\in U$ be
distinct elements and let $v_1,\dotsc,v_r\in V$ be arbitrary.
We define \[\phi[u_1\subs v_1,\dotsc,u_r\subs v_r]\] to be the partial
function $U\parto V$ given by
\[
u\mapsto \begin{cases}
  v_i, & \text{if } u = u_i,\\
  u^\phi, & \text{otherwise}.
\end{cases}
\]
Let $U'\subset U$.
We identify partial functions $U'\parto V$ and those partial functions
$U\parto V$ that are undefined outside of $U'$.
Given $U\xparto\phi V$, we let $\phi\restriction U'$ denote the partial function
\[
  u\mapsto \begin{cases} u^\phi, & \text{if } u\in U',\\
    \udef,& \text{otherwise.}\end{cases}
\]

\paragraph{Selectors.}
Let $\Eta = (V,E,\increl)$ be a hypergraph.
By a \emph{(vertex) selector} of $\Eta$, we mean a partial function
$E\xparto\phi V$ such that $e^\phi \increl e$ for all $e\in \Dom(\phi)$.
In other words, a selector $\phi$ of~$\Eta$ consists of a subset $E' =
\Dom(\phi)$ of $E$ together with a choice of a vertex $e^\phi$ incident with
$e \in E'$.
Let $\Selectors(\Eta)$ be the set of all selectors of $\Eta$ and
$\Selectors_k(\Eta) = \{\phi\in\Selectors(\Eta) : \deg(\phi) = k \}$.

\paragraph{Minors.}
Using Propositions~\ref{prop:general_monomial_integral}
and \ref{prop:ask_Eta_integral}, the following
description of the minors of $\sC_\Eta$ provides a new proof
of Theorem~\ref{thm:uniformity}\ref{thm:uniformity1}, quite different from the
one in \cite[\S 4.4]{cico}.

\begin{prop}
  \label{prop:selectors}
  We have
  $\Minors_k \Eta = \Bigl\langle \mon\phi :
  \phi\in\Selectors_k(\Eta) \Bigr\rangle$.
  More precisely, the nonzero minors of $\sC_\Eta$ are precisely of the
  form $\pm \mon\phi$ for $\phi\in\Selectors(\Eta)$.
\end{prop}
\begin{proof}
  Let $\Eta$ have $m$ hyperedges and $n$ vertices.
  Write $E = \{ e_1,\dotsc,e_m \}$ and $V = \{ v_1,\dotsc,v_n \}$.
  Define and order $F = \Flags(\Eta)$ as in \S\ref{ss:ask_hypergraph}.
  Given $F'\subset F$ and $E'\subset E$ with $\card{F'} = \card{E'} = k$,
  let $\sC_\Eta[E'\mid F']$ be the submatrix of $\sC_\Eta$ with rows indexed by $F'$
  and columns indexed by $E'$.
  Write $m_\Eta[E'\mid F'] = \det(\sC_\Eta[E' \mid F'])$ for the associated minor.

  We first show that given $F'$ and $E'$ as above, either $m_\Eta[E' \mid F'] = 0$
  or $m_\Eta[E' \mid F'] = \pm \mon \phi$ for some $\phi \in \Selectors_k(\Eta)$.
  Let $(u_1,f_1),\dotsc,(u_k,f_k)$ be the distinct elements of $F'$.
  If $f_i = f_j$ for $i\not= j$, then the rows of $\sC_\Eta[E' \mid F']$ indexed by
  $(u_i,f_i)$ and $(u_j,f_j)$ are linearly dependent over $\QQ(X_V)$ whence
  $m_\Eta[E' \mid F'] = 0$.
  We may thus assume that $E'' := \{ f_1,\dotsc,f_k\}$ has cardinality $k$.
  Next, if $f_i \notin E'$ for some $i$, then the row of $\sC_\Eta[E' \mid F']$ indexed by
  $(u_i,f_i)$ is zero so that again $m_\Eta[E' \mid F'] = 0$.
  Therefore, we may assume that $E' = E''$.
  In this case, we clearly have $m_\Eta[E' \mid F'] = \pm \prod_{i=1}^k X_{u_i}$.
  Define a selector $E\xparto\phi V$ via $\Dom(\phi) = E'$ and $f_i^\phi =
  u_i$ for $i = 1,\dotsc,k$.
  Then $m_\Eta[E' \mid F'] = \pm \mon\phi$.

  Conversely, given $\phi\in \Selectors(\Eta)$,
  let $E' = \Dom(\phi)$ and
  $F' = \{ (e^\phi,e) : e\in E'\}$ so that
  $m_\Eta[E' \mid F'] = \pm \mon\phi$.
\end{proof}

Our proof of Proposition~\ref{prop:selectors} will act as a template for the
much more involved case of the matrices $\sC^\pm_\Gamma$ in \S\ref{s:animations}.

\begin{cor}
  \label{cor:rank_C_Eta_selectors}
  Let $\Eta = (V,E,\increl)$ be a hypergraph.
  Then
  \[
    \rank_{\QQ(X_V)}(\sC_\Eta) = \max(k \ge 0 : \Selectors_k(\Eta) \not=
    \emptyset) = \max(\deg(\phi) : \phi\in \Selectors(\Eta)).
    \tag*{\qed}
  \]
\end{cor}

\begin{proof}[New proof of Theorem~\ref{thm:uniformity}\ref{thm:uniformity1}]
  Combine
  Proposition~\ref{prop:general_monomial_integral},
  Proposition~\ref{prop:ask_Eta_integral}, and
  Proposition~\ref{prop:selectors}.
\end{proof}

\section{Towards minors: paths and cycles in hypergraphs}
\label{s:towards_minors}

Our next goal, to be achieved in \S\ref{s:animations}, is to provide a
combinatorial parameterisation of the minors of the matrices
$\sC^\pm_\Gamma$ associated with a graph $\Gamma = (V,E)$
in \S\ref{ss:ask_graph}.  The columns and rows of $\sC^\pm_\Gamma$
correspond to elements of $V$ and $E$, respectively.  Given subsets
$V'\subset V$ and $E'\subset E$ with $\card{V'} = \card{E'} = k$, we
therefore obtain an associated $k\times k$ minor
$m_\Gamma^{\pm}[V'\mid E']$ of $\sC^\pm_\Gamma$.  It turns out that
these minors can be conveniently studied in terms of the
subhypergraphs $\Eta[V'\mid E']$ of the incidence hypergraph $\Eta =
\Inc(\Gamma)$ of $\Gamma$.

\subsection{Degeneracy and connectivity of hypergraphs}{

Let $\Eta = (V,E,\increl)$ be a hypergraph.
Two hyperedges $e,e'\in E$ are \emph{parallel} if $\norm{e} = \norm{e'}$.
If $\norm{e} \not= \norm{e'}$ whenever $e\not= e'$ for $e,e'\in E$, then
$\Eta$ is \emph{simple}.
A \emph{loop} of $\Eta$ is a hyperedge $e\in E$ with $\#\norm e = 1$.
Following \cite[\S 1.1]{Bre13}, we call
$\isolated(\Eta) = \{ v\in V: v \!\not\increl e \text{ for all } e\in E\}$ and
$\empties(\Eta) = \{ e\in E : \norm e = \emptyset\}$
the sets of \emph{isolated} vertices and \emph{empty} hyperedges of~$\Eta$,
respectively.
We call $\Eta$ \emph{nondegenerate} if $\isolated(\Eta) = \empties(\Eta) =
\emptyset$.

The following notion of connectivity of hypergraphs is adapted from
\cite[\S 1.2]{Bre13}.
Given vertices $u,v\in V$, a \emph{walk} of length $r \ge 0$ from $u$ to $v$
is a sequence
\[
  u = u_1,\, e_1,\, u_2,\, e_2,\, \dotsc,\, u_r, \, e_{r},\, u_{r+1} = v
\]
with $u_1,\dotsc,u_{r+1} \in V$ and $e_1,\dotsc,e_r \in E$ such that always
$u_i \increl e_i$ and $u_{i+1} \increl e_i$.

We say that vertices $u,v\in V$ of $\Eta$ are \emph{connected} if there
exists a walk from $u$ to~$v$.
This defines an equivalence relation on $V$.
Let $V_1,\dotsc,V_c$ be the distinct equivalence classes.
Write $E_j = \{ e\in E : \norm e \cap V_j \not= \emptyset\}$ and define
$\Eta_j = \Eta[V_j\mid E_j]$ to be the associated induced subhypergraph.
Let $\Eta_0 = (\emptyset, \empties(\Eta),\emptyset)$.
It is then easy to see that $\Eta = \Eta_0 \oplus \dotsb \oplus \Eta_c$.
We refer to the subhypergraphs $\Eta_1,\dotsc,\Eta_c$ as the
\emph{connected components} of $\Eta$.
We say that $\Eta$ is \emph{connected} if $\empties(\Eta) = \emptyset$ and $c
= 1$.
(Note that using this definition, every hypergraph without vertices is
disconnected.)
If $\Gamma$ is a graph, then
$\Inc(\Gamma)$ is connected if and only if $\Gamma$ is connected in the usual
graph-theoretic sense.

\begin{lemma}
  A hypergraph $\Eta = (V,E,\increl)$ is connected if and only if $V \not= \emptyset$,
  $\empties(\Eta) = \emptyset$, and whenever $\Eta = \Eta_1\oplus \Eta_2$ for
  hypergraphs $\Eta_1$ and $\Eta_2$, one of the summands is
  $(\emptyset,\emptyset,\emptyset)$.
\end{lemma}
\begin{proof}
  Suppose that $\Eta$ is connected.
  As connectivity yields a (single) equivalence class on~$V$,
  we have $V \not =\emptyset$.
  Suppose that $\Eta = \Eta_1 \oplus \Eta_2$
  for $\Eta_j = (V_j,E_j,\increl_j)$.
  If $v_j \in V_j$ for $j=1,2$, then $v_1$ and $v_2$ cannot be connected in
  $\Eta$.
  Since $\Eta$ is connected, we may assume that, without loss of
  generality, $V_1 = \emptyset$.
  As $\empties(\Eta) = \emptyset$, we then obtain
  $E_1 = \emptyset = \,\increl_1$.

  Conversely, suppose that $\Eta$ satisfies the conditions stated.
  Write $\Eta = \Eta_0 \oplus \dotsb \oplus \Eta_c$ as in the paragraph preceding
  this lemma.
  Since $V\not= \emptyset$, we have $c \ge 1$ and since $\empties(\Eta) =
  \emptyset$, we have $\Eta_0 = (\emptyset,\emptyset,\emptyset)$.
  Hence, $\Eta = \Eta_1\oplus \dotsb \oplus \Eta_c$ whence $c = 1$ by
  assumption.
\end{proof}

The following lemma simply asserts that the determinant of a
block diagonal square matrix (with entries in some ring) can only be nonzero
if all diagonal blocks are squares.

\begin{lemma}
  \label{lem:A_Eta_factor}
  Let $\Eta = \Eta_1 \oplus \dotsb \oplus \Eta_r$, where each $\Eta_j =
  (V_j,E_j, \increl_j)$ is a hypergraph.
  Write $m_j = \card{E_j}$ and $n_j = \card{V_j}$.
  Suppose that $\sum_{j=1}^r m_j = \sum_{j=1}^r n_j$.
  \begin{enumerate}[(i)]
  \item
  \label{lem:A_Eta_factor1}
    If $m_j \not= n_j$ for some $j$, then $\det(\sA_\Eta) = 0$.
  \item
  \label{lem:A_Eta_factor2}
    If $m_j = n_j$ for all $j$, then
    $\det(\sA_\Eta) = \pm \prod_{j=1}^r \det(\sA_{\Eta_j})$. 
  \end{enumerate}
\end{lemma}
\begin{proof}
  Only \ref{lem:A_Eta_factor1} merits a proof.
  If $m_j \not= n_j$ for some $j$
  and $\sum_{j=1}^r m_j = \sum_{j=1}^r n_j$,
  then $n_j > m_j$ for some $j$.
  The rows of $\sA_{\Eta_j} \in \Mat_{n_j \times
    m_j}(\ZZ[X_{V_j}])$ are then linearly dependent over $\QQ(X_{V_j})$ whence
  $\det(\sA_{\Eta}) = 0$.
\end{proof}

\subsection{Hypergraph decompositions of graphs and nonzero minors}
\label{ss:hypergraph_decompositions}

Let $\Gamma = (V,E)$ be a graph with associated incidence hypergraph
$\Eta = \Inc(\Gamma) = (V,E,\increl)$.
The minors of $\sC^\pm_\Gamma$ are parameterised by pairs $(V',E')$ with $V'
\subset V$, $E' \subset E$, and $\card{V'} = \card{E'}$.
While such a pair $(V',E')$ may not be a subgraph of $\Gamma$, we may
always consider the associated induced subhypergraph $\Eta' = \Eta[V'\mid E']$
of $\Eta$ as in \S\ref{ss:hypergraphs}.

Let $\Gamma$ have $m$ edges and $n$ vertices.
Order and index the vertices and edges of $\Gamma$ as in~\S\ref{ss:ask_graph}.
Given $V'\subset V$ and $E'\subset E$, we
let $\sC^\pm_\Gamma[V'\mid E']$ denote the submatrix of~$\sC^\pm_\Gamma$ obtained
by selecting the rows indexed by elements of $E'$ and the columns indexed by
elements of $V'$.
Henceforth, we assume that $\card{V'} = \card{E'} = k$.
We write $m^\pm_\Gamma[V'\mid E'] = \det(\sC^\pm_\Gamma[V'\mid E'])$ for the
minor of $\sC^\pm_\Gamma$ corresponding to $(V',E')$.

\begin{lemma}
  \label{lem:degenerate}
  If the hypergraph $\Eta[V'\mid E']$ is degenerate, then $m^\pm_\Gamma[V' \mid E'] = 0$.
\end{lemma}
\begin{proof}
  An isolated vertex (resp.\ empty hyperedge) of $\Eta[V' \mid E']$ gives rise
  to a zero column (resp.\ zero row) of $\sC^\pm_\Gamma[V'\mid E']$.
\end{proof}

Let $\Eta' = \Eta[V'\mid E']$ be nondegenerate.
Let $\Eta' = \Eta'_1 \oplus \dotsb \oplus \Eta'_c$
be its decomposition into connected components.
Write $\Eta'_j = \Eta[V_j' \mid E_j']$, where
$V' = V'_1 \sqcup \dotsb \sqcup V'_c$ and $E' = E'_1\sqcup \dotsb \sqcup E'_c$.
We call a hypergraph \emph{square} if it has as many hyperedges as
vertices.

\begin{lemma}
  \label{lem:minor_factorisation}
  \quad
  \begin{enumerate}
  \item
    If some $\Eta'_j$ is nonsquare, then $m^\pm_\Gamma[V'\mid E'] = 0$.
  \item
    If each $\Eta'_j$ is square, then
    $m^\pm_\Gamma[V'\mid E'] = \pm \prod\limits_{j=1}^c m^\pm_\Gamma[V_j' \mid E_j']$.
  \end{enumerate}
\end{lemma}
\begin{proof}
  By changing our total order on $V$ (resp.\ $E$) if needed,
  we may assume that each element of $V'_j$ (resp.\ $E'_j$) precedes
  every element of $V'_{j+1}$ (resp.\ $E'_{j+1}$).
  The matrix $\sC^\pm_\Gamma[V' \mid E']$ is then a block diagonal matrix with blocks
  $\sC^\pm_\Gamma[V'_j\mid E'_j]$.
  Both claims then follow from Lemma~\ref{lem:A_Eta_factor}.
  Indeed, every $a\times b$ matrix (with entries in a ring) is obtained by
  specialising $\sA_{\Eta_{a,b}}$, where, following \cite[\S 3.1]{cico},
  $\Eta_{a,b}$ is the block hypergraph with $a$
  vertices and $b$ hyperedges such that every vertex is incident with every
  hyperedge.
\end{proof}

\begin{rem}
  \label{rem:snapping_edges}
  Even though $\Gamma$ is assumed to be simple as a graph---equivalently,
  $\Inc(\Gamma)$ is assumed simple as a hypergraph---$\Eta'$ need not be simple.
  In particular, suppose that $e = \{u,v\}$ and $f = \{u,w\}$ 
  with $v\not= w$ are edges of $\Gamma$ that both belong to $E'$.
  Further suppose that $u\in V'$ but $v,w \notin V'$.
  Then $e$ and $f$ are distinct parallel loops of $\Eta[V'\mid E']$.
\end{rem}

To summarise: in studying the minors $m^\pm_\Gamma[V'\mid E']$, we
obtained a reduction to the case that $\Eta'$ is nondegenerate,
connected, and square.

\subsection{Unicyclic graphs and related hypergraphs}

Following \cite{Man69}, a graph is \emph{unicyclic} if it is
connected and contains a unique cycle.
Equivalently, a graph is unicyclic if and only if it is connected and contains
as many vertices as edges.
Unicyclic graphs are precisely those graphs obtained from a tree by
adding an edge connecting two previously non-adjacent (not necessarily
distinct!)~vertices.

In the preceding subsection, we reduced the study of the minors
$m^\pm_{\Gamma}[V'\mid E']$ to the case when $\Eta[V'\mid E']$ is connected
and nondegenerate, where $\Eta = \Inc(\Gamma)$.
The following outlines the highly restricted possible shapes of the subhypergraphs
$\Eta[V'\mid E']$.

\begin{lemma}[``Unicyclicity lemma'']
  \label{lem:unicyclic}
  Let $\Gamma = (V,E)$ be a graph
  and $\Eta = \Inc(\Gamma)$.
  Let $V'\subset V$ and $E'\subset E$ have the same cardinality.
  Let $\Eta' = \Eta[V' \mid E']$ be connected
  and nondegenerate.
  Then precisely one of the following conditions is
  satisfied.
  \begin{enumerate}[(U1)]
  \item
    \label{lem:unicyclic1}
    $(V',E')$ is a unicyclic loopless subgraph of $\Gamma$ and
    $\Eta'$ is its incidence hypergraph.
  \item
    \label{lem:unicyclic2}
    $\Eta'$ contains a unique loop $e_\circ$ and $(V',E'\setminus\{e_\circ\})$
    is a subtree of $\Gamma$.
  \end{enumerate}
\end{lemma}
\begin{proof}
  We write $k = \card{V'} = \card{E'}$.
  As $\Eta'$ is nondegenerate and $\Gamma$ is a graph,
  every hyperedge of $\Eta'$ contains either one or two vertices.
  Let $E'' \subset E'$ consist of those $e\in E'$ with
  $\#\norm{e}_{\Eta'} = 2$.
  As loops are irrelevant for connectivity, the hypergraph $\Eta'' =
  \Eta[V'\mid E'']$ is still connected.
  Note that this hypergraph is the incidence hypergraph of the loopless graph
  $\Gamma'' = (V',E'')$.
  It follows that $\Gamma''$ is connected (as a graph) with $k$ vertices.
  Therefore, $\Gamma''$ contains at least $k-1$ edges.
  On the other hand, since $E'' \subset E'$, the graph $\Gamma''$
  contains at most $k$ edges.
  This leaves us with two cases.
  \begin{enumerate}
  \item
    If $E''= E'$ (i.e.\ $\Gamma''$ has $k$ edges), then $\Eta'$ is the
    incidence hypergraph of the unicyclic loopless graph $(V',E')$.
  \item
    Otherwise, $\Eta'$ contains a unique loop $e_\circ$,
    we have $E' = E'' \cup \{ e_\circ\}$,
    and $\Gamma'' = (V',E'\setminus\{e_\circ\})$ is a tree.
    The edge $e_\circ$ may or may not be a loop of $\Gamma$;
    cf.\ Remark~\ref{rem:snapping_edges}.
    \qedhere
  \end{enumerate}
\end{proof}

\section{Animations of graphs}
\label{s:animations}

Proposition~\ref{prop:selectors} not only shows that each ideal
$\Minors_k \Eta$ of minors attached to a hypergraph $\Eta$ is
monomial. It also provides explicit (monomial) generators parameterised by
combinatorial gadgets, namely selectors.  This parameterisation is not
faithful: it is easy to produce examples of distinct selectors giving
rise to the same minor.  In this section, we will derive similar
(albeit more delicate) parameterisations of monomial generators of
$\Minors^-_k \Gamma$ and $\frac 1 2 \Minors^+_k \Gamma$ by means of
what we call \itemph{animations} of $\Gamma$.  These parameterisations
are not faithful either.  In later sections, this will emerge as a
useful feature which allows us to manipulate animations by means of
combinatorial procedures.

Throughout this section, let $\Gamma = (V,E)$ be a graph with incidence
hypergraph $\Eta = \Inc(\Gamma)$; see \S\ref{ss:graphs}.

\subsection{Nilpotency, periodic, and transient points}
\label{ss:nilpotency}

Let $U$ be a finite set.
Let $\Parfun(U)$ denote the set of all partial functions $U\parto U$;
see~\S\ref{s:selectors}.
This is a monoid with respect to composition (of functions $U_\udef \to U_\udef$) with
zero element given by the nowhere defined function.
In particular, we have a natural notion of nilpotency for elements of
$\Parfun(U)$.
Namely, $\phi \in \Parfun(U)$ is \emph{nilpotent} if there exists some $n \ge
1$ such that the $n$-fold composite $\phi^n$ sends all points of $U$ to $\udef$.

Let $U\xparto\phi U$ be given.
Borrowing terminology from finite dynamical systems,
we call $u \in U$ a \emph{$\phi$-periodic} point if $u^{\phi^n} = u$ for some $n
\ge 1$; otherwise, $u$ is a \emph{$\phi$-transient} point.
Let $U^{\per} \subset U$ consist of all $\phi$-periodic points.
Then $U^\per_\udef = U^\per \sqcup\{\udef\}$ is the set of periodic points
of $\phi$ viewed as a total function $U_\udef \to U_\udef$.
Let $U^\tra\subset V$ be the set of $\phi$-transient points on $U_\udef$.
That is,
\[
  U^\tra = \bigl\{
  u \in U : u^{\phi^k}  \in U^\per_\udef \text{ for some } k\ge 1
  \bigr\}.
\]

Clearly, $\phi$ induces a permutation of $U^\per$.
By the \emph{$\phi$-orbits} on $U^\per$, we mean the orbits
of the infinite cyclic group acting on $U^\per$ via $\phi$.
Let $\odlen(\phi)$ denote the number of $\phi$-orbits of odd length $ > 1$.
We call $\phi$ \emph{odd-periodic} if
all $\phi$-periodic points have odd $\phi$-periods.

\subsection{Animations}
\label{ss:animations}

An \emph{animation} of $\Gamma$ is a partial function $V\xparto\alpha V$
such that $v^\alpha \sim v $ for all $v \in V$ with $v^\alpha \not= \udef$.
Equivalently, an animation of $\Gamma$ is a selector of the adjacency
hypergraph $\Adj(\Gamma)$ (see \S\ref{ss:graphs}).
We write $\Animations(\Gamma)$ for the set of all animations of~$\Gamma$.
Let $\Nilpotents(\Gamma)$ and $\Oddperiods(\Gamma)$ denote the set of
nilpotent and odd-periodic animations of~$\Gamma$, respectively.
Let $\Fixed(\Gamma)$ denote the set of those animations
$\alpha \in \Animations(\Gamma)$ 
such that every $\alpha$-periodic point is a fixed point of $\alpha$.
We clearly have
\[
  \Nilpotents(\Gamma) \subset \Fixed(\Gamma) \subset \Oddperiods(\Gamma)
  \subset \Animations(\Gamma).
\]
We write $\Nilpotents_k(\Gamma)$, $\Fixed_k(\Gamma)$, $\Oddperiods_k(\Gamma)$, and
$\Animations_k(\Gamma)$ for the respective subsets consisting of animations of
degree $k$.
The following result, proved over the course of this section, is the key
ingredient of all main results of the present article.

\begin{thm}
  \label{thm:animations}
  Let $\Gamma$ be a graph and $k \ge 0$.
  Then:
  \begin{enumerate}[(i)]
  \item\label{thm:animations-}
    $\Minors^-_k \Gamma = \Bigl\langle \mon\alpha :
    \alpha\in\Fixed_k(\Gamma) \Bigr\rangle$.
    More precisely, the nonzero minors of $\sC^-_\Gamma$ are precisely of the
    form $\pm \mon\alpha$ for $\alpha\in\Fixed(\Gamma)$.
  \item\label{thm:animations+}
    $\frac 1 2 \Minors^+ _k \Gamma = \Bigl\langle \mon\alpha :
    \alpha\in\Oddperiods_k(\Gamma) \Bigr\rangle$.
    More precisely, the nonzero minors of $\sC^+_\Gamma$ are precisely of the
    form $\pm 2^{\odlen(\alpha)}\mon\alpha$ for $\alpha\in\Oddperiods(\Gamma)$.
  \end{enumerate}
\end{thm}

\begin{cor}
  \label{cor:rank_C_animations}
  Let $\Gamma = (V,E)$ be a graph.
  Then
  \begin{align*}
    \rank_{\QQ(X_V)}(\sC^-_\Gamma) & =
    \max(k \ge 0 : \Fixed_k(\Gamma) \not= \emptyset) =
    \max(\deg(\alpha) : \alpha \in \Fixed(\Gamma))
                                     \text{ and}\\
    \rank_{\QQ(X_V)}(\sC^+_\Gamma) & =
    \max(k \ge 0 : \Oddperiods_k(\Gamma) \not= \emptyset). =
    \max(\deg(\alpha) : \alpha \in \Oddperiods(\Gamma)).
    \tag*{\qed}
  \end{align*}
\end{cor}

For group-theoretic applications, the ask zeta functions $W^-_\Gamma(X,T)$
(see Theorem~\ref{thm:uniformity}) are of particular interest when $\Gamma$ is
loopless.

\begin{cor}
  \label{cor:loopless_minus_minors}
  Let $\Gamma$ be loopless.
  Then $\Minors^-_k \Gamma = \Bigl\langle \mon\alpha :
  \alpha\in\Nilpotents_k(\Gamma) \Bigr\rangle$.
  More precisely, the nonzero minors of $\sC^-_\Gamma$ are precisely of the
  form $\pm \mon\alpha$ for $\alpha\in\Nilpotents(\Gamma)$.
\end{cor}
\begin{proof}
  If $\Gamma$ is loopless, then $\Nilpotents(\Gamma) = \Fixed(\Gamma)$.
  Now apply Theorem~\ref{thm:animations}\ref{thm:animations-}.
\end{proof}

In \S\ref{ss:animations_yield_minors} we prove that (suitable)
animations give rise to minors of $\sC_\Gamma^\pm$, proving ``half''
of Theorem~\ref{thm:animations}.  In
Sections~\ref{ss:minors_connected}--\ref{ss:minors_yield_animations}
we establish the other half by showing that all nonzero minors of
$\sC_\Gamma^\pm$ arise from suitable animations.

\subsection{Animations yield minors}
\label{ss:animations_yield_minors}

We begin by proving half of Theorem~\ref{thm:animations}:
we show that if $\alpha \in \Fixed(\Gamma)$
(resp.\ $\alpha \in \Oddperiods(\Gamma)$), then
$\pm 2^{\odlen(\alpha)} \mon\alpha$
is a minor of $\sC^-_\Gamma$ (resp.\ $\sC^+_\Gamma$).

\paragraph{Notation for submatrices and minors of $\sC^\pm_\Gamma$.}

We use the following notation (which builds upon on that
from \S\ref{ss:hypergraph_decompositions}) throughout this entire
section.  Recall that the rows and columns of $\sC^\pm_\Gamma$ are
indexed by edges and vertices of $\Gamma$, respectively.  As we are
only interested in minors up to signs, we are free to order vertices
and edges as we see fit.  If $v_1,\dotsc,v_n$ are the distinct
vertices of $\Gamma$, then, up to multiplication by $\pm 1$, each
nonloop edge $\{v_i,v_j\}$ of $\Gamma$ gives rise to a row
$X_{v_i} \std j \pm X_{v_j} \std i \in
\ZZ[X_V]^n$ of $\sC^\pm_\Gamma$,
while each loop $\{v_i\}$ gives rise to a row $X_{v_i} \std i$.
Having ordered $V$, we order $E$ (e.g.\ lexicographically as in
\S\ref{ss:ask_graph}) to define $\sC^\pm_\Gamma$.
Given subsets $V'\subset V$ and $E'\subset E$, we then obtain a
submatrix $\sC^\pm_\Gamma[V'\mid E']$ of $\sC^\pm_\Gamma$ obtained by
selecting the columns indexed using the elements of $V'$ (in their chosen
order) and the rows indexed using the elements of $E'$ (again in their chosen
order).
If $\card{V'} = \card{E'} = k$, then we write $m^\pm_\Gamma[V'\mid E'] =
\det(\sC^\pm_\Gamma[V'\mid E'])$ for the corresponding $k\times k$ minor of
$\sC^\pm_\Gamma$.

\paragraph{Partitioning vertices using an animation.}

Let $\alpha \in \Animations(\Gamma)$.  We describe a canonical
partition of the vertices of $\Gamma$ which will play a crucial role
in our subsequent construction of a minor of $\sC^\pm_\Gamma$ related
to $\mon\alpha$.

As in \S\ref{ss:nilpotency}, let $V^\per \subset V$ (resp.\ $V^\tra\subset V$)
be the set of $\alpha$-periodic (resp.\ $\alpha$-transient) points in $V$.
Let $V^\per_\udef = \bigsqcup_{\lambda \in \Lambda} \Omega_\lambda$ be the
decomposition into orbits of the infinite cyclic group
acting on $V^\per_\udef$ via $\alpha$.
We assume that $0\in \Lambda$ and $\Omega_0 = \{\udef\}$.
We write $\Lambda^+ = \Lambda\setminus\{0\}$ and
$N_\lambda = \card{\Omega_\lambda}$.
By definition, $\alpha\in \Oddperiods(\Gamma)$ if and only if each
$N_\lambda$ is odd, in which case $\odlen(\alpha) = \#\{\lambda
\in \Lambda: N_\lambda > 1\}$.
Furthermore, $\alpha\in\Fixed(\Gamma)$ if and only if $N_\lambda =
1$ for all $\lambda \in \Lambda$.
Finally, $\alpha \in \Nilpotents(\Gamma)$ if and only if $\Lambda = \{0\}$.

By repeated application of $\alpha$, every transient point $v\in
V^\tra$ is moved into precisely one of the orbits,
$\Omega_{\omega(v)}$ say.  We let $\delta(v)$ be the least positive
integer with $v^{\alpha^{\delta(v)}} \in \Omega_{\omega(v)}$.  Define
$\Omega_\lambda^{(0)} = \Omega_\lambda$ and $\Omega^{(k+1)}_\lambda
= \{ v\in V^\tra : v^\alpha \in \Omega^{(k)}_\lambda\}$.
Equivalently, for $k\ge 1$, we have $\Omega^{(k)}_\lambda = \{ v\in
V^\tra :\omega(v) = \lambda \text{ and } \delta(v) = k \}$.  Given
$\lambda\in\Lambda$, we let $\kappa(\lambda)$ be the largest $k\ge 0$
with $\Omega^{(k)}_\lambda\not= \emptyset$.  We thus obtain a natural
partition
\begin{equation}
  \label{eq:block_decomposition}
  V =
  \bigsqcup_{\lambda \in \Lambda}
  \bigsqcup_{k=0}^{\kappa(\lambda)} \Omega^{(k)}_\lambda.
\end{equation}

\paragraph{Cycle graphs and $\sC^+_\Gamma$.}

Let $\Cycle_k$ denote the cycle graph on $\{1,\dotsc,k\}$.  Various
graph-theoretic properties and invariants (e.g.\ chromatic numbers) of
$\Cycle_k$ depend on the parity of $k$.  Here, we encounter another
instance of this phenomenon: the following elementary lemma explains
the curious role of \itemph{odd}-periodic animations
in~Theorem~\ref{thm:animations}.

\begin{lemma}
  \label{lem:cycle}
  Let $k\ge 3$.
  Then
  \[
    \det(\sC^+_{\Cycle_k})
    = \begin{cases}
      \pm 2 X_1\dotsb X_k, &\text{if $k$ is odd}, \\
      0, & \text{if $k$ is even.}
    \end{cases}
  \]
\end{lemma}
\begin{proof}
  We order the vertices $1,\dotsc,k$ and the edges as
  $\{1,2\}, \{2,3\}, \dotsc, \{k-1,k\}, \{1,k\}$.
  With respect to these orders, we then have
  \[
    \sC^+_\Gamma =
    \left[
    \begin{array}{cccccc}
      X_2 & X_1 \\
      & X_3 & X_2 \\
      && \ddots & \ddots \\
      & & &X_{k} & X_{k-1}\\
      X_k & & &&X_1
    \end{array}
    \right].
  \]
  Laplace expansion along the first column yields
  $\det(\sC^+_\Gamma) = (1 + (-1)^{k+1}) X_1\dotsb X_k$.
\end{proof}

\paragraph{Order.}

For each $\lambda \in \Lambda$, we may write
$\Omega^{(0)}_\lambda = \{ w(\lambda,1), \dotsc, w(\lambda, N_\lambda) \}$
where $w(\lambda,i)^\alpha = w(\lambda,i+1)$ for $i < N_\lambda$ and
$w(\lambda,N_\lambda)^\alpha = w(\lambda, 1)$.
We let $\lte$ be the total order on $\Omega^{(0)}_\lambda$ with
$w(\lambda,1) \lt \dotsb \lt w(\lambda,N_\lambda)$.
For $k = 1,\dotsc,\kappa(\lambda)$, we choose an
arbitrary total order $\lte$ on~$\Omega^{(k)}_\lambda$.
Finally, we choose a total order $\lte$ on $\Lambda$.
By \eqref{eq:block_decomposition}, each point in $V$ corresponds uniquely to a
triple $(\lambda,k,u)$, where $\lambda \in \Lambda$,
$k\in\{0,\dotsc,\kappa(\lambda)\}$, and $u\in \Omega^{(k)}_\lambda$.
By ordering these triples lexicographically, we obtain a total order $\lte$ on
$V$.

\paragraph{Building a minor from an animation.}

Suppose that $\alpha\in\Oddperiods(\Gamma)$ (resp.\
$\alpha \in \Fixed(\Gamma)$).  Define $V' = \Dom(\alpha)$ and $E'
= \{ \{ v, v^\alpha\} : v\in V'\}$.  We show that $\card{V'}
= \card{E'}$ and that the minor $m^+_\Gamma[V'\mid E']$ (resp.\
$m^-_\Gamma[V'\mid E']$) is given by $\pm
2^{\odlen(\alpha)} \mon\alpha$.

The map $V'\xto\Phi E'$ which sends $v$ to $\{v,v^\alpha\}$ is onto by
construction.
Suppose that $\{ u, u^\alpha\} = \{ v, v^\alpha\}$ for $u,v\in\Dom(\alpha)$.
If $u \not= v$, then $u^\alpha = v$ and $v^\alpha = u$ whence
$\{u,v\}$ is an orbit of even length, which contradicts
$\alpha\in \Oddperiods(\Gamma)$. 
Hence, $\Phi$ is bijective and $\card{V'} = \card{E'}$.
By transport of structure via $\Phi$, our existing total order on $V'$ induces
a total order on $E'$ which we arbitrarily extend to a total order on all of
$E$.

Note that $V\setminus V' = \Omega^{(1)}_0 = \{ v\in V^\tra : v^\alpha =
\udef\}$.
We may write $V' = \bigsqcup_{\lambda\in\Lambda} V'_\lambda$,
where
\[
  V'_{\lambda} = \begin{cases}
    \bigsqcup_{k=2}^{\kappa(\lambda)} \Omega^{(k)}_\lambda, & \text{if } \lambda
    = 0,\\
    \bigsqcup_{k=0}^{\kappa(\lambda)} \Omega^{(k)}_\lambda, & \text{otherwise;}
  \end{cases}
\]
note that $V'_0$ might be empty but $V'_\lambda \not= \emptyset$ for $\lambda\in\Lambda^+$.
(The case $\Lambda^+ = \emptyset$ is possible.)
Let $E'_\lambda$ denote the image of $V'_\lambda$ under the bijection $\Phi$;
hence, $E' = \bigsqcup_{\lambda\in\Lambda} E'_\lambda$.
For each $\lambda\in \Lambda$, the endpoints of every edge in
$E'_\lambda$ belong to $V'_\lambda$ except when $\lambda = 0$ and the edge
is of the form $\{ v, v^\alpha\}$ for $v\in \Omega^{(2)}_0$;
in the latter case, $v^\alpha \in \Omega^{(1)}_0$ and thus $v^\alpha
\notin V'$.
We thus see that $\sC^\pm_\Gamma[V'\mid E']$ is lower block triangular with
diagonal blocks given by the $\sC^\pm_\Gamma[V'_\lambda \mid E'_\lambda]$.
In particular, $m^\pm_\Gamma[V'\mid E'] = \prod_{\lambda\in\Lambda}
m^\pm_\Gamma[V'_\lambda \mid E'_\lambda]$.
We will further elucidate the block triangular structure in the following.

We let $D^{(i)}_\lambda$ denote the
$\card{\Omega^{(i)}_\lambda} \times \card{\Omega^{(i)}_\lambda}$ diagonal
matrix whose diagonal entries are given by the $X_{v^\alpha}$ as $v$ ranges
over $\Omega^{(i)}_\lambda$.

\paragraph{Nilpotent points ($\lambda = 0$).}
Assuming that $V'_0\not= \emptyset$, let us consider the case $\lambda = 0$.
By applying $\Phi$, the decomposition $V'_0 = \bigsqcup_{k=2}^{\kappa(0)}
\Omega^{(k)}_0$
yields a corresponding decomposition of $E'_0$.
As we observed before, if $v\in \Omega^{(2)}_0$, then $\udef \not=
v^\alpha\notin V'$.
Viewing the rows of $\sC^\pm_\Gamma[V'_0\mid E'_0]$ as elements of
$\ZZ[X_{V'_0}] E'_0$ and up to signs, the row corresponding to the edge
$\{v,v^\alpha\}$ associated with $v\in \Omega^{(2)}_0$ is therefore simply
$X_{v^\alpha} \std v$.
Next, if $v\in \Omega^{(k)}_0$ for $k\ge 3$, then $v^\alpha \lt v$ and the row corresponding
to the edge $\{v, v^\alpha\}$ is $X_{v^\alpha} \std{v} \pm X_v \std
{v^\alpha}$ (up to the sign).
The submatrix $\sC^\pm_\Gamma[V'_0\mid E'_0]$ is therefore of the form
\[
  \left[
    \begin{array}{c|c|c}
      D^{(2)}_0 &&\\\hline
      * & D^{(3)}_0 & \\\hline
          & \ddots & \ddots
    \end{array}
  \right].
\]
Hence, $m^\pm_\Gamma[V'_0\mid E'_0] = \prod\limits_{v\in V'_0} X_{v^\alpha} =
\mon{\alpha \restriction {V'_0}}$.

\paragraph{Fixed points.}
Next, let $\lambda \in \Lambda^+$ with $N_\lambda = 1$.
In other words, $\Omega_\lambda = \{ u_\lambda\}$ consists of a fixed point
$u_\lambda$ of $\alpha$ distinct from $\udef$.
Arguing analogously to before, here we find that
the submatrix $\sC^\pm_\Gamma[V'_\lambda\mid E'_\lambda]$ is of the form
\[
  \left[
    \begin{array}{c|c|c|c}
      X_{u_\lambda} &&&\\\hline
      * & D^{(1)}_\lambda &&\\\hline
      &* & D^{(2)}_\lambda & \\\hline
      &   & \ddots & \ddots
    \end{array}
  \right].
\]
The first row corresponds to the loop $\{u_\lambda\} = u_\lambda^\Phi$.
Hence, we again find that $m^\pm_\Gamma[V'_\lambda\mid E'_\lambda] =
\prod\limits_{v\in V'_\lambda} X_{v^\alpha} = \mon{\alpha \restriction
  {V'_\lambda}}$.

\paragraph{General periodic points, $+$-case.}

Let $\lambda \in \Lambda^+$ with $N_\lambda > 1$.
As $\alpha \in \Oddperiods(\Gamma)$, we know that $N_\lambda \ge 3$ is odd.
In this case, we only need to consider the ``positive minor''
$m^+_\Gamma[V'_\lambda \mid E'_\lambda]$.
(This is because we assume that $\alpha \in \Nilpotents(\Gamma)$ in the
$-$-case.
We note that Corollary~\ref{cor:no_minus_cycles} below will imply that
$m^-_\Gamma[V'_\lambda \mid E'_\lambda] = 0$ if $N_\lambda > 1$.)

Recall that $\Omega^{(0)}_\lambda = \{ w(1,\lambda),\dotsc, w(N_\lambda,\lambda)\}$.
We abbreviate $w_i = w(i,\lambda)$ and $N = N_\lambda$.
Proceeding as above, we now find that
$\sC^+_\Gamma[V'_\lambda\mid E'_\lambda]$ is of the form
\[
  \left[
    \begin{array}{c|c|c|c}
      C &&&\\\hline
      * & D^{(1)}_\lambda &&\\\hline
      &* & D^{(2)}_\lambda & \\\hline
      &   & \ddots & \ddots
    \end{array}
  \right],
\]
where $C$ is the matrix
\[
  C = \begin{bmatrix}
    X_{w_2} & X_{w_1} \\
    & X_{w_3} & X_{w_2} \\
    & & \ddots & \ddots \\
    &&& X_{w_N} & X_{w_{N-1}} \\
    X_{w_N} &&&& X_{w_1}
    \end{bmatrix}.
  \]
By Lemma~\ref{lem:cycle}, we have $\det(C) = \pm 2 X_{w_1}\dotsb X_{w_k}$
whence $m^\pm_\Gamma[V'_\lambda\mid E'_\lambda] =
\pm 2 \prod\limits_{v\in V'_\lambda} X_{v^\alpha}
= \mon{\alpha \restriction {V'_\lambda}}
$.

\paragraph{Conclusion.}

In summary, we have shown the following.

\begin{prop}
  \label{prop:animation_to_minor}
  \quad
  \begin{enumerate}
  \item
    If $\alpha \in \Fixed(\Gamma)$ (in which case $N_\lambda = 1$ for all
    $\lambda \in \Lambda$), then $m^-_\Gamma[V'\mid E'] = \pm \mon\alpha$.
  \item
    If $\alpha \in \Oddperiods(\Gamma)$ (in which case each $N_\lambda$ is
    odd), then
    $m^+_\Gamma[V'\mid E'] = \pm 2^{\odlen(\alpha)} \mon\alpha$.
    \qed
  \end{enumerate}
\end{prop}

It remains to show that, conversely, every minor
$m^\pm_\Gamma[V'\mid E']$ given by $V'\subset V$ and $E'\subset E$ with
$\card{V'} = \card{E'}$ is either zero, or given by
$\pm 2^{\odlen(\alpha)} \mon\alpha$ for a suitable animation, according to the
two cases in Theorem~\ref{thm:animations}.

\subsection{Minors yield animations: the connected and nondegenerate case}
\label{ss:minors_connected}

Let $E'\subset E$ and $V'\subset V$ with $\card{V'} = \card{E'}$.
Suppose that $\Eta' = \Eta[V' \mid E']$ is connected and nondegenerate
so that Lemma~\ref{lem:unicyclic} is applicable.
We show that the minor $m^\pm_\Gamma[V'\mid E']$ arises from a
judiciously chosen animation of $\Gamma$.
In showing this, we distinguish two cases reflecting the two conditions
\ref{lem:unicyclic1} and \ref{lem:unicyclic2} in Lemma~\ref{lem:unicyclic}.

\subsubsection[Case (U2): a decorated tree]{Case \ref{lem:unicyclic2}: a
  decorated tree}
\label{sss:decorated_tree}

\paragraph{Animations from rooted trees.}

Let $\Tau$ a tree on the vertex set $V$ with chosen root $r\in V$.
Let $\dist_\Tau(u,v)$ denote the distance between the vertices $u$ and $v$ in
$\Tau$, i.e.\ the length of the unique path from $u$ to $v$ in $\Tau$.
By a \emph{distance order} of $\Tau$ with respect to $r$, we mean a total
order $\lte$ on $V$ such that whenever $\dist_\Tau(r,u)$ < $\dist_\Tau(r,v)$
for $u,v\in V$, then we have $u \lt v$.
Let $V \xparto{\pred(\Tau,r)} V$ be the partial function defined on
$V\setminus\{r\}$ which sends each $v \in V\setminus\{r\}$ to its predecessor
on the unique path from $r$ to $v$.
Clearly, $\pred(\Tau,r)$ is a nilpotent animation of $\Tau$.
By construction, we have $\dist_\Tau(v^{\pred(\Tau,r)},r) < \dist_\Tau(v,r)$
for all $v\in V'\setminus\{r\}$.
By minor abuse of notation, we regard animations of subgraphs of $\Gamma$,
such as $\pred(\Tau,r)$ when $\Tau$ is a spanning tree of $\Gamma$, as
animations of $\Gamma$.

\paragraph{}

Suppose that we are in case \ref{lem:unicyclic2} of
Lemma~\ref{lem:unicyclic}.
Let $e_\circ \in E'$ be the unique loop of $\Eta'$.
By Lemma~\ref{lem:unicyclic}, $\Tau =
(V',E\setminus\{e_\circ\})$ is a subtree of $\Gamma$.
While $e_\circ$ is a loop of $\Eta'$, it may or may not be a loop of $\Eta$
(equivalently: $\Gamma$).
We consider these two cases in turn.

\paragraph{Subcase: $e_\circ$ is a nonloop of $\Eta$.}
In this case, $e_\circ$ has two distinct endpoints in $\Eta$ but only one of
these belongs to $V'$.
That is, $e_\circ = \{u, r\}$ where $u\in V'$ and $r\in V\setminus V'$.
As $\Tau$ is a tree so is then the subgraph $\Tau' = (V'\cup\{r\}, E')$ of
$\Gamma$.
The nilpotent animation $\pred(\Tau',r)$ of $\Tau'$ and $\Gamma$ turns out to
give rise to the minor $m^\pm_\Gamma[V'\mid E']$:

\begin{lemma}
  \label{lem:nonloop_plumbing}
  $m^\pm_\Gamma[V'\mid E'] = \pm \mon{\pred(\Tau',r)}$.
\end{lemma}
\begin{proof}
  Let $\lte$ be a distance order of $\Tau'$ with respect to $r$.
  We arbitrarily extend $\lte$ to a total order on $V$ (denoted using the same symbol).
  Let $V'_i$ denote the subset of $V'$ consisting of vertices of distance $i$
  from $r$ in $\Tau'$.
  Given $\lte$, we order $E$ lexicographically as in \S\ref{ss:ask_graph}.
  Write $\alpha = \pred(\Tau',r)$.
  With respect to these orders, up to changing the signs of rows,
  the submatrix $\sC^\pm_\Gamma[V'\cup\{r\} \mid E']$ is of the form
  \[
    \left[
      \begin{array}{c|c|c|c}
        * & D_1 &&\\\hline
          & * & D_2 & \\\hline
          & & \ddots & \ddots
      \end{array}
    \right],
  \]
  where the column groups consist of $1, \card{V'_1},\card{V'_2}, \dotsc$
  columns, respectively, and $D_i$ is the diagonal matrix with diagonal
  entries given by the $X_{v^\alpha}$ as $v$ ranges over $V'_i$ (in the given
  order).
  That is, if $V'_i$ consists of $v_{i1} \lt v_{i2} \lt \dotsb$,
  then $D_i = \diag(X_{v_{i1}^\alpha},X_{v_{i2}^\alpha}, \dotsc)$.
  The claim follows by deleting the first column and taking the determinant
  to obtain~$m^\pm_\Gamma[V'\mid E']$.
\end{proof}

\paragraph{Subcase: $e_\circ$ is a loop of $\Eta$.}
In this case, $e_\circ = \{ r\}$ for $r \in V'$.
Using our notation from~\S\ref{s:selectors}, we may regard
$\pred(\Tau,r)[r\subs r]$, the partial function which agrees with
$\pred(\Tau,r)$ except that it sends $r$ to itself, as an element of
$\Fixed(\Gamma)$.

\begin{lemma}
  \label{lem:loop_plumbing}
  $m^\pm_\Gamma[V'\mid E'] = \pm \mon{\pred(\Tau,r)[r\subs r]}$.
\end{lemma}
\begin{proof}
  Similarly to the proof of Lemma~\ref{lem:nonloop_plumbing}, let $\lte$ be
  a distance order of $\Tau$ with respect to $r$ and extend this to a total
  order on $V$.
  We again order edges lexicographically.

  Write $\alpha = \pred(\Tau,r)$.
  Let $V'_i$ denote the set of all vertices in $V'$ of distance $i$ from $r$.
  Let $D_i$ be the diagonal matrix with entries $X_{v^\alpha}$ as $v$ ranges over
  $V'_i$ (in the given order).
  Then the submatrix
  $\sC^\pm_\Gamma[V'\mid E']$ is of the form
    \[
    \left[
      \begin{array}{c|c|c|c}
        X_r & & & \\\hline
        * & D_1 &&\\\hline
          & * & D_2 & \\\hline
          & & \ddots & \ddots
      \end{array}
    \right],
  \]
  where the first row corresponds to $e_\circ$.
  We thus obtain
  \[
    m^\pm_\Gamma[V' \mid E'] = \pm X_r \, \det(D_1) \det(D_2) \dotsb = \pm X_r
    \mon\alpha = \pm \mon{\alpha[r\gets r]}.
    \pushQED{\qed}
    \qedhere
  \]
\end{proof}

\subsubsection[Case (U1): a unicyclic graph]{Case \ref{lem:unicyclic1}: a unicyclic graph}
\label{sss:unicyclic}

Suppose that we are in case \ref{lem:unicyclic1} of
Lemma~\ref{lem:unicyclic}.
Then $\Gamma = (V',E')$ is a unicyclic loopless graph.
The matrices $\sC^+_\Gamma$ and
$\sC^-_\Gamma$ behave quite differently in this case.

\begin{lemma}
  If $\Lambda$ is a loopless graph with $k$ vertices and $k$ edges,
  then $\det(\sC^-_\Lambda) = 0$.
\end{lemma}
\begin{proof}
  We may assume that $1,\dotsc,k$ are the vertices of $\Lambda$.
  As $\Lambda$ is loopless, every row of $\sC^-_\Lambda$ is of the form
  $X_i \std j - X_j \std i$.
  In particular, $\sC^-_\Lambda [X_1,\dotsc,X_n]^\top = 0$.
\end{proof}

\begin{cor}
  \label{cor:no_minus_cycles}
  In case \ref{lem:unicyclic1} of Lemma~\ref{lem:unicyclic}, we have
  $m^-_\Gamma[V'\mid E'] = 0$.
  \qed
\end{cor}

Let now $\Gamma' = (V',E')$ be a unicyclic loopless subgraph of $\Gamma$.
Clearly, this forces $\card{V'} \ge 3$.
Let $u_1,\dotsc,u_r$ with $r \ge 3$ be distinct vertices in $V'$ with
$u_1 \sim u_2 \sim \dotsb \sim u_r \sim u_1$ in $\Gamma'$.
Let $U = \{ u_1,\dotsc,u_r \}$.
That is, $U$ is the vertex set of the unique cycle within $\Gamma'$.
Let $\Tau = \Gamma'/U$ be the tree obtained from $\Gamma'$ by contracting all
the $u_i$ to a single vertex $U$ (and deleting all edges on the cycle formed
by $U$).

We order the vertices in $V'\setminus U$ by means of a distance order of
$\Tau$ with respect to the root $U$.
We further order the elements of $U$ as written above, with each of them
preceding each element of $V'\setminus U$.
Next, we extend our order to a total order on $V$ and we order $E$
lexicographically.

We now construct an animation $V \xto\alpha V$ defined on $V'$ as follows.
As $\Gamma'$ is unicyclic, for each $v\in V'\setminus U$, there exists a
unique shortest path from \itemph{some} vertex in $U$ to $v$.
This path can be obtained from the unique path from $U$ to $v$ in $\Tau$.
For $v\in V'\setminus U$, we define $v^\alpha$ to be the predecessor of $v$ on
the aforementioned unique path to $v$.
Next, we define $u_i^\alpha = u_{i+1}$ for $i < r$ and $u_r^\alpha = u_1$.
Clearly, $\alpha$ is an animation of $\Gamma$.
Intuitively, $\alpha$ cyclically permutes the vertices in $U$ (having arbitrarily
oriented the cycle above) and it moves
vertices in $V'\setminus U$ towards $U$.

Let $V'_i$ denote the set of all vertices in $V'$ of distance $i$ from $U$.
Let $D_i$ be the diagonal matrix with entries $X_{v^\alpha}$ as $v$ ranges over
$V'_i$ (in the given order).
We then have
\[
  \sC^+_\Gamma[V'\mid E'] =
  \left[
    \begin{array}{c|c|c|c}
      C & & & \\\hline
      * & D_1 &&\\\hline
        & * & D_2 & \\\hline
        & & \ddots & \ddots
    \end{array}
  \right],
\]
where $C$ is equivalent (up to relabelling of variables)
to $\sC^+_{\Cycle_r}$.
By construction and Lemma~\ref{lem:cycle}, we obtain the following.

\begin{lemma}
  \label{lem:oddcycle}
  \quad
  \begin{enumerate}
  \item
    If the length $r$ of the unique cycle in $\Gamma'$ is even, then
    $m^+_\Gamma[V'\mid E'] = 0$.
  \item
    If $r$ is odd, then
    $m^+_\Gamma[V'\mid E'] = \pm 2 \mon{\alpha}$ and $\alpha \in
    \Oddperiods(\Gamma)$.
    \qed
  \end{enumerate}
\end{lemma}

\subsection{Minors yield animations: the general case}
\label{ss:minors_yield_animations}

Let $\Gamma = (V,E)$ be a graph and let $\Eta = \Inc(\Gamma)$ be its
incidence hypergraph.  Let $E'\subset E$ and $V'\subset V$ with
$\card{V'} = \card{E'} = k$.  We are now ready to complete the proof
of the remaining half of Theorem~\ref{thm:animations}.

\begin{prop}
  \label{prop:minor_to_animation}
  \quad
  \begin{enumerate}
  \item
    If $m^-_\Gamma[V'\mid E'] \not= 0$, then there exists $\alpha \in
    \Fixed(\Gamma)$ with $m^-_\Gamma[V'\mid E'] = \pm \mon\alpha$.
  \item
    If $m^+_\Gamma[V'\mid E'] \not= 0$, then there exists $\alpha \in
    \Oddperiods(\Gamma)$ with $m^+_\Gamma[V'\!\mid\! E'] = \pm 2^{\odlen(\alpha)} \mon\alpha$.
  \end{enumerate}
\end{prop}
\begin{proof}
  Suppose that $m^\pm_\Gamma[V'\mid E'] \not= 0$.
  By Lemma~\ref{lem:degenerate}, we may assume that $\Eta' := \Eta[V'\mid E']$
  is nondegenerate.
  Next, we decompose $\Eta' = \Eta'_1\oplus \dotsb \oplus \Eta'_c$ into
  connected components as in \S\ref{ss:hypergraph_decompositions}.

  Writing $\Eta'_j = \Eta'[V'_j\mid E'_j]$,
  Lemma~\ref{lem:minor_factorisation} shows that each $\Eta'_j$ is square.
  Our results in \S\ref{ss:minors_connected} show that for $j=1,\dotsc,c$, we
  may construct an explicit animation $\alpha_j \in \Fixed(\Gamma)$
  (resp.\ $\alpha\in\Oddperiods(\Gamma)$) with $m^-_\Gamma[V'_j\mid E'_j] =
  \pm \mon{\alpha_j}$
  (resp.\ $m^+_\Gamma[V'_j\mid E'_j] = \pm 2^{\odlen(\alpha_j)}
  \mon{\alpha_j}$).
  (We note that in the ``$+$-case'', we have $\odlen(\alpha_j)\in \{0,1\}$.
  Namely, $\odlen(\alpha_j) = 0$ if $\alpha_j$ is obtained via
  Lemma~\ref{lem:nonloop_plumbing} or Lemma~\ref{lem:loop_plumbing} and
  $\odlen(\alpha_j) = 1$ if $\alpha_j$ arises via
  Lemma~\ref{lem:oddcycle}.

  Using our specific constructions of the $\alpha_j$, the domain of definition
  $\Dom(\alpha_j)$ is contained in $V'_j$.
  We may thus define a partial function $V'\xto\alpha V'$ which agrees with
  $\alpha_j$ on $V'_j$ for $j=1,\dotsc,c$.
  Clearly, $\alpha \in \Fixed(\Gamma)$ (resp.\ $\alpha \in\Oddperiods(\Gamma)$)
  if all $\alpha_j$ satisfy $\alpha_j \in\Fixed(\Gamma)$
  (resp.\ $\alpha_j\in\Oddperiods(\Gamma)$).
  Moreover, $\odlen(\alpha) = \odlen(\alpha_1) + \dotsb + \odlen(\alpha_c)$.
  The claim thus follows since
  \[
    m^\pm_\Gamma[V'\mid E'] = \pm \prod_{j=1}^c m^{\pm}_{\Gamma}[V'_j\mid E'_j]
    = \pm \prod_{j=1}^c 2^{\odlen(\alpha_j)} \mon{\alpha_j} = \pm
    2^{\odlen(\alpha)}\mon\alpha.
    \pushQED{\qed}
    \qedhere
  \]
\end{proof}

\begin{proof}[Proof of Theorem~\ref{thm:animations}]
  Combine Proposition~\ref{prop:animation_to_minor}
  and Proposition~\ref{prop:minor_to_animation}.
\end{proof}

\begin{proof}[New proof of Theorem~\ref{thm:uniformity}\ref{thm:uniformity2}--\ref{thm:uniformity3}]
  Combine Proposition~\ref{prop:general_monomial_integral},
  Prop\-o\-si\-tion~\ref{prop:ask_Gamma_integral}, and
  Theorem~\ref{thm:animations}.  For
  Theorem~\ref{thm:uniformity}\ref{thm:uniformity2}, the base ring is
  $R = \ZZ[1/2]$; for
  Theorem~\ref{thm:uniformity}\ref{thm:uniformity2}, it is $R = \ZZ$.
\end{proof}

\section{The Reflexive Graph Modelling Theorem}
\label{s:RGMT}

In this section,
we prove Theorem~\ref{thm:rgmt} by showing that, assuming that $\Gamma$
is reflexive, the matrices $\sC^\pm_\Gamma$ and $\sC_{\Adj(\Gamma)}$ have the
same rank (over $\QQ(X_V)$) and the same ideals of minors (in $\ZZ[\frac 1
2][X_V]$ in the $+$-case and in $\ZZ[X_V]$ otherwise).
All of this relies on our parameterisation of minors in terms of selectors
and animations developed in \S\ref{s:selectors} and \S\ref{s:animations}.

\subsection[The ranks of three matrices]{The ranks of $\sC_\Eta$ and $\sC_\Gamma^\pm$}
\label{ss:ranks}

\begin{prop}
  \label{prop:rank_C_Eta}
  Let $\Eta = (V,E,\increl)$ be a hypergraph.
  Then
  \[
    \rank_{\QQ(X_V)}(\sC_\Eta)= \#\{ e\in E : \norm{e}_\Eta \not=
    \emptyset\}.
  \]
\end{prop}
\begin{proof}
  By Corollary~\ref{cor:rank_C_Eta_selectors},
  $\rank_{\QQ(X_V)}(\sC_\Eta) = \max(\deg(\phi) : \phi \in \Selectors(\Eta))$.
  Clearly, the degree of every selector of $\Eta$ is at most $\card{E} -
  \card{\empties(\Eta)}$ and this bound is attained.
\end{proof}

\begin{prop}
  \label{prop:rank_C-}
  Let $\Gamma = (V,E)$ be a graph with $n$ vertices.
  Let $d$ be the number of connected components of $\Gamma$ that do \underline{not}
  contain a loop.
  Then $\Fixed_{n-d}(\Gamma) \not= \emptyset$ but $\Fixed_{n-d+1}(\Gamma) = \emptyset$.
  In particular,
  $\rank_{\QQ(X_V)}(\sC^-_\Gamma) = n - d$.
\end{prop}
\begin{proof}
  This easily reduces to the case that $\Gamma$ is connected, which we
  now assume.
  Our arguments in \S\ref{sss:decorated_tree} show that
  $\Fixed_{n-1}(\Gamma)\not= \emptyset$.
  Indeed, every choice of a spanning tree $\Tau$ of $\Gamma$ and root $r\in V$
  gives rise to a (nilpotent) animation $\pred(\Tau,r)$ of $\Gamma$.
  The domain of such an animation is $V\setminus\{r\}$ whence
  $\emptyset \not= \Nilpotents_{n-1}(\Gamma) \subset \Fixed_{n-1}(\Gamma)$.
  As every nilpotent animation of $\Gamma$ is necessarily undefined
  somewhere, we have $\Nilpotents_n(\Gamma) = \emptyset$.
  Hence, if $\Gamma$ is loopless (equivalently: $d = 1$), then
  $\Fixed_n(\Gamma) = \Nilpotents_n(\Gamma) = \emptyset$.
  On the other hand, if $r \in V$ with $r\sim r$, then $d = 0$ and
  $\pred(\Tau,r)[r \subs r] \in \Fixed_n(\Gamma)$,
  where $\Tau$ is a spanning tree of $\Gamma$ as above.
  Of course, since the domain of any animation of $\Gamma$ is a subset of $V$,
  we have $\Fixed_{n+1}(\Gamma) = \emptyset$.
  The final claim follows from Corollary~\ref{cor:rank_C_animations}.
\end{proof}

\begin{cor}
  \label{cor:rank_C-_loopless}
  Let $\Gamma$ be a loopless graph with $n$ vertices and $c$ connected components.
  Then $\Nilpotents_{n-c}(\Gamma) \not= \emptyset$ but $\Nilpotents_{n-d+1}(\Gamma) = \emptyset$.
  In particular,
  $\rank_{\QQ(X_V)}(\sC^-_\Gamma) = n - c$.
  \qed
\end{cor}

\begin{rem}
  The final part of Corollary~\ref{cor:rank_C-_loopless}
  can also be deduced from \cite[Lem.\ 3.2]{csp}.
\end{rem}

\begin{prop}
  \label{prop:rank_C+}
  Let $\Gamma = (V,E)$ be a graph with $n$ vertices.
  Let $d$ be the number of connected components of $\Gamma$ that do
  \underline{not} contain an odd cycle;
  here, loops are counted amongst odd cycles.
  Then $\Oddperiods_{n-d}(\Gamma) \not= \emptyset$ but
  $\Oddperiods_{n-d+1}(\Gamma) = \emptyset$.
  In particular, $\rank_{\QQ(X_V)}(\sC^+_\Gamma) = n - d$.
\end{prop}
\begin{proof}
  As in the proof of Proposition~\ref{prop:rank_C-}, we may assume that
  $\Gamma$ is connected and we find that
  $\emptyset \not= \Nilpotents_{n-1}(\Gamma) \subset
  \Oddperiods_{n-1}(\Gamma)$ and $\Nilpotents_n(\Gamma) = \emptyset$.
  If $\Gamma$ does not contain any cycles of odd length, then
  $d = 1$ and
  $\Oddperiods_n(\Gamma) = \Nilpotents_n(\Gamma) = \emptyset$.
  If $\Gamma$ contains a loop, then
  we can show that $\Fixed_n(\Gamma) \not= \emptyset$ as in the proof of
  Proposition~\ref{prop:rank_C-}.
  Finally, if $\Gamma$ contains a cycle of odd length which is not a loop,
  then we can construct
  $\alpha\in\Oddperiods_n(\Gamma)$ using the procedure from \S\ref{sss:unicyclic}.
  As before, $\Oddperiods_{n+1}(\Gamma) = \emptyset$ and
  the final claim follows from Corollary~\ref{cor:rank_C_animations}.
\end{proof}

\begin{cor}
  \label{cor:reflexive_rank}
  Let $\Gamma = (V,E)$ be a reflexive graph with $n$ vertices.
  Then
  \[
    \rank_{\QQ(X_V)}(\sC^+_\Gamma) = \rank_{\QQ(X_V)}(\sC^-_\Gamma) =
    \rank_{\QQ(X_V)}(\sC_{\Adj(\Gamma)}) = n.
    \pushQED{\qed}
    \qedhere
  \]
\end{cor}

\subsection{Proof of Theorem~\ref{thm:rgmt}}
\label{ss:proof_RGMT}

\begin{lemma}
  \label{lem:reflexive_cycle_breaking}
  Let $\Gamma$ be a reflexive graph.
  Let $\alpha \in \Animations(\Gamma)$.
  Then there exists $\beta \in \Fixed(\Gamma)$ with $\mon\alpha = \mon\beta$.
\end{lemma}
\begin{proof}
  Let $V^\per$ denote the set of $\alpha$-periodic points.
  Let $\mathscr l(\alpha)$ be the number of $\alpha$-orbits of size $> 1$ on
  $V^\per$.
  We proceed by induction on $\mathscr l(\alpha)$.
  If $\mathscr l(\alpha) = 0$, then $\alpha \in \Fixed(\Gamma)$ so we may take
  $\beta = \alpha$.
  Next, let $\mathscr l(\alpha) > 0$.
  Then we can find $r \ge 2$ and distinct points $u_1,\dotsc,u_r\in V$ with
  $u_i^\alpha = u_{i+1}$ for $i < r$ and $u_r^\alpha = u_1$.
  Using the notation from \S\ref{s:selectors}, define
  \[
    \alpha' = \alpha[u_1 \subs u_1, \dotsc, u_r \subs u_r].
  \]
  Since $\Gamma$ is reflexive, $\alpha' \in \Animations(\Gamma)$.
  By construction, $\mon{\alpha'} = \mon\alpha$ and $\mathscr{l}(\alpha') <
  \mathscr{l}(\alpha)$ whence the claim follows by induction.
\end{proof}

\begin{proof}[{Proof of Theorem~\ref{thm:rgmt}}]
  Let $\Gamma$ have $n$ vertices.
  Fix a compact \DVR{} $\fO$ with \itemph{odd} residue characteristic.
  Let $s\in \CC$ with $\Real(s) > n$ be arbitrary.
  By combining Proposition~\ref{prop:proj_circ_int},
  Corollary~\ref{cor:reflexive_rank}, 
  Lemma~\ref{lem:reflexive_cycle_breaking},
  Proposition~\ref{prop:selectors}, and
  Theorem~\ref{thm:animations},
  we see that
  each of
  \[
    (1-q^{-s}) W^\pm_\Gamma(q,q^{-s}) = (1-q^{-s})
    \zeta^{\ask}_{\sA^\pm_\Gamma/\fO}(s)
  \]
  and
  \[
    (1-q^{-s})W_{\Adj(\Gamma)}(q,q^{-s}) = (1-q^{-s})
    \zeta^{\ask}_{\sA_{\Adj(\Gamma)}/\fO}(s)
  \]
  is given by
  \[
    1 + (1-q^{-1})^{-1}
    \int\limits_{(\fO V)^\times \times\, \fP}
    \abs{z}^{s - 1} \prod_{i=1}^n
    \frac{\norm{I_{i-1}(x)}}{\norm{I_i(x) \cup z I_{i-1}(x)}}
    \dd\!\mu(x,z),
  \]
  where $I_k = \langle \mon\alpha : \alpha \in \Fixed_k(\Gamma)\rangle$.
  (Recall that $\Animations(\Gamma) = \Selectors(\Adj(\Gamma))$.)
\end{proof}

\section{Nilpotent animations}
\label{s:nilpotent}

Proposition~\ref{prop:ask_Gamma_integral} and
Corollary~\ref{cor:loopless_minus_minors}
suggest that by studying the nilpotent animations of a loopless graph $\Gamma$,
we might learn something about the rational functions $W^-_\Gamma(X,T)$.
In this section, we develop basic tools for working with and modifying
nilpotent animations.

\setcounter{subsection}{-1}
\subsection{Nilpotent animations and in-forests}
\label{ss:in-forests}

Nilpotent animations can be equivalently described in terms of
\itemph{in-forests}.
While this description is not logically required in the following,
we include it for the benefit of readers who appreciate helpful
pictures in graph-theoretic papers.

\paragraph{In-forests.}
The following is folklore.
Let $\Gamma = (V,E)$ be a graph.
An \emph{orientation} of~$\Gamma$ is pair $(\source,\target)$ of functions
$E\to V$ such that $e = \{ \source(e), \target(e)\}$ for all $e\in E$.
We call $\source(e)$ and $\target(e)$ the \emph{source} and \emph{target} of
$e$, respectively.
The \emph{outdegree} $\outdeg_\Gamma(v)$ of a vertex $v\in V$ is the number of edges
$e\in E$ with $\source(e) = v$.
A vertex $v$ with $\outdeg_\Gamma(v) = 0$ is a \emph{sink}.
An \emph{oriented graph} is a graph endowed with an orientation.
A \emph{forest} is an acyclic loopless graph.
An \emph{in-forest} is an oriented forest $\Phi$ with $\outdeg_\Phi(v) \le 1$ for
each vertex $v$ of $\Phi$.
An \emph{in-tree} is a connected in-forest.
Every in-tree contains a unique sink.

An in-forest structure on a given forest is equivalently described by a
choice of sinks, one from each connected component.
In detail, let $\Phi = (V,E)$ be a forest.
Let $V = V_1 \sqcup \dotsb \sqcup V_c$ be the decomposition of $V$ into the
connected components of $\Phi$.
For $i=1,\dotsc,c$, choose $s_i \in V_i$.
For each $v\in V_i\setminus\{s_i\}$, there exists a unique path from $v$ to
$s_i$ in $\Phi$.
We endow $\Phi$ with an orientation as follows:
for each $v\in V$, say $v\in V_i$, consider the unique path from $v$ to $s_i$
in $\Phi$ and orient all edges on this path towards $s_i$.
This turns $\Phi$ into an in-forest whose sinks are precisely the $s_i$.
Conversely, all in-forest structures on $\Phi$ arise in this fashion;
cf.\ \cite[Prop.\ 7.6]{cico}.

Given an arbitrary graph $\Gamma = (V,E)$,
by an \emph{in-forest in $\Gamma$},
we mean an in-forest whose underlying forest is a subgraph of $\Gamma$ with
vertex set $V$.
Hence, an in-forest in $\Gamma$ uniquely determines and is uniquely determined
by to two pieces of data:
\begin{itemize}
\item an acyclic set of edges $E'\subset E$ and
\item a choice of sinks, one from each connected component of the forest
  $(V,E')$.
\end{itemize}

\paragraph{Nilpotent animations of $\Gamma$ and in-forests in $\Gamma$.}
Let $\Gamma = (V,E)$ be a graph with $n$ vertices.
As we explain in the following, nilpotent animations of $\Gamma$ and
in-forests in $\Gamma$ are naturally in bijection.
Very briefly, given $\alpha \in \Nilpotents(\Gamma)$,
an identity $v^\alpha = w$ (where $v,w\in V$) corresponds to an oriented edge
$v\to w$ in the in-forest attached to $\alpha$.

In greater detail, suppose that $\Phi = (V,E')$
is an in-forest in $\Gamma$ with orientation $(\source, \target)$.
Define $\alpha \in \Animations(\Gamma)$ as follows.
For $v\in V$, we have $\outdeg_\Phi(v) \in \{0,1\}$.
If $\outdeg_\Phi(v) = 0$, then we let $v^\alpha = \udef$.
Otherwise, there exists a unique edge $e_v \in E'$ with $\source(e_v) = v$.
We then let $v^\alpha = \target(e_v)$.
As $\Phi$ is acyclic, it follows easily that $\alpha$ is a nilpotent animation
of $\Gamma$.

Conversely, let $\alpha \in \Nilpotents(\Gamma)$.
Define $E' = \{ \{ v,v^\alpha\} : v\in \Dom(\alpha)\}\subset E$.
Then $\Phi := (V, E')$ is a forest.
Each connected component of $\Phi$ contains a unique vertex $v$ with $v^\alpha
= \udef$.
Hence, by taking the elements of $V\setminus \Dom(\alpha)$ as our sinks, we
turn $\Phi$ into an in-forest in $\Gamma$.

The preceding two constructions yield mutually inverse bijections between
in-forests in $\Gamma$ and nilpotent animations of $\Gamma$.
If $\alpha\in\Nilpotents(\Gamma)$ has degree $k$, then the corresponding
in-forest has precisely $n-k$ connected components.
The monomial associated with an in-forest is the product of the variables
attached to the targets of its edges.

\begin{ex}
  \label{ex:in-forest}
  Consider the following graph:
  
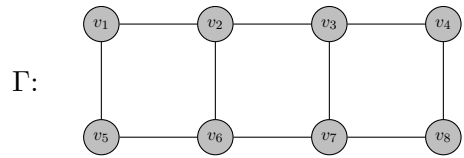
\begin{center}
    \begin{tikzpicture}
      \tikzstyle{Grey Vertex}=[fill=lightgray, draw=black, shape=circle, scale=0.6]

      \node at (-1,-0.75) {$\Gamma\colon$};
      
      \node [style=Grey Vertex] (1) at (0,0) {$v_1$};
      \node [style=Grey Vertex] (2) at (1.5,0) {$v_2$};
      \node [style=Grey Vertex] (3) at (3,0) {$v_3$};
      \node [style=Grey Vertex] (4) at (4.5,0) {$v_4$};

      \node [style=Grey Vertex] (5) at (0,-1.5) {$v_5$};
      \node [style=Grey Vertex] (6) at (1.5,-1.5) {$v_6$};
      \node [style=Grey Vertex] (7) at (3,-1.5) {$v_7$};
      \node [style=Grey Vertex] (8) at (4.5,-1.5) {$v_8$};

      \draw (1) to (2) to (3) to (4) to (8) to (7) to (6) to (5) to (1);
      \draw (2) to (6);
      \draw (3) to (7);
      \end{tikzpicture}
    \end{center}

    The following is an in-forest in $\Gamma$. Here, oriented edges
    belong to $\Phi$ while dashed ones belong to $\Gamma$ but not to $\Phi$.
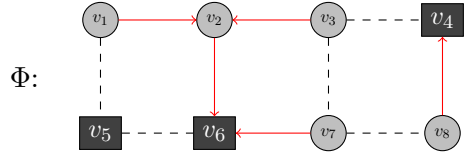
        \begin{center}
    \begin{tikzpicture}
      \tikzstyle{Grey Vertex}=[fill=lightgray, draw=black, shape=circle, scale=0.6]
      \tikzstyle{Sink}=[fill=darkgray, text=white, draw=black, shape=rectangle, scale=0.9]

      \node at (-1,-0.75) {$\Phi\colon$};
      
      \node [style=Grey Vertex] (1) at (0,0) {$v_1$};
      \node [style=Grey Vertex] (2) at (1.5,0) {$v_2$};
      \node [style=Grey Vertex] (3) at (3,0) {$v_3$};
      \node [style=Sink] (4) at (4.5,0) {$v_4$};

      \node [style=Sink] (5) at (0,-1.5) {$v_5$};
      \node [style=Sink] (6) at (1.5,-1.5) {$v_6$};
      \node [style=Grey Vertex] (7) at (3,-1.5) {$v_7$};
      \node [style=Grey Vertex] (8) at (4.5,-1.5) {$v_8$};

      \draw[->,red] (1) to (2);
      \draw[->,red] (2) to (6);
      \draw[->,red] (3) to (2);
      \draw[->,red] (7) to (6);
      \draw[->,red] (8) to (4);

      \draw[dashed] (1) to (5) to (6);
      \draw[dashed] (3) to (4);
      \draw[dashed] (3) to (7) to (8);
      \end{tikzpicture}
    \end{center}

    The sinks  $v_4$, $v_5$, and $v_6$ of $\Phi$ are drawn as squares in inverted colours.
    The nilpotent animation $\alpha$ corresponding to $\Phi$ is given by
    $v_1^\alpha = v_3^\alpha = v_2$, $v_2^\alpha = v_7^\alpha = v_6$,
    $v_8^\alpha = v_4$, and $v_4^\alpha = v_5^\alpha = v_6^\alpha = \udef$.
\end{ex}

\subsection{Using animations to order vertices}
\label{ss:animations_order_vertices}

Let $\Gamma = (V,E)$ be a graph.
Let $\alpha \in \Animations(\Gamma)$.
Define a binary relation $\cover_\alpha$ on $V$ by letting $v \cover_\alpha w$
if and only if $w = v^\alpha$ for $v,w\in V$.
(Note that $v \cover_\alpha w$ forces $v^\alpha \not= \udef$.)
Let $\lte_\alpha$ be the reflexive transitive closure of $\cover_\alpha$.
Then $\lte_\alpha$ is a preorder on $V$.
Given $v,w\in V$, we have $v\lte_\alpha w$ if and only if there exists $n\ge
0$ with $v^{\alpha^n} = w$.
For $v\in V$, let $\mathcal L_\alpha(v) = \{ w \in W : w\lte_\alpha v\}$ be
the associated lower set.
We record some basic properties of $\lte_\alpha$.
Recall that given a preorder $\olte$, an element $x$ is \emph{maximal} if $x
\olte y$ implies $y \olte x$.

\begin{prop}
  \label{prop:animation_order}
  Let $\Gamma = (V,E)$ be a graph,
  $\alpha \in \Animations(\Gamma)$, and $v,w \in V$.
  Then:
  \begin{enumerate}[(i)]
  \item \label{prop:animation_order1}

    $v$ is a $\lte_\alpha$-maximal element of $V$ if and only if
    $v^\alpha = \udef$ or $v$ is an $\alpha$-periodic point.
  \item \label{prop:animation_order2}
    $v$ and $w$ are $\lte_\alpha$-comparable if and only if
    $\mathcal L_\alpha(v) \cap \mathcal L_\alpha(w) \not= \emptyset$.
  \item \label{prop:animation_order3}
    The preorder $\lte_\alpha$ is a partial order if and only if  $\alpha \in
    \Fixed(\Gamma)$.
    If $\Gamma$ is loopless, then the latter condition is equivalent to
    $\alpha \in \Nilpotents(\Gamma)$.
  \item \label{prop:animation_order4}
    If $\alpha \in \Nilpotents(\Gamma)$, then $\cover_\alpha$ is the covering
    relation associated with $\lte_\alpha$.
  \item \label{prop:animation_order5}
    Suppose that $\alpha \in \Fixed(\Gamma)$.
    Then, given $v\in V$, there exists a unique $\lte_\alpha$-maximal element
    $z\in V$ with $v \lte_\alpha z$.
  \end{enumerate}
\end{prop}
\begin{proof}
  \begin{enumerate*}
  \item
    Clear.
  \item
    If $v \lte_\alpha w$, say, then
    $v \in \mathcal L_\alpha(v) \cap \mathcal L_\alpha(w)$.
    Conversely, suppose that
    $r \in \mathcal L_\alpha(v) \cap \mathcal L_\alpha(w)$.
    Then there are $m,n\ge 0$ with $v = r^{\alpha^m}$ and $w = r^{\alpha^n}$.
    Suppose, without loss of generality, that $m \le n$.
    Then $w = v^{\alpha^{n-n}}$ and therefore $v \lte_\alpha w$.
  \item
    Clear.
  \item
    Clear.
  \item
    Given $v$, there exists a least $n \ge 0$ such that $\alpha$ sends $z :=
    v^{\alpha^n}$ to $\udef$ or to itself.
    Then $z$ is $\lte_\alpha$-maximal with $v \lte_\alpha z$.
    The uniqueness of $z$ follows from~\ref{prop:animation_order2}--\ref{prop:animation_order3}.
  \end{enumerate*}
\end{proof}

In the setting of
Proposition~\ref{prop:animation_order}\ref{prop:animation_order5},
we write $\last_\alpha(v) = z$, where $z$ is the unique
$\lte_\alpha$-maximal element of $V$ with $v\lte_\alpha z$.
Recall from \S\ref{s:selectors} that $\Dom(\alpha)$ denotes the domain of
definition of $\alpha$.
Given a nilpotent animation $\alpha$ of $\Gamma = (V,E)$, the
$\lte_\alpha$-maximal elements are precisely the elements of $V \setminus
\Dom(\alpha)$.
We thus have the following.

\begin{lemma}
  \label{lem:nilpotent_number_maximals}
  Let $\Gamma = (V,E)$ be a graph with $n$ vertices.
  Let $\alpha \in \Nilpotents_k(\Gamma)$.
  Then the number of $\lte_\alpha$-maximal elements of $V$ is $n-k$.
  \qed
\end{lemma}

\subsection{New nilpotent animations from old ones,  I}
\label{ss:new_nilpotents}

Let $\Gamma = (V,E)$ be a loopless graph.
In the next section, we will carry out various types of manipulations
applied to nilpotent animations with the goal of optimising various
parameters.
The following two lemmas are basic steps of these manipulations.
The first lemma tells us precisely when redefining (or extending) a nilpotent
animation gives rise to an animation which is again nilpotent.
Recall that $\sim$ indicates adjacency in graphs.

\begin{lemma}[Redefining nilpotent animations: minor surgery]
  \label{lem:redefine_nilpotent}
  Let $\alpha \in \Nilpotents(\Gamma)$.
  Let $v,w\in V$ with $v\sim w$.
  Let $\beta = \alpha[v\subs w] \in \Animations(\Gamma)$.
  (We emphasise that we do not require that $v^\alpha \not=\udef$
  so $\beta$ might be an extension of $\alpha$.)
  Then $\beta \in \Nilpotents(\Gamma)$ if and only if $w \notlte_\alpha v$.
\end{lemma}
\begin{proof}
  First note that since $\Gamma$ is loopless and $v\sim w$,
  we have $v \not= w$.
  We will prove the lemma through a series of auxiliary claims and steps.
  \begin{enumerate}[(a)]
  \item \label{lem:redefine_nilpotent1}
    We first claim that $\beta\notin \Nilpotents(\Gamma)$ if and only if $v$ is
    $\beta$-periodic.

    Clearly, if $v$ is $\beta$-periodic, then $\beta \notin
    \Nilpotents(\Gamma)$.
    Conversely, suppose that $\beta\notin \Nilpotents(\Gamma)$.
    Then there exists a $\beta$-periodic vertex, $u \in V$ say.
    Let $k \ge 1$ with $u^{\alpha^k} = u$.
    Since $\alpha \in \Nilpotents(\Gamma)$, the vertex $u$ is not
    $\alpha$-periodic.
    We conclude that the sequence
    \[
      u, u^\beta, u^{\beta^2},\dotsc,u^{\beta^{k-1}}, u^{\beta^k} = u
    \]
    must contain $v$.
    In particular, $v$ is $\beta$-periodic.
  \item \label{lem:redefine_nilpotent2}
    Next, we note that $v$ is $\beta$-periodic if and only if $w \lte_\beta v$.
    Indeed, as $v^\beta = w$, we see that $v$ is
    $\beta$-periodic if and only if repeated application of $\beta$ takes $w$
    to $v$ or, equivalently, 
    $w \lte_\beta v$.
  \item \label{lem:redefine_nilpotent3}
    We claim that if $w \lte_\beta v$, then also $w \lte_\alpha v$.

    Suppose that $w \lte_\beta v$.
    Let $k\ge 0$ be minimal with $v = w^{\beta^k}$.
    Since $v\not= w$, we have $k\ge 1$.
    By the minimality of $k$, each of $w, w^\beta,\dotsc,w^{\beta^{k-1}}$ is
    distinct from $v$.
    Thus, $\beta$ acts like $\alpha$ on these points whence $w\lte_\alpha v$.
  \item \label{lem:redefine_nilpotent4}
    Suppose that $w\lte_\alpha v$.
    By construction, $\beta$ acts like $\alpha$ on the points
    $w,w^\alpha,\dotsc$ preceding $v$.
    Hence, $w \lte_\beta v \lte_\beta w$ whence $\beta
    \notin\Nilpotents(\Gamma)$
    by Proposition~\ref{prop:animation_order}\ref{prop:animation_order3}.
    This proves the ``only if'' part.
  \item \label{lem:redefine_nilpotent5}
    Suppose that $w\notlte_\alpha v$.
    By step \ref{lem:redefine_nilpotent3}, we then have $w \notlte_\beta v$.
    By step \ref{lem:redefine_nilpotent2}, $v$ is then \underline{not}
    $\beta$-periodic whence $\beta\in\Nilpotents(\Gamma)$ follows from step
    \ref{lem:redefine_nilpotent1}.
    This proves the ``if part''.
    \qedhere
  \end{enumerate}
\end{proof}    

\begin{rem}
  Expressed in the language of in-forests from \S\ref{ss:in-forests},
  Lemma~\ref{lem:redefine_nilpotent} asserts the following.
  Let $\Phi$ be an in-forest in $\Gamma$.
  Let $v$ and $w$ be adjacent vertices of $\Gamma$.
  Let $\Phi'$ be the oriented graph obtained from $\Phi$ by deleting,
  if it exists,   the (necessarily unique) edge with source $v$ and by
  inserting an oriented edge $v\to w$.
  Then $\Phi'$ is an in-forest if and only if $\Phi$ does not contain a
  directed path from $w$ to $v$.
\end{rem}

Given a nilpotent animation $\alpha$ with $v^\alpha = w$, let $z =
\last_\alpha(w)$.
Suppose that $w \not= z$ and $v \sim z \sim w$.
Then the following lemma allows us to construct an explicit
$\beta \in \Nilpotents(\Gamma)$ with $v \cover_\beta z \cover_\beta w$ and
such that $\mon\alpha = \mon\beta$.

\begin{lemma}[Redefining nilpotent animations: bypass surgery I]
  \label{lem:bypass_nilpotent}
  Let $\alpha \in \Nilpotents(\Gamma)$.
  Let $v,w,y,z\in V$ with $z^\alpha = \udef$.
  Suppose that
  \[
    v \cover_\alpha w \lte_\alpha y \cover_\alpha z
  \]
  and $v \sim z \sim w$.
  (Note that then necessarily $\#\{v,y,z\} = 3$.)
  Let
  \[
    \beta = \alpha[v \subs z, \, y \subs \udef, \, z \subs w].
  \]
  Then $\beta \in \Nilpotents(\Gamma)$ and $\mon\alpha = \mon\gamma$.
\end{lemma}
\begin{proof}
  This follows by repeated application of Lemma~\ref{lem:redefine_nilpotent}
  as follows.
  First, let $\alpha' = \alpha[v \subs \udef, y \subs \udef]$.
  As $\alpha$ belongs to $\Nilpotents(\Gamma)$, so does $\alpha'$.
  Note that $v$, $y$, and $z$ are distinct $\lte_{\alpha'}$-maximal elements
  of $V$.
  By construction, we have $w \lte_{\alpha'} y$.
  Proposition~\ref{prop:animation_order}\ref{prop:animation_order5}
  thus shows that $w\notlte_{\alpha'} z$.
  Lemma~\ref{lem:redefine_nilpotent} therefore shows that
  $\alpha'' = \alpha'[z \subs w] \in \Nilpotents(\Gamma)$.
  Next, $v$ and $y$ are distinct $\lte_{\alpha''}$-maximal elements of $V$ and
  $z \lte_{\alpha''} y$.
  Again, Proposition~\ref{prop:animation_order}\ref{prop:animation_order5}
  shows that $z \notlte_{\alpha''} v$.
  By applying Lemma~\ref{lem:redefine_nilpotent} to $\alpha''$,
  we thus obtain $\beta = \alpha''[v \subs z] \in \Nilpotents(\Gamma)$.
\end{proof}

\begin{rem}
  In terms of in-forests, Lemma~\ref{lem:bypass_nilpotent} asserts the
  following.
  We use the same notation as in Example~\ref{ex:in-forest}.
  Suppose that the following is part of an in-forest $\Phi$ in $\Gamma$.
  
  \begin{center}
    \begin{tikzpicture}
      \tikzstyle{Grey Vertex}=[fill=lightgray, draw=black, shape=circle, scale=0.6]
      \tikzstyle{Sink}=[fill=darkgray, text=white, draw=black, shape=rectangle, scale=0.9]

      \node [style=Grey Vertex] (1) at (0,0) {$v$};
      \node [style=Grey Vertex] (2) at (1.5,0) {$w$};
      \node  (3) at (3,0) {$\dotso$};
      \node [style=Grey Vertex] (4) at (4.5,0) {$y$};
      \node [style=Sink] (5) at (6,0) {$z$};

      \draw[->,red] (1) to (2);
      \draw[->,red] (2) to (3);
      \draw[->,red] (3) to (4);
      \draw[->,red] (4) to (5);
      \end{tikzpicture}
    \end{center}
    Suppose that $v \sim z \sim w$.
    Then the oriented graph obtained from $\Phi$ by rerouting as follows is an
    in-forest in $\Gamma$ with the same associated monomial as $\Phi$.

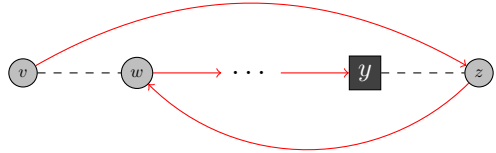
    \begin{center}
    \begin{tikzpicture}
      \tikzstyle{Grey Vertex}=[fill=lightgray, draw=black, shape=circle, scale=0.6]
      \tikzstyle{Sink}=[fill=darkgray, text=white, draw=black, shape=rectangle, scale=0.9]

      \node [style=Grey Vertex] (1) at (0,0) {$v$};
      \node [style=Grey Vertex] (2) at (1.5,0) {$w$};
      \node  (3) at (3,0) {$\dotso$};
      \node [style=Sink] (4) at (4.5,0) {$y$};
      \node [style=Grey Vertex] (5) at (6,0) {$z$};

      \draw[dashed] (1) to (2);
      \draw[->,red] (2) to (3);
      \draw[->,red] (3) to (4);
      \draw[dashed] (4) to (5);

      \draw[->,red] (1) to[out=30, in=150] (5);
      \draw[->,red] (5) to[out=-135, in=-45] (2);
      \end{tikzpicture}
    \end{center}
\end{rem}

\subsection{Ordering monomials relative to a point}
\label{ss:ordering_monomials}

In addition to the partial orders $\lte_\alpha$ associated with nilpotent
animations $\alpha$ from \S\ref{ss:animations_order_vertices}, we also
consider partial orders $\lte_u$ defined by a choice of a distinguished
vertex~$u$.

Let $V$ be a finite set.
Recall that $\fO$ denotes a compact \DVR{} as in~\S\ref{ss:notation}.
Following \cite[\S 4.2]{cico},
for a subset $\mathscr{S} \subset \RR V$, we write
$\mathscr{S}(\fO) = \{ x\in \fO V : \ord(x) \in \mathscr{S}\}$.
We let $\cdot$ denote the inner product $x\cdot y = \sum\limits_{v\in V} x_v y_v$ on
$\RR V$.
Recall that the \emph{dual cone} of $\mathscr S\subset \RR V$ is
\[
  \mathscr{S}^\vee = \left\{
    x \in \RR V : x \cdot y \ge 0 \text{ for all }
    y\in \mathscr{S}
    \right\}.
\]

Let $u \in V$.
Define
\[
  \Cone_u V = \left\{
    x \in \RR_{\ge 0} V : x_u \le x_v \text{ for all } v\in V
  \right\}
\]
We define a binary relation $\lte_u$ on $\ZZ V$ by letting $a \lte_u b$ if and only
if $b - a \in (\Cone_u V)^\vee$.

\begin{prop}
  \label{prop:vertex_order}
  \quad
  \begin{enumerate}[(i)]
  \item
    \label{prop:vertex_order1}
    $\lte_u$ is a partial order on $\ZZ V$.
  \item
    \label{prop:vertex_order2}
    Let $a,b\in \ZZ V$ with $a \lte_u b$.
    Let $x\in \Cone_u V(\fO)$.
    Then $x^b / x^a \in \fO$.
  \item
    \label{prop:vertex_order3}
    Let $a,b \in \ZZ V$ with $a_v \le b_v$ for all $v\in V\setminus\{u\}$ and
    $a_u \le b_u + \sum\limits_{v\in V\setminus\{u\}}(b_v - a_v)$.

    Then $a \lte_u b$.
  \end{enumerate}
\end{prop}
\begin{proof}
  \quad
  \begin{enumerate}
  \item
    Only the antisymmetry of $\lte_u$ needs a justification.
    Suppose that $a\in \ZZ V$ with $a,-a \in (\Cone_u V)^\vee$.
    Since $\std v \in \Cone_u V$ for $v \in V\setminus\{u\}$, we have
    $a_v = 0$ for all $v \in V \setminus\{u\}$.
    Next, $z := \sum_{v\in V} \std v\in \Cone_u V$ and thus $a_u = a \cdot z =
    0$.
  \item
    For every $c\in \ZZ V$, we have $\ord(x^c) = \ord(x) \cdot c$.
    Hence, if $a \lte_u b$, then $\ord(x^{b-a}) = \ord(x) \cdot (b-a) \ge 0$ whence
    $\ord(x^a) \le \ord(x^b)$.
  \item
    Let $x\in \Cone_u V$ be arbitrary.
    Then
    \[
      (a_u - b_u)x_u \le
      \sum_{v\in V\setminus\{u\}} (b_v - a_v) x_u \le
      \sum_{v\in V\setminus\{u\}} (b_v - a_v) x_v
    \]
    whence $(b-a) \cdot x \ge 0$.
    Thus, $b-a \in (\Cone_u V)^\vee$.
    \qedhere
  \end{enumerate}
\end{proof}

We also write $\lte_u$ for the partial order on Laurent monomials in $X_V$
given by $X^a \lte_u X^b$ if and only if $a \lte_u b$.

\begin{prop}
  \label{prop:ordering_monomials}
  Let $u \in V$.
  \begin{enumerate}[(i)]
  \item
    \label{prop:ordering_monomials1}
    Let $w \in V$.
    Let $m$ be an arbitrary Laurent monomial in $X_V$.
    Then $X_u  m \lte_u X_w m$.
    In particular, $X_u^{\phantom 1}X_w^{-1} m \lte_u m$.
  \item 
    \label{prop:ordering_monomials2}
      Let $\fO$ be a compact \DVR{}.
      Let $x\in \fO V$ with $\ord(x_u) \le \ord(x_v)$
      for all $v\in V$ and such that $\prod\limits_{v\in V} x_v \not= 0$.
      Let $e, f \in \NN_0 V$ with $X_V^e \lte_u X_V^f$.
      Then $x^e \mid x^f$.
  \end{enumerate}
\end{prop}
\begin{proof}
  \quad
  \begin{enumerate}
  \item
    We may assume that $u \not= w$.
    Write $m = X_V^e$ for $e\in \ZZ V$.
    Let $a = \std u + e$ and $b = \std w + e$.
    Then
    \begin{itemize}
    \item
      $a_v = b_v$ for all $v \in V \setminus \{ u,w \}$,
    \item
      $a_u = e_u + 1$,
    \item
      $b_u = e_u$,
    \item
      $b_w = a_w + 1$, and
    \item
      $a_u = e_u + 1 = b_u + \sum\limits_{v\in V\setminus\{u\}} (b_v - a_v)$.
    \end{itemize}
    By Proposition~\ref{prop:vertex_order}\ref{prop:vertex_order3},
    $a\lte_u b$ and thus
    $X_u m = X_V^a \lte_u X_V^b = X_w m$.
    The final claim follows by replacing $m$ by $X_w^{-1} m$.
  \item
    We have $\ord(x) \in \Cone_u V(\fO)$ and $x^e,x^f\in \fO \setminus\{0\}$.
    The claim thus follows from Proposition~\ref{prop:vertex_order}\ref{prop:vertex_order2}.
    \qedhere
  \end{enumerate}
\end{proof}

\subsection{New nilpotent animations from old ones, II}
\label{ss:new_nilpotents}

Let $\Gamma = (V,E)$ be a loopless graph.
We record two lemmas in the spirit of
Lemmas~\ref{lem:redefine_nilpotent}--\ref{lem:bypass_nilpotent}, but with the
relation $\lte_u$ in place of equality of monomials.

\begin{lemma}[Redefining nilpotent animations: rerouting]
  \label{lem:reroute_nilpotent}
  Let $\alpha \in \Nilpotents(\Gamma)$.
  Let $u,v,z\in V$ with $u \lte_\alpha v \lt_\alpha z$ and $z^\alpha = \udef$.
  Suppose that $u \sim z$.
  Let $\beta = \alpha[v \subs \udef, z\subs u]$.
  Then $\beta \in \Nilpotents(\Gamma)$, $\deg(\beta) = \deg(\alpha)$, and
  $\mon\beta \lte_u \mon\alpha$.
\end{lemma}
\begin{proof}
  Let $\alpha' = \alpha[v \subs \udef]$.
  Then $u \lte_{\alpha'} v$
  and since $v$ and $z$ are distinct $\lte_{\alpha'}$-maximal elements of $V$,
  Proposition~\ref{prop:animation_order}\ref{prop:animation_order5}
  implies that $u \notlte_{\alpha'} z$.
  Lemma~\ref{lem:redefine_nilpotent} thus shows that
  $\beta = \alpha'[z \subs u] = \alpha[v\subs \udef, z \subs u] \in
  \Nilpotents(\Gamma)$ and $\deg(\alpha) = \deg(\beta)$.
  By Proposition~\ref{prop:ordering_monomials}\ref{prop:ordering_monomials1},
  $\mon\beta  = X_u^{\phantom 1} X_{v^\alpha}^{-1}\mon\alpha \lte_u \mon\alpha$.
\end{proof}

\begin{rem}
  In terms of in-forests, Lemma~\ref{lem:reroute_nilpotent} asserts the
  following.
  Suppose that the following is part of an in-forest $\Phi$ in $\Gamma$ and
  that $u\sim z$.
  
  \begin{center}
    \begin{tikzpicture}
      \tikzstyle{Grey Vertex}=[fill=lightgray, draw=black, shape=circle, scale=0.6]
      \tikzstyle{Sink}=[fill=darkgray, text=white, draw=black, shape=rectangle, scale=0.9]

      \node [style=Grey Vertex] (1) at (0,0) {$u$};
      \node (2) at (1.5,0) {$\dotso$};
      \node [style=Grey Vertex] (3) at (3,0) {$v$};
      \node [style=Grey Vertex] (4) at (4.5,0) {$\phantom x$};
      \node  (5) at (6,0) {$\dotso$};
      \node [style=Sink] (6) at (7.5,0) {$z$};

      \draw[->, red] (1) to (2);
      \draw[->, red] (2) to (3);
      \draw[->, red] (3) to (4);
      \draw[->, red] (4) to (5);
      \draw[->, red] (5) to (6);
      \end{tikzpicture}
    \end{center}
    Then the oriented graph $\Phi'$ obtained from $\Phi$ by rerouting as
    follows is an in-forest in $\Gamma$ whose associated monomial is less than
    or equal w.r.t.\ $u$ than the monomial of $\Phi$ and of the same degree.

    
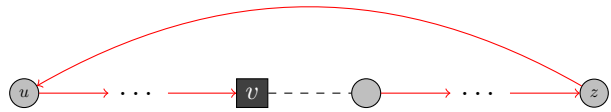
\begin{center}
    \begin{tikzpicture}
      \tikzstyle{Grey Vertex}=[fill=lightgray, draw=black, shape=circle, scale=0.6]
      \tikzstyle{Sink}=[fill=darkgray, text=white, draw=black, shape=rectangle, scale=0.9]

      \node [style=Grey Vertex] (1) at (0,0) {$u$};
      \node (2) at (1.5,0) {$\dotso$};
      \node [style=Sink] (3) at (3,0) {$v$};
      \node [style=Grey Vertex] (4) at (4.5,0) {$\phantom x$};
      \node  (5) at (6,0) {$\dotso$};
      \node [style=Grey Vertex] (6) at (7.5,0) {$z$};

      \draw[->, red] (1) to (2);
      \draw[->, red] (2) to (3);
      \draw[dashed] (3) to (4);
      \draw[->, red] (4) to (5);
      \draw[->, red] (5) to (6);
      \draw[->,red] (6) to[out=150, in=30] (1);
    \end{tikzpicture}
    \end{center}
\end{rem}

\begin{lemma}[Redefining nilpotent animations: bypass surgery II]
  \label{lem:unit_bypass_nilpotent}
  Let $\alpha \in \Nilpotents(\Gamma)$.
  Let $u,i,v,z \in V$ with $i^\alpha = v$ and $z^\alpha = \udef$.
  Suppose that
  \[
    u \lte_\alpha i \cover_\alpha v \lt_\alpha z
  \]
  and $u \sim v \sim z$.
  (Note that then necessarily $\#\{i,v,z\} = 3$.)
  Let
  \[
    \beta = \alpha[i \subs \udef, v \subs u, z \subs v].
  \]
  Then $\beta \in \Nilpotents(\Gamma)$, $\deg(\beta) = \deg(\alpha)$, and
  $\mon\beta \lte_u \mon\alpha$.
\end{lemma}
\begin{proof}
  Like Lemma~\ref{lem:bypass_nilpotent},
  this also follows by repeated application of Lemma~\ref{lem:redefine_nilpotent}.
  First, let $\alpha' = \alpha[i \subs \udef, v \subs \udef]$.
  Then $i$, $v$, and $z$ are distinct $\lte_{\alpha'}$-maximal elements of
  $V$.
  As in the proof of Lemma~\ref{lem:bypass_nilpotent}, we find that
  $u\notlte_{\alpha'} v$ whence $\alpha'' = \alpha'[v \subs u] \in
  \Nilpotents(\Gamma)$.
  Since $i$ and $z$ are distinct $\lte_{\alpha''}$-maximal elements with $v
  \lte_{\alpha''} u \lte_{\alpha''} i$, we obtain $\beta = \alpha''[z \subs v] \in
  \Nilpotents(\Gamma)$.
  Clearly, $\deg(\alpha) = \deg(\beta)$.
  Finally,
  by Proposition~\ref{prop:ordering_monomials}\ref{prop:ordering_monomials1},
  \[
    \mon\beta = \frac{X_u X_v}{X_{i^\alpha} X_{v^\alpha}}\mon\alpha =
    \frac{X_u}{X_{v^\alpha}}\mon\alpha\lte_u\mon\alpha.
    \qedhere
  \]
\end{proof}

\begin{rem}
  In terms of in-forests, Lemma~\ref{lem:bypass_nilpotent} asserts the
  following.
  Suppose that the following is part of an in-forest $\Phi$ in $\Gamma$ and
  that $u\sim v \sim z$.
  
  \begin{center}
    \begin{tikzpicture}
      \tikzstyle{Grey Vertex}=[fill=lightgray, draw=black, shape=circle, scale=0.6]
      \tikzstyle{Sink}=[fill=darkgray, text=white, draw=black, shape=rectangle, scale=0.9]

      \node [style=Grey Vertex] (1) at (0,0) {$u$};
      \node (2) at (1.5,0) {$\dotso$};
      \node [style=Grey Vertex] (3) at (3,0) {$i$};
      \node [style=Grey Vertex] (4) at (4.5,0) {$v$};
      \node [style=Grey Vertex] (5) at (6,0) {$\phantom x$};
      \node  (6) at (7.5,0) {$\dotso$};
      \node [style=Sink] (7) at (9,0) {$z$};

      \draw[->, red] (1) to (2);
      \draw[->, red] (2) to (3);
      \draw[->, red] (3) to (4);
      \draw[->, red] (4) to (5);
      \draw[->, red] (5) to (6);
      \draw[->, red] (6) to (7);
      \end{tikzpicture}
    \end{center}
    Then the oriented graph $\Phi'$ obtained from $\Phi$ by rerouting as
    follows is an in-forest in $\Gamma$ whose associated monomial is less than
    or equal w.r.t.\ $u$ than the monomial of $\Phi$ and of the same degree.

    
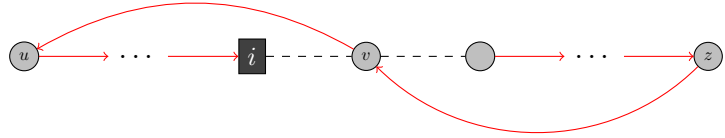
\begin{center}
    \begin{tikzpicture}
      \tikzstyle{Grey Vertex}=[fill=lightgray, draw=black, shape=circle, scale=0.6]
      \tikzstyle{Sink}=[fill=darkgray, text=white, draw=black, shape=rectangle, scale=0.9]

      \node [style=Grey Vertex] (1) at (0,0) {$u$};
      \node (2) at (1.5,0) {$\dotso$};
      \node [style=Sink] (3) at (3,0) {$i$};
      \node [style=Grey Vertex] (4) at (4.5,0) {$v$};
      \node [style=Grey Vertex] (5) at (6,0) {$\phantom x$};
      \node  (6) at (7.5,0) {$\dotso$};
      \node [style=Grey Vertex] (7) at (9,0) {$z$};

      \draw[->, red] (1) to (2);
      \draw[->, red] (2) to (3);
      \draw[dashed] (3) to (4);
      \draw[dashed] (4) to (5);
      \draw[->, red] (5) to (6);
      \draw[->, red] (6) to (7);

      \draw[->,red] (4) to[out=150, in=30] (1);
      \draw[->,red] (7) to[out=-135, in=-45] (4);
    \end{tikzpicture}
    \end{center}
\end{rem}

\section{Animations of joins of graphs}
\label{s:plumbing}

Let $\Gamma_1$ and $\Gamma_2$ be loopless graphs.
Let $\Gamma = \Gamma_1 \join \Gamma_2$ be their join.
Let $V$ be the vertex set of $\Gamma$.
Theorem~\ref{thm:animations}\ref{thm:animations-}
suggests that in order to prove Theorem~\ref{thm:join}, we should
relate the nilpotent animations of $\Gamma$ to those of $\Gamma_1$ and $\Gamma_2$.
In this section, we accomplish just that.
Recall that by Proposition~\ref{prop:ask_Gamma_integral}, we may express
$W^-_\Gamma(q,q^{-s})$ in terms of the integral
$\int_{(\fO V)^\times \times \fP}\Gamma^-(s)$.
When considering the integrand in \eqref{eq:Gamma_integral},
for each $(x,z) \in (\fO V)^\times \times \fP$,
there exists some vertex $u \in V$ whose associated coordinate $x_u$ is a unit.
As a major ingredient of our proof of Theorem~\ref{thm:join},
instead of characterising all nilpotent animations of~$\Gamma$,
in this section, we exhibit a subset of \itemph{$u$-centred} animations (see
Definition~\ref{d:centred}) relative to an arbitrary but fixed vertex $u \in
V$, corresponding to a unit coordinate as above.
For each choice of $u$, the key features of $u$-centred animations are as
follows.
\begin{itemize}
\item
  Let $u$ belong to $\Gamma_i$. 
  Then $u$-centred animations of $\Gamma$ arise very explicitly from nilpotent
  animations of $\Gamma_i$.
  In particular, the monomial associated with a $u$-centred animation can be explicitly
  described in terms of the monomial associated with an associated animation
  of $\Gamma_i$ (Example~\ref{ex:small_centred} and Proposition~\ref{prop:large_centred_minors}).
\item
  Every nilpotent animation is ``dominated'' by a $u$-centred one of the same
  degree (Theorem~\ref{thm:centred_main}).
  This allows us to focus on $u$-centred animations only.
\end{itemize}

Let $V_i$ be the vertex set of $\Gamma_i$.
Clearly,
\begin{equation}
  \label{eq:3part_int}
  \int\limits_{(\fO V)^\times \times \fP}
  \!\!\!\!
  \Gamma^-(s)
  =
  \int\limits_{(\fO V_1)^\times \times (\fO V_2)^\times \times \fP}
  \!\!\!\!
  \Gamma^-(s)
  +
  \int\limits_{(\fO V_1)^\times \times \fP  V_2 \times \fP}
  \!\!\!\!
  \Gamma^-(s)
  +
  \int\limits_{\fP V_1 \times (\fO V_2)^\times \times \fP}
  \!\!\!\!
  \Gamma^-(s).
\end{equation}

Among the summands on the right-hand side of \eqref{eq:3part_int},
the first is the easiest to analyse since it can be computed explicitly in
terms of $n_1$ and $n_2$; see Lemma~\ref{lem:units_everywhere_integral}.
Our proof of Theorem~\ref{thm:join} in \S\ref{s:proof_join} will rely heavily
on an analysis of the second and third summand in~\eqref{eq:3part_int} using
the machinery surrounding $u$-centred animations developed in the following.

\subsection{Setup, centred animations, and main result}
\label{ss:plumbing_setup}

Let $\Gamma_1 = (V_1,E_1)$ and $\Gamma_2 = (V_2,E_2)$ be loopless graphs.
We assume that $V_1 \cap V_2 = \emptyset$ and
$V_1\not= \emptyset \not= V_2$.
Write $n_i = \card{V_i}$ and $V = V_1 \sqcup V_2$.
Let $\Gamma = \Gamma_1 \join \Gamma_2$ be the join of $\Gamma_1$ and
$\Gamma_2$.

\begin{defn}
  \label{d:centred}
  Let $\alpha \in \Nilpotents_k(\Gamma)$.
  Let $u \in V$.
  Write $\{ 1, 2\} = \{ i, j\}$ with $u\in V_i$.
  We say that $\alpha$ is \emph{$u$-centred} if $\card{V_j \cap \Dom(\alpha)}
  = \min(k, n_j)$ and $v^\alpha = u$ for all $v \in V_j \cap \Dom(\alpha)$.
\end{defn}

\begin{rem}
  \label{rem:centred}
  \quad
  \begin{enumerate}[(i)]
  \item \label{rem:centred1}
    Whether $\alpha$ is $u$-centred generally depends on the specific
    representation of $\Gamma$ as a join  $\Gamma_1 \join \Gamma_2$ of
    subgraphs $\Gamma_1$ and $\Gamma_2$.
    These decompositions are far from unique,
    as the example of a complete graph $\mathrm K_n$ with $n\ge 4$ shows.
  \item \label{rem:centred2}
    We expand Definition~\ref{d:centred} as follows.
    Let $\alpha \in \Nilpotents_k(\Gamma)$.
    If $k\le n_j$, then $\alpha$ is $u$-centred if and only if $\Dom(\alpha)
    \subset V_j$ and $v^\alpha = u$ for all $v\in \Dom(\alpha)$.
    If $k\ge n_j$, then $\alpha$ is $u$-centred if and only if
    $V_j \subset \Dom(\alpha)$ and $v^\alpha = u$ for all $v\in V_j$.
  \end{enumerate}
\end{rem}

To avoid having to carry around the indices $i$ and $j$ all the time,
in the following, we simply assume that $u \in V_1$
so that $(i,j) = (1,2)$.

We will see in \S\ref{ss:minors_centred} that the minors associated with $u$-centred
nilpotent animations arise, in an explicit fashion, from minors associated
with nilpotent animations of $\Gamma_i$.
The following is the main result of this section.

\begin{thm}
  \label{thm:centred_main}
  Let $\Gamma = \Gamma_1 \join \Gamma_2$ and $u\in V$ as above.
  Let $\alpha \in \Nilpotents_k(\Gamma)$.
  Then there exists $\beta \in \Nilpotents_k(\Gamma)$ such that $\beta$ is
  $u$-centred and $\mon\beta \lte_u\mon\alpha$.
\end{thm}

\subsection{Minors of centred animations}
\label{ss:minors_centred}

Let $\Gamma = \Gamma_1 \join \Gamma_2$ as in \S\ref{ss:plumbing_setup}.
Without loss of generality, let $u \in V_1$.
We show the following.
\begin{enumerate}[(a)]
\item
  Nilpotent $u$-centred animations of $\Gamma$ of degree at most $n_2$
  (``small'') can be
  described explicitly.
  The associated minors are simply powers of $X_u$;
  see Example~\ref{ex:small_centred}.
\item
  Nilpotent $u$-centred animations of $\Gamma$ of degree $k\ge n_2$
  (``large'') arise explicitly from nilpotent animations of
  $\Gamma_1$.
  The associated minors are precisely of the form $X_u^{n_2} m \cdot \mon{\alpha'}$,
  where $\alpha' \in \Nilpotents(\Gamma_1)$ and $m$ is a monomial in
  $X_{V_2}$;
  see Proposition~\ref{prop:large_centred_minors}.
\end{enumerate}

\begin{ex}[Small $u$-centred animations and their minors]
  \label{ex:small_centred}
  Let $0 \le k \le n_2$.
  Let $v_1,\dotsc,v_k\in V_2$ be distinct.
  Define $\alpha\in\Nilpotents_k(\Gamma)$ via
  $\Dom(\alpha) = \{ v_1,\dotsc,v_k\}$ and $v_i^\alpha = u$.
  Then $\alpha$ is $u$-centred and $\mon\alpha = X_u^k$.
  Conversely, every $u$-centred nilpotent animation of degree $k \le n_2$ is
  of this form.
\end{ex}

As $\Gamma$ is connected and loopless with $n := n_1+n_2$ vertices,
Corollary~\ref{cor:rank_C-_loopless} shows that $\Nilpotents_{n-1}(\Gamma)
\not= \emptyset = \Nilpotents_{n}(\Gamma)$.

\begin{prop}[Minors of large $u$-centred animation]
  \label{prop:large_centred_minors}
  Let $k \ge n_2$.
  \begin{enumerate}
  \item
    Let $\alpha \in \Nilpotents_k(\Gamma)$ be $u$-centred.
    (The existence of $\alpha$ forces $k \le n_1 + n_2 - 1$.)
    Let $V_1' = \{ v \in V_1 : v^\alpha \in V_1\}$.
    We view $\alpha' = \alpha \restriction V_1'$ as an element of $\Nilpotents(\Gamma_1)$.
    Then
    \[
      \mon\alpha = X_u^{n_2}\, \prod_{v\in V_2^{\alpha^*}} X_{v^\alpha} \cdot \mon{\alpha'}.
    \]
  \item
    Conversely, let $k = n_2 + \ell + d \le n_1 + n_2 - 1$ for $\ell, d \ge 0$.
    Let $\alpha' \in \Nilpotents_\ell(\Gamma_1)$ and let $m$ be a monomial of
    degree $d$ in $X_{V_2}$.
    Then there exists a $u$-centred animation $\alpha \in
    \Nilpotents_k(\Gamma)$ with
    \[
      \mon\alpha = X_u^{n_2} \, m\cdot \mon{\alpha'}.
    \]
  \end{enumerate}
\end{prop}
\begin{proof}
  \quad
  \begin{enumerate}
  \item
    Let $v\in \Dom(\alpha)$.
    Then $v \in V_1$ or $v\in V_2$.
    If $v \in V_2$, then $v^\alpha = u$ and since $\card{V_2\cap \Dom(\alpha)}
    = n_2$, this contributes the factor $X_u^{n_2}$.
    If $v\in V_1$, then either $v \in V_1'$ (if $v^\alpha\in V_1$) or $v\in
    V_2^{\alpha^*}$ (if $v^\alpha \in V_2$).
  \item
    Let $y(1),\dotsc,y(d) \in V_2$ (not necessarily distinct) with
    $m = X_{y(1)}\dotsb X_{y(d)}$.
    By Lemma~\ref{lem:nilpotent_number_maximals},
    the number of $\lte_{\alpha'}$-maximal elements of $V_1$ is $n_1 - \ell$.
    Since $k = n_2 + \ell + d \le n_1 + n_1 - 1$ we have $n_1 - \ell \ge
    d+1$.
    Hence, there are distinct $\lte_{\alpha'}$-maximal elements
    $x(1),\dotsc,x(d+1) \in V_1$.
    We may assume that $u \lte_{\alpha'} x(d+1)$.
    Let $\alpha_0 \in \Nilpotents_{n_2+\ell}(\Gamma)$ be defined by
    \[
      v^\alpha = \begin{cases} v^{\alpha'}, & \text{if }v\in V_1,\\
        u, & \text{if }u \in V_2.\end{cases}
    \]
    Note that $x(1),\dotsc,x(d+1)$ are $\lte_{\alpha_0}$-maximal with
    $y(i) \lte_{\alpha_0} u \lte_{\alpha_0} x(d+1)$ for $i=1,\dotsc,d$.
    Let $\alpha = \alpha_0[x(1)\subs y(1),\dotsc,x(d)\subs y(d)]$.
    Repeated application of Lemma~\ref{lem:redefine_nilpotent} shows that
    $\alpha \in \Nilpotents_k(\Gamma)$.
    By construction, $\mon\alpha = X_u^{n_2} m \mon{\alpha'}$.
    \qedhere
  \end{enumerate}
\end{proof}

\subsection{Proof of Theorem~\ref{thm:centred_main}}

Let $\Gamma = (V,E)$ be a loopless graph.
Let $u \in V$.
Suppose that $V = V_1 \sqcup V_2$ such that $V_1\not= \emptyset\not= V_2$ and
$v_1 \sim v_2$ for all $v_1 \in V_1$ and $v_2 \in V_2$.
Without loss of generality, suppose that $u \in V_1$.
Write $n_i = \card{V_i}$.
For $\alpha \in \Nilpotents(\Gamma)$, let
\begin{align*}
  R(\alpha) & = \#\{v \in V_2 : v^\alpha \in V_2\}, \\
  L_u(\alpha) & = \#\{v \in V_2 : v^\alpha \in V_1 \setminus \{ u \} \},
                \text{ and}\\
  M(\alpha) & = \#\{v \in V_2 : v^\alpha = u\}.
\end{align*}

Given $\alpha \in \Nilpotents(\Gamma)$ and $v\in V$, recall that
$\last_\alpha(v)$ denotes the unique $\lte_\alpha$-maximal element of $V$ with
$v \lte_\alpha \last_\alpha(v)$ (see \S\ref{ss:animations_order_vertices}).
As we will see below, Theorem~\ref{thm:centred_main} will follow from an
explicit procedure which, starting with an initial animation $\alpha \in
\Nilpotents_k(\Gamma)$,  minimises $R(\alpha)$ and $L_u(\alpha)$ and which
maximises $M(\alpha)$.
It is based on the following three lemmas.

\begin{lemma}
  \label{lem:minimise_R}
  Let $\alpha \in \Nilpotents_k(\Gamma)$.
  Then there exists $\beta \in \Nilpotents_k(\Gamma)$
  with $\mon\beta \lte_u \mon\alpha$ and $R(\beta) = 0$.
\end{lemma}
\begin{proof}
  Suppose that $R(\alpha) > 0$, say $v^\alpha = w$ for $v,w\in V_2$.
  By induction, it suffices to find $\beta \in \Nilpotents_k(\Gamma)$ with
  $\mon\beta \lte_u \mon\alpha$ and $R(\beta) < R(\alpha)$.
  
  If $u \notlte_\alpha v$, then Lemma~\ref{lem:redefine_nilpotent} shows that
  we may simply take $\beta = \alpha[v \subs u] \in \Nilpotents_k(\Gamma)$,
  which indeed satisfies $\mon\beta \lte_u\mon\alpha$ and $R(\beta) <
  R(\alpha)$.
  We may thus assume $u \lte_\alpha v$.

  Let $z = \last_\alpha(w)$.
  Suppose that $z\in V_2$ so that $u \sim z$.
  Let $\beta = \alpha[v \subs \udef, z \subs u]$.
  By Lemma~\ref{lem:reroute_nilpotent},
  $\beta \in \Nilpotents_k(\Gamma)$ and $\mon\beta\lte_u\mon\alpha$.
  By construction, $R(\beta) < R(\alpha)$.
  
  Suppose that $z\in V_1$ so that $w \not= z$.
  In this final case, the distinguished vertex $u$ plays no role.
  Thus, as $w \lt_\alpha z$, there exists $y\in V$ with
  $v \cover_\alpha w \lte_\alpha y \cover_\alpha z$.
  Since $v,w\in V_2$ and $z\in V_1$, we have $v \sim z \sim w$.
  We therefore obtain $\beta = \alpha[v\subs z, z \subs w, y \subs \udef] \in
  \Nilpotents_k(\Gamma)$ as in Lemma~\ref{lem:bypass_nilpotent}.
  In particular, $\mon\alpha = \mon\beta$ and clearly also $R(\beta) < R(\alpha)$.
\end{proof}

\begin{lemma}
  \label{lem:minimise_Lu}
  Let $\alpha \in \Nilpotents_k(\Gamma)$ with $R(\alpha) = 0$.
  Then there exists $\beta \in \Nilpotents_k(\Gamma)$ with
  $\mon\beta \lte_u \mon\alpha$ and such that $L_u(\beta) = R(\beta) = 0$.
\end{lemma}
\begin{proof}
  Suppose that $R(\alpha) = 0$ and $L_u(\alpha) > 0$.
  It suffices to find $\beta \in \Nilpotents_k(\Gamma)$ with
  $\mon\beta \lte_u \mon\alpha$, $R(\beta) = 0$, and
  $L_u(\beta) < L_u(\alpha)$.

  Let $v \in V_2$ and $w \in V_1\setminus\{u\}$ with $v^\alpha = w$.
  Note that $u \sim v$ and $u \not= v$.
  If $u \notlte_\alpha v$, then by Lemma~\ref{lem:redefine_nilpotent}, we may
  simply take $\beta = \alpha[v \subs u]$.
  Thus, suppose that $u \lte_\alpha v$.
  Since $u \not= v$, there exists $i \in V$ with
  $u\lte_\alpha i \cover_\alpha v$.
  Let $z = \last_\alpha(w)$.
  Then
  \[
    u \lte_\alpha i \cover_\alpha v \cover_\alpha w \lte_\alpha z
  \]

  Suppose that $z \in V_2$ so that $u \sim z$.
  Let $\beta = \alpha[v \subs \udef, z \subs u]$.
  By Lemma~\ref{lem:reroute_nilpotent},
  $\beta \in \Nilpotents_k(\Gamma)$ and $\mon\beta\lte_u\mon\alpha$.
  By construction, $R(\beta) = R(\alpha) = 0$ and $L_u(\beta) < L_u(\alpha)$.

  Suppose that $z \in V_1$ so that $u \sim v \sim z$.
  Let $\beta = \alpha[i \subs \udef, v\subs u, z \subs v]$.
  By Lemma~\ref{lem:unit_bypass_nilpotent},
  $\beta \in \Nilpotents_k(\Gamma)$ and $\mon\beta\lte_u\mon\alpha$.
  Clearly, $R(\beta) = R(\alpha) = 0$ and $L_u(\beta) < L_u(\alpha)$.
\end{proof}

\begin{lemma}
  \label{lem:maximise_M}
  Let $\alpha \in \Nilpotents_k(\Gamma)$ with $L_u(\alpha) = R(\alpha) = 0$.
  Then there exists $\beta \in \Nilpotents_k(\Gamma)$ with
  $\mon\beta \lte_u \mon\alpha$, $L_u(\beta) = R(\beta) = 0$, and
  $M(\beta) = \min(k,n_2)$.
\end{lemma}
\begin{proof}
  First suppose that $k\le n_2$.
  By Example~\ref{ex:small_centred}, there exists
  $\beta \in \Nilpotents_k(\Gamma)$ with $\mon\beta = X_u^k$,
  $L_u(\beta) = R(\beta) = 0$ and $M(\beta) = k$.
  Clearly, $\mon\beta \lte_u \mon\alpha$.
  Henceforth, let $k \ge n_2$.
  Suppose that $L_u(\alpha) = R(\alpha) = 0$ but $b := M(\alpha) < n_2$.
  It suffices to find $\beta\in\Nilpotents_k(\Gamma)$ with $\mon\beta \lte_u
  \mon\alpha$, $L_u(\beta) = R(\beta) = 0$, and $M(\beta) > M(\alpha)$.
  As $L_u(\alpha) = R(\alpha) = 0$, we have
  $b = M(\beta) = \card{\Dom(\alpha) \cap V_2}$.
  Let $y_1,\dotsc,y_b \in V_2$ be distinct with $y_i^\alpha = u$ for
  $i=1,\dotsc,b$.
  Since $b < n_2$, there exists $z \in V_2 \setminus\{y_1,\dotsc,y_b\}$.
  Note that $L_u(\alpha) = R(\alpha) = 0$ and $z \not= y_i$ for
  $i=1,\dotsc,b$ together imply that $z^\alpha = \udef$.

  Suppose that $u \notlte_\alpha z$.
  Since $b < n_2 \le k = \deg(\alpha)$, there exists $v \in V_1$ with
  $v^\alpha\not= \udef$.
  Let $\beta = \alpha[v\subs \udef, z\subs u]$.
  By Lemma~\ref{lem:redefine_nilpotent} (applied to $\alpha[v \subs \udef]$),
  $\beta \in \Nilpotents_k(\Gamma)$.
  We clearly have $\mon\beta\lte_u\mon\alpha$, $L_u(\beta) = R(\beta) = 0$ and
  $M(\beta) > M(\alpha)$.

  Suppose that $u \lte_\alpha z$.
  In this case, we may apply Lemma~\ref{lem:reroute_nilpotent} with $u = v$ to
  obtain $\beta = \alpha[u \subs \udef, z \subs u] \in \Nilpotents_k(\Gamma)$
  with $\mon\beta \lte_u \mon\alpha$.
  Clearly, $L_u(\beta) = R(\beta) = 0$ and $M(\beta) > M(\alpha)$.
\end{proof}

\begin{proof}[{Proof of Theorem~\ref{thm:centred_main}}]
  Recall that $\Gamma_i = (V_i,E_i)$.
  Without loss of generality, we may assume that $u\in V_1$.
  (Hence, in the setting of \S\ref{ss:plumbing_setup}, $(i,j) = (1,2)$.)
  By applying Lemmas~\ref{lem:minimise_R}--\ref{lem:maximise_M} in succession,
  we obtain $\beta \in \Nilpotents_k(\Gamma)$
  such that $\mon\beta \lte_u \mon\alpha$, $L_u(\beta) = R(\beta) = 0$, and
  $M(\beta) = \min(k, n_2)$.
  We conclude that if $v\in V_2$, then $v^\beta \in \{ u,\udef\}$
  and $\card{\Dom(\beta) \cap V_2} = \min(k,n_2)$.
  Hence, $\beta$ is $u$-centred.
\end{proof}

\subsection{Bivariate monomials as minors}

Let $\Gamma = \Gamma_1 \join \Gamma_2$
as in \S\ref{ss:plumbing_setup}.
As before, let $V = V_1 \sqcup V_2$ and $n = n_1 + n_2$.

\begin{lemma}
  For $i=1,2$, let $u_i \in V_i$.
  Let $0\le k \le n-1$.
  Then there are $e_1,e_2\ge 0$ with $e_1 + e_2 = k$ such that there exists
  $\alpha \in \Nilpotents_k(\Gamma)$ with $\mon\alpha = X_{u_1}^{e_1}
  X_{u_2}^{e_2}$. 
\end{lemma}
\begin{proof}
  Suppose, without loss of generality, that $n_1 \le n_2$.
  If $k \le n_2$, then we simply choose distinct elements $v_1,\dotsc,v_k \in
  V_2$ and let $\alpha$ be given by $v_i^\alpha = u_1$ for $i=1,\dotsc,k$.
  In this case, $\mon\alpha = X_{u_1}^k$.
  Thus, let $n_2 \le k < n = n_1 + n_2$ so that $k-n_2 < n_1$.
  Let $V_2 = \{ v_1,\dotsc,v_{n_2}\}$ and let $w_1,\dotsc,w_{k-n_2}$ be distinct
  elements of $V_1\setminus\{u_1\}$.
  Define $\alpha \in \Nilpotents_k(\Gamma)$ via
  $v_i^\alpha = u_1$ for $i=1,\dotsc,n_2$ and $w_j^\alpha = u_2$
  for $j=1,\dotsc,k-n_2$.
  Then $\mon\alpha = X_{u_1}^{n_2} X_{u_2}^{k-n_2}$.
\end{proof}

The following observation will be the key to computing the first summand on
the right-hand side of \eqref{eq:3part_int} in
Lemma~\ref{lem:units_everywhere_integral}.

\begin{cor}
  \label{cor:double_unit}
  Let $x\in \fO V$.
  Let $u_i \in V_i$ for $i = 1,2$ and suppose that $x_{u_1}, x_{u_2} \in
  \fO^\times$.
  Then for each $k$ with $0 \le k \le n_1 + n_2 - 1$, there exists
  $\alpha\in\Nilpotents_k(\Gamma)$ with $\mon\alpha(x) \in \fO^\times$.
  \qed
\end{cor}

\section{Adding generic rows (or columns) to matrices of linear forms}
\label{s:generic_rows}

In this section, we study the effect of adding generic rows to matrices of
linear forms on the minors of their $\circ$-duals.
Our work here will play a crucial role in our proof of Theorem~\ref{thm:join}
in \S\ref{s:proof_join}; see, in particular, \S\ref{ss:survivalism}.
Apart from providing tools to be used in our proof of Theorem~\ref{thm:join},
we also obtain the following result of potential independent interest. 

\begin{thm}
  \label{thm:add_generic_row}
  Let $\fO$ be a compact \DVR{} with residue field of size $q$.
  Let $U$ be a finite set.
  Let $A \in \Mat_{n\times m}(\fO[X_U])$ be a matrix of linear forms.
  Let $\tilde U$ be obtained from $U$ by adding~$m$ further symbols.
  Let $\tilde A \in \Mat_{(n+1)\times m}(\fO[X_{\tilde U}])$ be a matrix of
  linear forms obtained from $A$ by adding a row populated (in some order)
  with the variables attached to the aforementioned symbols from $\tilde U
  \setminus U$.
  Then
  $\Zeta^{\ask}_{\tilde A/\fO}(T) =
    \Zeta^{\ask}_{A/\fO}(T) \cdot \frac{1-q^{n-m}T}{1-q^{n-m+1}T}$.
\end{thm}

\begin{rem}
  Theorem~\ref{thm:add_generic_row} generalises the first part of \cite[Prop.\
  5.24]{cico}, which establishes the special case that $A = \sA_\Eta$ for a
  hypergraph $\Eta$.
\end{rem}

For the sake of completeness, we note that the effect of adding a generic new
\itemph{column} rather than row to a matrix of linear forms is easily deduced
from Theorem~\ref{thm:add_generic_row}.

\begin{cor}
  Let the notation be as in Theorem~\ref{thm:add_generic_row}.
  Let $\undertilde U$ be obtained from~$U$ by adding $n$ symbols.
  Let $\undertilde A \in \Mat_{n \times (m+1)}(\fO[X_{\undertilde U}])$ be
  obtained from $A$ by adding a column containing variables
  attached to symbols from $\undertilde U \setminus U$.
  Then $\Zeta^{\ask}_{\undertilde A/\fO}(T) = \Zeta^{\ask}_{A/\fO}(q^{-1}T)
  \cdot \frac{1-q^{-1}T}{1-T}$.
\end{cor}
\begin{proof}
  Let $B$ be any $d\times e$ matrix of linear forms over $\fO$.
  Then \cite[Lem.\ 2.4]{ask} yields
  $\Zeta^{\ask}_{B/\fO}(T) = \Zeta^{\ask}_{B^\top\!/\fO}(q^{d-e} T)$.
  Applying Theorem~\ref{thm:add_generic_row} to $A^\top$, we obtain
  $\Zeta^{\ask}_{\widetilde{A^\top}/\fO}(T) =
  \Zeta^{\ask}_{A^\top/\fO}(T) \cdot \frac{1 - q^{m-n}T}{1-q^{m-n+1}T} =
  \Zeta^{\ask}_{A/\fO}(q^{m-n}T) \cdot \frac{1 - q^{m-n}T}{1-q^{m-n+1}T}$.
  We may identify $\widetilde{A^\top} = \undertilde{A}^\top$.
  Thus,
  \[
    \Zeta^{\ask}_{\undertilde A/\fO}(T) =
    \Zeta^{\ask}_{\undertilde A^\top/\fO}(q^{n-m-1}T ) =
    \Zeta^{\ask}_{\widetilde{A^\top}/\fO}(q^{n-m-1}T) 
    = \Zeta^{\ask}_{A/\fO}(q^{-1}T) \cdot \frac{1-q^{-1}T}{1-T}.
    \pushQED{\qed}
    \qedhere
    \]
\end{proof}

\subsection{Some minor matrix manipulations}
\label{ss:matrix_shenanigans}

Let $R$ be a ring.
The group $\GL_n(R)\times \GL_m(R)$ acts on $\Mat_{n\times m}(R)$ via
$A.(U,V) = U^{-1} A V$
for $A\in \Mat_{n\times m}(R)$, $U\in \GL_n(R)$, and $V\in \GL_m(R)$.
Let $\simeq$ denote the corresponding equivalence relation on $\Mat_{n\times
  m}(R)$.

\begin{lemma}
  \label{lem:edt_extra}
  Let $A,B\in \Mat_{n\times m}(R)$ with $A\simeq B$.
  Let $r \in R$.
  Then
  \[
    \begin{bmatrix} A \\ r 1_m\end{bmatrix}
    \simeq
    \begin{bmatrix} B \\ r 1_m\end{bmatrix}.
  \]
\end{lemma}
\begin{proof}
  Let $B = UAV$ for $U\in \GL_n(R)$, $V\in \GL_m(R)$.
  Then 
  $\left[\begin{smallmatrix} U & 0 \\ 0 & V^{-1}\end{smallmatrix}\right]
  [\begin{smallmatrix} A \\r1_m\end{smallmatrix}]
  V = [\begin{smallmatrix} b \\ r 1_m\end{smallmatrix}]$.
\end{proof}

Recall that for a matrix $A$ over $R$, we write $\Minors_m (A)$ for the ideal of
$R$ generated by the $k\times k$ minors of $A$.

\begin{lemma}
  \label{lem:edt_gauss}
  Let $A \in \Mat_{n\times m}(R)$.
  Let $r\in R$.
  Define
  $\widetilde A = [\begin{smallmatrix}
      A \\ r 1_m
    \end{smallmatrix}]
    \in \Mat_{(n+m) \times m}(R)$. 
  Let $0 \le k \le m$.
  Then
  $\Minors_k(\widetilde A) = \sum\limits_{i=0}^k r^i \Minors_{k-i}(A)$.
\end{lemma}
\begin{proof}
  Let $I\subset \{1,\dotsc,n+m\}$ and $J\subset \{1,\dotsc,m\}$ with $\card I
  = \card J = k$.
  Let $\widetilde A[I \mid J]$ be the submatrix of $\widetilde A$ consisting of the
  rows indexed by elements of $I$ and the columns indexed by elements of $J$;
  we use analogous notation for submatrices of $A$.
  Let $I' = \{i - n : i \in I\} \cap \{ 1,\dotsc,m\}$.
  If $I'\not\subset J$, then $\widetilde A[I\mid J]$ contains a zero row
  whence $\det(\widetilde A[I\mid J]) = 0$.
  Thus, suppose that $I'\subset J$.
  Let $i = \card{I'} \le m$ and note that
  \[
    \det (\widetilde A[I\mid J]) = r^i \det(\widetilde A[I\setminus I' \mid J\setminus
    I']
    = r^i \det(A[I\setminus I'\mid J\setminus I']).
  \]
  This shows that every nonzero $k\times k$ minor of $\widetilde A$ is of the form
  $r^i m$ for $0 \le i \le k$ and a $(k-i)\times (k-i)$ minor $m$ of $A$.
  Conversely, by reversing our reasoning, we find that each such element 
  $r^i m$ arises as a minor of $\widetilde A$.
\end{proof}

\subsection{Reminder: elementary divisors and minors}
\label{ss:edt}

Recall that $\fO$ denotes a compact \DVR{} with maximal ideal $\fP$, residue
field size $q$, uniformiser $\pi \in \fP \setminus \fP^2$,
normalised valuation $\ord$, and field of fractions~$K$.
Let $A \in \Mat_{n\times m}(\fO)$ have rank $r$ over~$K$.
Using the structure theory of modules over the local \PID{} $\fO$,
we obtain a unique $\lambda(A) = (\lambda_1(A),\dotsc,\lambda_r(A))$ such that
$0 \le \lambda_1(A) \le \dotsb \le \lambda_r(A) < \infty$ and
$A \simeq
\left[\begin{smallmatrix} \diag(\pi^{\lambda(A)}) & 0 \\ 0 & 0
  \end{smallmatrix}\right]$,
where $\diag(\pi^{\lambda(A)}) := \diag(\pi^{\lambda_1(A)},\dotsc,\pi^{\lambda_r(A)})$.
The $\pi^{\lambda_i(A)}$ are the elementary divisors (and invariant factors) of $A$.

\begin{lemma}[{Cf.\ \cite[Lemma~4.6(ii)]{ask} or \cite[\S 2.2]{Vol10}}]
  \label{lem:edt_via_minors}
  Let $U$ be a finite set.
  Let $A = A(X_U) \in \Mat_{n\times m}(\fO[X_U])$ have rank $r$ over $K(X_V)$.
  Let $x \in \fO U$ with $\rank_K(A(x)) = r$.
  Let $z \in \fO \setminus \{0\}$.
  Then for $i = 1, \dotsc, r$, we have
  \[
    \frac{
      \norm{\Minors_{i-1}(A(x))}}
    {\norm{\Minors_i(A(x)) \cup z \Minors_{i-1}(A(x))}}
    = q^{\min(\lambda_i(A(x)), \ord(z))}.
  \]
\end{lemma}

We note that the condition $\rank_K(A(x)) = r$ on $x$ is satisfied outside of
a null set with respect to the normalised Haar measure $\mu$ on $\fO U$.

\subsection{Adding generic rows: elementary divisors of $\circ$-duals}
\label{ss:generic_row_circ_dual}

Let $R$ be a ring.
Let $U$ and $V$ be finite sets with $\ell$ and $n$ elements, respectively.
Write $U = \{ u_1,\dotsc,u_\ell \}$ and $V = \{ v_1,\dotsc,v_n \}$.
For $1 \le i \le n$, $1 \le j \le m$, and $1 \le k \le \ell$, let
$\alpha_{ijk} \in R$.
Define $A(X_U) \in \Mat_{n\times m}(R[X_U])$ and
$C(X_V) \in \Mat_{\ell \times m}(R[X_V])$
via $A(X_U)_{ij} = \sum_{k=1}^\ell \alpha_{ijk} X_{u_k}$
and $C(X_V)_{kj} = \sum_{i=1}^n \alpha_{ijk} X_{v_i}$
as in \eqref{eq:A_circ_C}.
Hence, $A(X_U)$ and $C(X_V)$ are $\circ$-duals of each other.

Let $d \ge 1$.
Let $\widetilde U = U \sqcup \{ g_{rs} : 1\le r \le d, 1\le j \le m\}$, where
the $g_{rj}$ are distinct.
Let $\widetilde V = V \sqcup \{ w_1,\dotsc,w_d\}$, where the the $w_s$ are
distinct.
Write $G(d,m) = [X_{g_{rj}}] \in \Mat_{d\times m}(R[X_{\widetilde U}])$;
hence, $G(d,m)$ is a generic $d\times n$ matrix in variables distinct from the
$X_u$ ($u\in U$).
Define
\begin{align*}
  \widetilde A(X_{\widetilde U}) & =
  \left[
    \begin{array}{c}
      A(X_U) \\
      \hline
      G(d,m)
    \end{array}
  \right]
  \in \Mat_{(n+d)\times m}(R[X_{\widetilde U}]) \quad\text{and}\\
  \widetilde C(X_{\widetilde V}) & =
  \left[
    \begin{array}{c}
      C(X_V) \\
      \hline
      X_{w_1} 1_{m} \\
      \vdots \\
      X_{w_d} 1_{m}
    \end{array}
  \right]
  \in \Mat_{(\ell + dm)\times m}(R[X_{\widetilde V}]).
\end{align*}

\begin{lemma}
  $\widetilde C(X_{\widetilde V})$ is a $\circ$-dual of $C(X_V)$.
\end{lemma}
\begin{proof}
  Ordering the elements of $\widetilde U$ and $\widetilde V$ as
  $u_1,\dotsc,u_\ell,g_{11},\dotsc,g_{1m}, \dotsc, g_{d1},\dotsc,g_{dm}$
  and
  $v_1,\dotsc,v_n,w_1,\dotsc,w_d$,
  respectively, the claim follows by inspection.
\end{proof}

It turns out that the elementary divisors of specialisations of $\widetilde
C(X_{\widetilde V})$ can be easily expressed in terms of those of specialisations
of $C(X_V)$.

\begin{lemma}
  \label{lem:edt_tilde_C}
  Let $\fO$ be a compact \DVR{} with an $R$-algebra structure.
  Let $r$ be the rank of $C(X_V)$ over $K$.
  Let $(x,y) \in \fO \widetilde V = \fO V \oplus \fO (\widetilde V\setminus V)$
  with $\rank_K(C(x)) = r$ and $\prod\limits_{i=1}^d y_i \not= 0$.
  Then
  \[
    \lambda_i(\widetilde C(x,y)) =
    \begin{cases}
      \min\bigl(\lambda_i(C(x)), \ord(y_{w_1}),\dotsc,\ord(y_{w_d})\bigr), & \text{for }
      i=1,\dotsc,r,\\
      \min\bigl(\ord(y_{w_1}),\dotsc,\ord(y_{w_d})\bigr), & \text{for } i=r+1,\dotsc,m.
    \end{cases}
  \]
\end{lemma}
\begin{proof}
  Note that $\widetilde C(X_{\widetilde V})$ has rank $m$ over $K(X_{\widetilde V})$.
  For $(x,y)$ satisfying the conditions in the lemma,
  we have
  $\rank_K(\widetilde C(x,y)) = m$,
  \begin{align*}
    C(x) & \simeq \begin{bmatrix}
      \diag(\pi^{\lambda(C(x))}) & 0 \\ 0 & 0
    \end{bmatrix} \in \Mat_{\ell \times m}(\fO),
  \text{ and}\\
  \widetilde C(x,y) & \simeq \begin{bmatrix}
      \diag(\pi^{\lambda(\widetilde C(x,y))}) & 0 \\ 0 & 0
    \end{bmatrix} \in \Mat_{(\ell+dm)\times m}(\fO).
  \end{align*}
  
  On the other hand,
  by Lemma~\ref{lem:edt_extra}, we may apply
  elementary row operations to the $(\ell + dm)\times m$ matrix $\widetilde
  C(x,y)$ to obtain

  {\small
  \begin{align*}
    & \widetilde C(x,y) \simeq
    \left[
      \begin{array}{c}
      \begin{array}{cc}
        \diag(\pi^{\lambda(C(x))})
        & 0 \\
        0 & 0
      \end{array}\\
        \hline
        y_{w_1} 1_m\\
        \vdots \\
        y_{w_d} 1_m
      \end{array}
    \right]
    \simeq
    \left[
    \begin{array}{c|c}
      \diag(\pi^{\min(\lambda_1(C(x)),e)}, \dotsc,\pi^{\min(\lambda_r(C(x)),e)})
      \\\hline
        &
          \pi^e 1_{m-r}
      \\\hline
      0 & 0
      \end{array}
    \right],
  \end{align*}}%
  where $e = \min(\ord(y_{w_1}),\dotsc,\ord(y_{w_d}))$.
  The claim follows from the uniqueness of the $\lambda_i(\widetilde C(x,y))$;
  note that the exponents along the diagonal entries of the preceding matrix
  are nondecreasing.
\end{proof}

\begin{cor}
  \label{cor:edt_tilde_C_many_variables}
  Let the assumptions be as in Lemma~\ref{lem:edt_tilde_C}.
  Let $z\in \fO \setminus \{0\}$.
  Let $w\in \fO$ with $\abs w = \norm{y_{w_1},\dotsc,y_{w_d}, z}$.
  Then
  \[
    \frac
    {\norm{\Minors_{i-1}(\widetilde C(x,y))}}
    {\norm{\Minors_i(\widetilde C(x,y)) \cup z \Minors_{i-1}(\widetilde C(x,y))}}
    =
    \begin{cases}
      \frac{\norm{\Minors_{i-1}(C(x))}}
      {\norm{\Minors_i(C(x)) \, \cup \, w \Minors_{i-1}(C(x))}}, & \text{for }
      i=1,\dotsc,r, \\
      \abs{w}^{-1}, & \text{for } i=r+1,\dotsc, m.
    \end{cases}
  \]
  Moreover, if $\abs w = 1$, then
  $\frac {\norm{\Minors_{i-1}(\widetilde C(x,y))}}
  {\norm{\Minors_i(\widetilde C(x,y)) \cup z \Minors_{i-1}(\widetilde C(x,y))}} = 1$.
\end{cor}
\begin{proof}
  Let $\Lambda = \min(\lambda_i(\widetilde C(x,y), \ord(z))$.
  By Lemma~\ref{lem:edt_via_minors}, the left-hand side
  of the displayed equation in Corollary~\ref{cor:edt_tilde_C_many_variables} 
  is $q^\Lambda$.
  By combining Lemmas~\ref{lem:edt_via_minors}
  and \ref{lem:edt_tilde_C}, we find that $q^\Lambda$ coincides with the
  right-hand side.
\end{proof}

\subsection[Proof of Theorem 10.1]{Proof of Theorem~\ref{thm:add_generic_row}}

While it is possible to prove Theorem~\ref{thm:add_generic_row}
by combining Corollary~\ref{cor:edt_tilde_C_many_variables} 
and the integrals in Proposition~\ref{prop:proj_circ_int},
a cleaner derivation is obtained using the zeta functions attached to
modules over polynomial rings from \cite[\S 2.6]{cico}.

Let $V$ be a finite set.
For $x\in \fO V$, we write $\fO_x$ for $\fO$ endowed with the
$\fO[X_V]$-module structure $X_v r = x_v r$ ($v\in V$, $r\in \fO$).
For a finitely generated $\fO[X_V]$-module $M$, define
\[
  \zeta_M(s) = \int\limits_{\fO V \times \fO}
  \abs{z}^{s-1} \cdot \card{M_x \otimes_\fO \fO/z} \,
  \dd\!\mu(x,z).
\]

\begin{prop}[{\cite[Cor.\ 2.15]{cico}}]
  \label{prop:zeta_M}
  Let $U$ and $V$ be finite sets with $\card{U} = \ell$
  and $\card{V} = n$.
  Let $A(X_U) \in \Mat_{n\times m}(\fO[X_U])$ be a matrix of linear forms with
  $\circ$-dual $C(X_V) \in \Mat_{\ell\times m}(\fO[X_V])$.
  Then
  $\zeta^{\ask}_{A(X_U)/\fO}(s) = (1-q^{-1})^{-1} \zeta_{\Coker(C(X_V))}(s-n+m)$.
\end{prop}

\begin{proof}[Proof of Theorem~\ref{thm:add_generic_row}]
  We work in the setting of \S\ref{ss:generic_row_circ_dual} with $d = 1$.
  Dropping a superscript, we write $\widetilde V = V \sqcup \{ w\}$ so that
  $\widetilde C(X_{\widetilde V}) = \left[\begin{smallmatrix} C(X_V) \\ X_w
    1_m\end{smallmatrix}\right] \in \Mat_{(\ell + m)\times m}(\fO[X_{\widetilde V}])$
  is a $\circ$-dual of $\widetilde A(X_{\widetilde U})$.
  Let $M = \Coker(C(X_V))$ and $\widetilde M = \Coker(\widetilde C(X_{\widetilde V}))$.
  By Proposition~\ref{prop:zeta_M}, it suffices to show that
  \begin{equation}
    \label{eq:zeta_M}
    \zeta_{\widetilde M}(s) = \frac{1-q^{-1-s}}{1-q^{-s}} \cdot \zeta_M(s+1).
  \end{equation}

  We may view $\widetilde M$ as the restriction of scalars of $M$ along the
  ring map $\fO[X_{\widetilde V}] \to \fO[X_V]$ which sends $X_w$ to $0$ and which
  fixes each $X_v$ ($v\in V$).
  We identify $\fO \widetilde V = \fO V\times \fO$, with the factor $\fO$
  corresponding to the direct summand $\fO \std w$ of $\fO \widetilde V$.
  The key observation is that
  for $(x,y) \in \fO \widetilde V$ and $z\in \fO$, we have
  $\widetilde M_{(x,y)} \otimes_\fO \fO/z \approx_{\fO} M_x \otimes_\fO
  \fO/\langle y,z\rangle$
  and thus
  \[
    \zeta_{\widetilde M}(s) =
    \int\limits_{\fO V \times \fO \times \fO}
    \abs{z}^{s-1}
    \card{M_x \otimes_\fO \fO/\langle y,z\rangle} \,\dd\!\mu(x,y,z).
  \]

  We partition $\fO V \times \fO \times \fO = W_1 \sqcup W_2$, where
  \begin{align*}
    W_1 & = \{ (x,y,z) \in \fO V \times \fO \times \fO : z \mid y\} \text{ and
          }\\
    W_2 &= \{ (x,y,z) \in \fO V \times \fO \times \fO : \pi y \mid z\}.
  \end{align*}

  To evaluate our integral over $W_1$, we perform a change of variables $y =
  zy'$ with $\abs{\dd\! y} = \abs{z} \cdot \abs{\dd\! y'}$.
  We thus find that
  \begin{align*}
    \int\limits_{W_1}
    \abs{z}^{s-1}
    \card{M_x \otimes_\fO \fO/\langle y,z\rangle} \,\dd\!\mu(x,y,z)
    & = \int\limits_{\fO V \times \fO \times \fO} \abs{z}^s \cdot \card{M_x
      \otimes_{\fO} \fO/z} \,\dd\!\mu(x,y',z) \\
    & = \zeta_M(s+1).
  \end{align*}
  Integrating over $W_2$ and changing variables via $z = \pi y z'$ and
  $\abs{\dd\! z} = q^{-1} \abs{y} \cdot \abs{\dd\! z'}$, using
  the Fubini-Tonelli theorem, we find that
  \begin{align*}
    \int\limits_{W_2}
    \abs{z}^{s-1}
    \card{M_x \otimes_\fO \fO/\langle y,z\rangle} \,\dd\!\mu(x,y,z)
    & =
      q^{-1}
      \int\limits_{\fO V \times \fO \times \fO}
      \abs{\pi y z'}^{s-1}
      \card{M_x \otimes_\fO \fO/y}
      \,
      \abs{y}
      \dd\!\mu(x,y,z') \\
    & = q^{-s}
      \int\limits_{\fO} \abs{z}^{s-1} \dd\!\mu(z)
      \cdot
      \zeta_M(s+1).
  \end{align*}
  It is well known (and easy to prove) that $\int\limits_{\fO}\abs{z}^s
  \dd\!\mu(z) = \frac{1-q^{-1}}{1-q^{-1-s}}$.
  Hence, we find that
  \[
    \zeta_{\widetilde M}(s) = \zeta_M(s+1) \left(1 + \frac{q^{-s}(1-q^{-1})}{1-q^{-s}}\right)
    = \frac{1-q^{-1-s}}{1-q^{-s}}\cdot \zeta_M(s+1),
  \]
  as claimed.
\end{proof}

\section{Proof of Theorem~\ref{thm:join} (and a new proof of Theorem~\ref{thm:cmt})}
\label{s:proof_join}

As always, $\fO$ denotes a compact \DVR{} with maximal ideal $\fP$ and residue
field size $q$.
We write $t = q^{-s}$.
In the following, we assume that $s$ is arbitrary but fixed and that
$\Real(s)$ is sufficiently large with respect to the graphs involved.

\subsection[Summary]{Summary: $W^-_\Gamma$ as an integral for a loopless graph
  $\Gamma$}

Let $\Gamma = (V,E)$ be a loopless graph with $n$ vertices and $c$ connected
components.
Prior to outlining the strategy of our proof of Theorem~\ref{thm:join},
we now summarise various results from \cite{cico} and the present article 
which provide us with a formula for $W^-_\Gamma$ in terms of
$\fP$-adic integrals.
First, by Theorem~\ref{thm:uniformity}\ref{thm:uniformity3} and
Proposition~\ref{prop:ask_Gamma_integral},

\begin{equation}
  \label{eq:loopless_WGamma_as_int}
  (1-t) W^-_\Gamma(q,t) = 1 + (1-q^{-1})^{-1}
  \int\limits_{(\fO V)^\times \times \fP}
  \!\!\!\! \Gamma^-(s).
\end{equation}

Next, by \eqref{eq:Gamma_integral} and
Corollary~\ref{cor:rank_C-_loopless}, for $W\subset \fO V \times \fO$, we have
\begin{equation}
  \label{eq:loopless_int}
  \int\limits_W \Gamma^-(s)
  =
  \int\limits_W \abs{z}^{s-c-1}
  \,
  \prod_{k=1}^{n-c}
  \frac
  {\norm{\Minors_{k-1}^-\Gamma(x)}}
  {\norm{\Minors_k^-\Gamma(x) \cup z \Minors_{k-1}^-\Gamma(x)}}
  \dd\!\mu(x,z).
\end{equation}

Corollary~\ref{cor:loopless_minus_minors} shows that
for $x\in \fO V$, we have
\begin{equation}
  \label{eq:loopless_minors}
  \norm{\Minors^-_k\Gamma(x)} = \norm{\mon\alpha(x) : \alpha \in \Nilpotents_k(\Gamma)}.
\end{equation}

Finally, Corollary~\ref{cor:rank_C-_loopless}
shows that $\Minors_k^-\Gamma$ is nonempty if and only if $0 \le k \le n-c$.

\subsection{Setup and strategy}

For the remainder of this section, let
$\Gamma_1 = (V_1,E_1)$ and $\Gamma_2 = (V_2,E_2)$ be loopless graphs and let
$\Gamma = \Gamma_1 \join \Gamma_2$ as in~\S\ref{ss:plumbing_setup}.
In particular, $\Gamma_i$ has $n_i$ vertices and $n = n_1 + n_2$ is the number
of vertices of $\Gamma$.
Note that $\Gamma$ is necessarily connected.
We identify $\fO V = \fO V_1 \times \fO V_2$ and $X_V = (X_{V_1},X_{V_2})$.
Our proof of Theorem~\ref{thm:join} is based on a series of auxiliary lemmas
and claims.
The first of these provides an equivalent form of Theorem~\ref{thm:join} in
terms of $\fP$-adic integrals.

\begin{lemma}
  \label{lem:join_via_int}
  Equation~\eqref{eq:join} holds for $\Gamma = \Gamma_1 \join \Gamma_2$ if and
  only if
  \begin{align}
    \nonumber
    \int\limits_{(\fO V)^\times \times \fP}
    \!\!\!\!\!\! \Gamma^-(s)
    & =
      \frac{(1-q^{-1})(1-q^{-n_1})(1-q^{-n_2}) qt}{1-qt}
    \\
    \nonumber
    & \quad + \frac{1-q^{1-n_2}t}{1-qt}\int\limits_{(\fO V_1)^\times\times \fP}
      \!\!\!\!\!\!\Gamma_1^-(s+n_2)
    \\
    & \quad + \frac{1-q^{1-n_1}t}{1-qt} \int\limits_{(\fO V_2)^\times\times \fP}
      \!\!\!\!\!\! \Gamma_2^-(s+n_1)
      \label{eq:join_integral_form}
  \end{align}
  holds for $\Real(s) \gg 0$.
\end{lemma}
\begin{proof}
  Using \eqref{eq:loopless_WGamma_as_int}, we may express
  $W^-_\Gamma(q,t)$ in terms of
  $\int\limits_{(\fO V)^\times\times\fP}\Gamma^-(s)$ and analogously for
  $W^-_{\Gamma_i}(q,t)$.
  The equivalence of \eqref{eq:join_integral_form} and \eqref{eq:join} then
  follows by inspection.
\end{proof}

Our proof of Theorem~\ref{thm:join} uses
\S\S\ref{s:plumbing}--\ref{s:generic_rows} 
to show that the three summands on the right-hande side of
\eqref{eq:3part_int} exactly match those on the right-hand side
of \eqref{eq:join_integral_form}.

For $(x,y,z) \in \fO V \times \fO = \fO V_1 \times \fO V_2 \times \fO$,
whenever the following fraction is defined, write
\begin{equation}
  \label{eq:Fk}
  F_k(x,y,z) := 
  \frac{\norm{\Minors^- _{k-1} \Gamma (x,y)}}
  {\norm{\Minors^-_k \Gamma(x,y) \cup z \Minors^-_{k-1} \Gamma(x,y)}}.
\end{equation}
For $k=1,\dotsc,n-1$,
by Lemma~\ref{lem:edt_via_minors}, we have $1 \le F_k(x,y,z) \le \abs{z}^{-1}$
for almost all $(x,y) \in \fO V$ and all nonzero $z\in \fO$.
Using \eqref{eq:loopless_minors}, in the following, will express $F_k(x,y,z)$
in terms of animations of $\Gamma$.

\subsection{The first summand}

We now show that the first summands on the right-hand sides of
\eqref{eq:3part_int} and \eqref{eq:join_integral_form} coincide.

\begin{lemma}
  \label{lem:units_everywhere_integral}
  $\displaystyle
  \int\limits_{(\fO V_1)^\times \times \,(\fO V_2)^\times \times \fP}
  \Gamma^-(s) =
  (1-q^{-1})(1-q^{-n_1})(1-q^{-n_2}) \frac{qt}{1-qt}$.
\end{lemma}
\begin{proof}
  Let $0 \le k \le n-1$.
  Corollary~\ref{cor:double_unit} shows that for almost all $(x,y,z) \in (\fO
  V_1)^\times \times (\fO V_2)^\times \times \fP$, the ideal
  $\Minors_k^-\Gamma$ contains a monomial $m$ such that $m(x,y) \in
  \fO^\times$.
  Hence, $F_k(x,y,z) = 1$ for $k=1,\dotsc,n-1$.
  The claim follows from
  $\int\limits_{\fP} \abs{z}^s \dd\!\mu(z) = (1-q^{-1})q^{-1}t/(1-q^{-1}t)$
  and
  $\int\limits_{(\fO V_1)^\times \times \,(\fO V_2)^\times \times \fP}
  \Gamma^-(s)
  =
  \int\limits_{(\fO V_1)^\times \times \,(\fO V_2)^\times \times \fP}
  \abs{z}^{s-2} \dd\!\mu(x,y,z)
  = \mu((\fO V_1)^\times) \mu((\fO V_2)^\times) \int\limits_{\fP}
  \abs{z}^{s-2} \dd\!\mu(z)$.
\end{proof}

\subsection{Towards the second summand: preparation}
\label{ss:survivalism}

We derive a series of auxiliary facts which will then help us show that
$\int\limits_{(\fO V_1)^\times \times \fP V_2 \times \fP} \Gamma^-(s)$ coincides
with the second summand on the right-hand side of
\eqref{eq:join_integral_form}.

We call $(x,y) \in \fO V_1\times \fO V_2$ \emph{strongly nonzero} if $x_v \not= 0
\not= y_w$ for all $v\in V_1$ and $w\in V_2$.
Almost all $(x,y) \in \fO V_1\times \fO V_2$ are strongly nonzero.

\begin{claim}
  \label{c:unit_Fk}
  Let $1 \le k \le n_2$ and
  let $(x,y) \in (\fO V_1)^\times \times \fO V_2$ be strongly nonzero.
  Then $F_k(x,y) = 1$.
\end{claim}
\begin{proof}
  Example~\ref{ex:small_centred} shows that
  as long as all components $x_v$ of $x$ are nonzero, the ideal $\Minors^-_k
  \Gamma(x,y)$ contains a unit for $0 \le k \le n_2$.
\end{proof}

\begin{claim}
  \label{c:reduce_to_Gamma_1}
  Let $0 \le e \le n_1 - 1$.
  Then for all strongly nonzero $(x,y) \in (\fO V_1)^\times \times \fO V_2$,
  we have
  \begin{align*}
    \Minors^-_{n_2 + e}\Gamma(x,y) =
    \Bigl\langle
    y^b \cdot \mon\alpha(x) :\,\, & 0 \le d \le e, \,\, b \in \NN_0 V_2 \text{ with } \sum
    b = d, \,\, \alpha \in \Nilpotents_{e-d}(\Gamma_1)
    \Bigr\rangle.
  \end{align*}
\end{claim}
\begin{proof}
  Combine Proposition~\ref{prop:ordering_monomials},
  Theorem~\ref{thm:centred_main}, and
  Proposition~\ref{prop:large_centred_minors}.
\end{proof}

Write $V_2 = \{ w(1),\dotsc,w(n_2)\}$.
Let $\Gamma_i$ have $m_i$ edges and $c_i$ connected components.
Then $\sC_{\Gamma_i}^- \in \Mat_{m_i \times n_i}(\ZZ[X_{V_i}])$ has rank
$n_i- c_i$ by Proposition~\ref{prop:rank_C-}.
Define
\[
  \widetilde \sC_{\Gamma_1} = \begin{bmatrix}
    \sC_{\Gamma_1} \\
    Y_{w(1)} 1_{n_1} \\
    \vdots \\
    Y_{w(n_2)} 1_{n_1}
  \end{bmatrix}
  \in \Mat_{(m_1 + n_1n_2)\times n_1}(\ZZ[X_V]).
\]

\begin{claim}
  \label{c:shifted_minors}
  Let $0 \le e \le n_1 - 1$.
  Then $\Minors^-_{n_2 + e}\Gamma(x,y) = \Minors_e(\widetilde \sC_{\Gamma_1}(x,y))$
  for all strongly nonzero $(x,y) \in (\fO V_1)^\times \times \fO V_2$.
\end{claim}
\begin{proof}
  This follows from Claim~\ref{c:reduce_to_Gamma_1} and $n_2$ applications of
  Lemma~\ref{lem:edt_gauss}, one for each block $Y_{w(j)} 1_{n_1}$ within
  $\widetilde \sC_{\Gamma_1}$.
\end{proof}

\begin{claim}
  \label{c:prod_F_Ctilde_minors}
  For all strongly nonzero $(x,y) \in (\fO V_1)^\times \times \fO V_2$ and nonzero
  $z\in \fO$, we have
  \begin{align*}
    \prod\limits_{k=1}^{n_1 + n_2 - 1} F_k(x,y,z)
    & =
      \prod_{e=1}^{n_1-1}
      \frac{\norm{\Minors_{e-1}(\widetilde \sC_{\Gamma_1}(x,y))}}
      {\norm{\Minors_e(\widetilde \sC_{\Gamma_1}(x,y)) \cup z
      \Minors_{e-1}(\widetilde \sC_{\Gamma_1}(x,y))}}.
  \end{align*}
\end{claim}
\begin{proof}
  This follows from Claims \ref{c:unit_Fk} and \ref{c:shifted_minors} and the
  definition of $F_k(x,y,z)$ in \eqref{eq:Fk}.
\end{proof}

Let
\[
  G_e(x,z) := 
  \frac{\norm{\Minors^- _{e-1} \Gamma_1 (x)}}
  {\norm{\Minors^-_e \Gamma_1(x) \cup z \Minors^-_{e-1} \Gamma_1(x)}}.
\]

Analogously to the case of $F_k(x,y,z)$,
for $e=1,\dotsc,n_1-c_1$,
by Lemma~\ref{lem:edt_via_minors}, we have $1 \le G_e(x,z) \le \abs{z}^{-1}$
for almost all $x \in \fO V_1$ and all nonzero $z\in \fO$.

\begin{claim}
  \label{c:Gamma_int_Gk}
  For each measurable set $W\subset (\fO V_1)^\times \times \fP V_2 \times \fP$,
  we have
  \[
    \int\limits_W
    \Gamma^-(s)
    =
    \int\limits_W
    \abs{z}^{s-2}
    \norm{y;z}^{1-c_1}
    \!\!
    \prod\limits_{e=1}^{n_1-c_1}
    G_e(x, \gcd(y;z))
    \,
    \dd\!\mu(x,y,z).
  \]
\end{claim}
\begin{proof}
  Recall that $\Gamma = \Gamma_1 \join \Gamma_2$.
  In particular, $\Gamma$ is connected.
  Using \eqref{eq:loopless_int} (with $c = 1$) and
  Claim~\ref{c:prod_F_Ctilde_minors},
  we obtain
  \[
    \int\limits_W \Gamma^-(s)
    = \int\limits_W \abs{z}^{s-2}
    \prod\limits_{e=1}^{n_1-c_1}
    \frac{\norm{\Minors_{e-1}(\widetilde \sC_{\Gamma_1}(x,y))}}
    {\norm{\Minors_e(\widetilde \sC_{\Gamma_1}(x,y)) \cup z
        \Minors_{e-1}(\widetilde \sC_{\Gamma_1}(x,y))}}
    \,
    \dd\!\mu(x,y,z).
  \]
  We rewrite the $e$-indexed factors in the preceding integrand by
  applying Corollary~\ref{cor:edt_tilde_C_many_variables} with
  $C = \sC_{\Gamma_1}$, $r = n_1 - c_1$, $d = n_2$, and $\widetilde C = \widetilde
  \sC_{\Gamma_1}$.
  Writing $g = \gcd(y;z)$ so that $\abs g = \norm{y;z}$, this yields
  \begin{align*}
    &
      \frac{\norm{\Minors_{e-1}(\widetilde \sC_{\Gamma_1}(x,y))}}
      {\norm{\Minors_e(\widetilde \sC_{\Gamma_1}(x,y)) \cup z \Minors_{e-1}(\widetilde
      \sC_{\Gamma_1}(x,y))}}
    \\
    & \quad
      =
      \begin{cases}
        \frac{\norm{\Minors_{e-1}(\sC_{\Gamma_1}(x))}}
        {\norm{\Minors_e(\sC_{\Gamma_1}(x)) \, \cup \, g \Minors_{e-1}(\sC_{\Gamma_1}(x))}}, & \text{for }
        e=1,\dotsc,n_1-c_1, \\
        \abs{g}^{-1}, & \text{for } e=n_1-c_1+1,\dotsc, n_1-1
      \end{cases}
  \end{align*}
  for almost all $(x,y)$ and all nonzero $z$.
  The claim then follows readily.
\end{proof}

We also require the following technical and elementary lemma.

\begin{lemma}
  \label{lem:repeated_dummy_variables}
  Let $g\colon [0,\infty) \to [0,\infty)$ be measurable.
  Suppose that for some $N\ge 0$ and all nonzero $y\in \fP$, we have $g(\abs{y})
  \le \abs{y}^{-N}$.
  For $d \ge 1$, let $F_d(s) =
  \int\limits_{\fP^d}\norm{y}^sg(\norm{y})\dd\!\mu(y)$.
  Then for $s\in \CC$ with $\Real(s) \ge N$,
  we have
  $F_d(s) = \frac{ 1-q^{-d}}{1-q^{-1}}F_1(s + d-1)$.
\end{lemma}
\begin{proof}
  We proceed by induction on $d$, the case $d = 1$ being clear.
  Suppose that the claim holds for some value of $d$.
  We partition the domain of integration $\fP^{d+1}$ in the definition of
  $F_{d+1}(s)$ as $R\sqcup S$, where
  $R = \{ y\in \fP^{d+1} : \norm{y_1,\dotsc,y_d} \le \abs{y_{d+1}}\}$
  and $S = \{ y\in \fP^{d+1} : \abs{y_{d+1}} < \norm{y_1,\dotsc,y_d}\}$.

  On $R$, we may write $y_i = y_{d+1} y_i' $ for $y_i'\in \fO$ and
  $i=1,\dotsc,d$.
  A change of variables
  using $\abs{\dd\!y_i} = \abs{y_{d+1}} \abs{\dd\!{y_i'}}$ for $i=1,\dotsc,d$ yields
  \begin{align*}
    \int\limits_R \norm{y}^s g(\norm y) \dd\!\mu(y)
    & =
    \int\limits_R \abs{y_{d+1}}^s g(\abs{y_{d+1}}) \dd\!\mu(y) \\
    & =
    \int\limits_{\fO^d \times \fO} \abs{y_{d+1}}^{s+d} g(\abs{y_{d+1}})
    \dd\!\mu(y_1',\dotsc,y_d',y_d)
    = F_1(s+d).
  \end{align*}

  On $S$, write $y_{d+1} = \pi \gcd(y_1,\dotsc,y_d) y_{d+1}'$ for
  $y_{d+1}'\in \fO$.
  (Recall that $\pi\in \fP\setminus\fP^2$ denotes a uniformiser of $\fO$.)
  A change of variables using $\abs{\dd\!{y_{d+1}}} = q^{-1}
  \norm{y_1,\dotsc,y_d} \abs{\dd\!{y_{d+1}'}}$ yields
  \begin{align*}
    \int\limits_S \norm{y}^s g(\norm y) \dd\!\mu(y)
    & =
      \int\limits_S \norm{y_1,\dotsc,y_d}^s g(\norm{y_1,\dotsc,y_d})
      \dd\!\mu(y)\\
    & = q^{-1} \int\limits_{\fO^d\times \fO}\norm{y_1,\dotsc,y_d}^{s+1}
      g(\norm{y_1,\dotsc,y_d}) \dd\!\mu(y_1,\dotsc,y_d,y_{d+1}')
    \\
    & =
      q^{-1} F_d(s+1) = q^{-1} \frac{1-q^{-d}}{1-q^{-1}} F_1(s+d).
  \end{align*}
  Hence,
  $F_{d+1}(s) = (1 + q^{-1} \frac{1-q^{-d}}{1-q^{-1}}) F_1(s+d) =
  \frac{1-q^{-(d+1)}}{1-q^{-1}}F_1(s+d)$ as claimed.
\end{proof}

\subsection{The second and third summand and a proof of Theorem~\ref{thm:join}}

We now deal with the second (and, by symmetry, the third) summand in
\eqref{eq:3part_int}.

\begin{lemma}
  \label{lem:left_unit_integral}
  $\displaystyle
  \int\limits_{(\fO V_1)^\times \times \,\fP V_2 \times \fP} \Gamma^-(s)
  = \frac{1-q^{1-n_2}t}{1-qt}\int\limits_{(\fO V_1)^\times\times \fP}
  \!\!\!\!\!\!\Gamma_1^-(s+n_2)
  $.
\end{lemma}
\begin{proof}
  Consider the partition $\fP V_2 \times \fP = R \sqcup S$, where
  $R = \{(y,z) \in \fP V_2 \times \fP : \abs z \le \norm y\}$
  and $S = \{(y,z) \in \fP V_2 \times \fP : \norm y < \abs z\}$.
  On $R$, we may write $z = \gcd(y) z'$ for $z'\in \fO$.
  By taking $W = (\fO V_1)^\times \times R$ in Claim~\ref{c:Gamma_int_Gk}
  and performing a change of variables
  using $\abs{\dd\!z} = \norm y \abs{\dd\!z'}$,
  we obtain
  \begin{align*}
    \int\limits_
    {(\fO V_1)^\times \times R } \Gamma^-(s) & =
    \int\limits_{(\fO V_1)^\times \times \fP V_2 \times \fO}
    \abs{z'}^{s-2}
    \norm{y}^{s-c_1}
    \prod_{k=1}^{n_1-c_1}
    G_k(x,\gcd(y))\,
    \dd\!\mu(x,y,z')\\
    & =
      \int\limits_\fO \abs{z'}^{s-2}\dd\!\mu(z') \cdot
      \int\limits_{(\fO V_1)^\times \times \fP V_2}
      \norm{y}^{s-c_1}
      \prod_{e=1}^{n_1-c_1}
      G_e(x,\gcd(y))\,
      \dd\!\mu(x,y)
    \\ &=
         \frac{1-q^{-1}}{1-qt}
         \cdot
      \int\limits_{(\fO V_1)^\times \times \fP V_2}
      \norm{y}^{s-c_1}
      \prod_{e=1}^{n_1-c_1}
      G_e(x,\gcd(y))\,
      \dd\!\mu(x,y).
  \end{align*}

  For fixed $x$, we may view $\prod_{e=1}^{n_1-c_1}G_e(x,\gcd(y))$ as a function of $\norm y$ to
  which Lemma~\ref{lem:repeated_dummy_variables} is applicable.
  Using the Fubini-Tonelli theorem and \eqref{eq:loopless_int} (applied to $\Gamma_1$),
  we thus find that
  \begin{align}
    \nonumber
    \int\limits_{(\fO V_1)^\times \times R } \Gamma^-(s)
    & =
    \frac{1-q^{-1}}{1-qt}
    \cdot
    \frac{1-q^{-n_2}}{1-q^{-1}}
    \int\limits_{(\fO V_1)^\times \times \fP}
      \abs{y}^{s+n_2-c_1-1}
      \prod_{k=1}^{n_1-c_1}
      G_k(x,y)\,
      \dd\!\mu(x,y)
    \\&=
    \frac{1-q^{-n_2}}{1-qt} \int\limits_{(\fO V_1)^\times \times \fP}\Gamma_1^-(s+n_2).
    \label{eq:integral_R}
  \end{align}

  On $S$, we may write $y_v = \pi z y_v'$ for $v\in V_2$ and $y_v'\in \fO$.
  By taking $W = (\fO V_1)^\times \times S$ in Claim~\ref{c:Gamma_int_Gk}
  and performing a change of variables
  using $\abs{\dd\! y_v} = q^{-1} \abs{z} \abs{\dd\!y_v}$, we obtain
  \begin{align}
    \nonumber
    \int\limits_
    {(\fO V_1)^\times \times S} \Gamma^-(s) & =
    q^{-n_2}                                              
    \int\limits_{(\fO V_1)^\times \times \fO V_2 \times \fP}
    \abs{z}^{s+n_2 - c_1 - 1} \prod_{k=1}^{n_1-c_1} G_k(x,z) \dd\!\mu(x,y',z)
    \\ & =
    q^{-n_2} \int\limits_{(\fO V_1)^\times\times \fP}\Gamma_1^-(s+n_2).         
    \label{eq:integral_S}
  \end{align}

  Together, \eqref{eq:integral_R} and \eqref{eq:integral_S} yield
  \begin{align*}
    \int\limits_{(\fO V_1)^\times \times \fP V_2 \times \fP} \Gamma^-(s)
    &=
    \frac{1-q^{-n_2}}{1-qt} \int\limits_{(\fO V_1)^\times \times
      \fP}\Gamma_1^-(s+n_2)
    + q^{-n_2} \int\limits_{(\fO V_1)^\times\times \fP}\Gamma_1^-(s+n_2)
    \\&= \frac{1-q^{1-n_2}t}{1-qt} \int\limits_{(\fO V_1)^\times\times \fP}\Gamma_1^-(s+n_2),
  \end{align*}
  as claimed.
\end{proof}

\begin{cor}
  \label{cor:right_unit_integral}
  $\displaystyle
  \int\limits_{\fP V_1 \times (\fO V_2)^\times \times \fP} \Gamma^-(s)
  = \frac{1-q^{1-n_1}t}{1-qt}\int\limits_{(\fO V_2)^\times\times \fP}
  \!\!\!\!\!\!\Gamma_2^-(s+n_1)
  $.
\end{cor}
\begin{proof}
  Interchange the roles of $\Gamma_1$ and $\Gamma_2$ in
  Lemma~\ref{lem:left_unit_integral}.
\end{proof}

\begin{proof}[Proof of Theorem~\ref{thm:join}]
  By Lemma~\ref{lem:join_via_int},
  the conclusion of Theorem~\ref{thm:join} holds for $\Gamma = \Gamma_1 \join
  \Gamma_2$ if and only if \eqref{eq:join_integral_form} holds.
  We then write $\int_{(\fO V)^\times \times \fP} \Gamma^-(s)$ as a sum
  of three integrals as in~\eqref{eq:3part_int}.
  The three summands in \eqref{eq:3part_int} agree with those in
  \eqref{eq:join_integral_form} by Lemma~\ref{lem:units_everywhere_integral},
  Lemma~\ref{lem:left_unit_integral}, and Corollary~\ref{cor:right_unit_integral}.
\end{proof}

For a cograph $\Gamma$, the following was previously spelled out in
\cite[Prop.\ 8.4]{cico}.

\begin{cor}
  Let $\Gamma$ be a loopless graph.
  Then $W^-_{\Gamma\join\bullet}(X,T) = \frac{1 - X^{-1}T}{1-XT{\phantom .}} \cdot
  W^-_\Gamma(X,X^{-1}T)$.
\end{cor}
\begin{proof}
  We know from \cite[Table~1]{cico} that $W^-_\bullet(X,T) = 1/(1-XT)$.
  Now apply Theorem~\ref{thm:join}.
\end{proof}

\subsection{A new proof of the Cograph Modelling Theorem (Theorem~\ref{thm:cmt})}
\label{ss:proof_cmt}

By combining Theorem~\ref{thm:join} and properties of the rational functions
$W_\Eta$ from \cite[\S 5]{cico}, we obtain a new (and quite short) proof of
Theorem~\ref{thm:cmt}.
This new proof does not use any results from \cite[\S\S 6--7]{cico}, the key
ingredients of the first proof of Theorem~\ref{thm:cmt}.

We first recall terminology and results from \cite[\S 5]{cico}.
Let $\Eta_1$ and $\Eta_2$ be hypergraphs.
Following \cite[\S 3.1]{cico}, the \emph{complete union} $\Eta_1 \circledast
\Eta_2$ is obtained from the disjoint union $\Eta_1\oplus \Eta_2$ by
adjoining all vertices of $\Eta_1$ to each hyperedge of $\Eta_2$ and all
vertices of $\Eta_2$ to each hyperedge of $\Eta_1$.
The following results from \cite{cico} explain the effects of disjoint
and complete unions on the rational functions $W_\Eta$.

\begin{prop}[{\cite[Prop.\ 5.12]{cico}}]
  \label{prop:Hadamard_Eta}
  $W_{\Eta_1\oplus\Eta_2} = W_{\Eta_1} *_T W_{\Eta_2}$ (Hadamard product in~$T$).
\end{prop}

\begin{prop}[{\cite[Prop.\ 5.18]{cico}}]
  \label{prop:freep_Eta}
  Let $\Eta_i$ have $n_i$ vertices and $m_i$ hyperedges.
  For $i = 1, 2$, write $y_i = X^{n_i}$ and $z_i = X^{-m_i}$.
  Then:
  {\small
  \[
    W_{\Eta_1\circledast\Eta_2}    =
    \frac{z_1 z_2 T - 1
      + W_{\Eta_1}(X, \! z_2 T)(1\!-\!z_2 T)(1\! - \! y_1 z_1 z_2 T)
      + W_{\Eta_2}(X, \! z_1 T)(1\!-\!z_1 T)(1\! - \! y_2 z_1 z_2
      T)}{(1-T)(1-y_1y_2 z_1z_2T)}.
    \]
  }
\end{prop}

For a hypergraph $\Eta$, as in \cite[\S 5.4]{cico}, let
$\Eta^{\mathbf 1}$ be obtained from $\Eta$ by adding a
single new hyperedge which contains all vertices.
Let $\Eta^{\mathbf 0}$ be obtained from $\Eta$ by adding a new hyperedge which
does not contain any vertices.

\begin{prop}[{\cite[Prop.\ 5.24]{cico}}]
  \label{prop:01_Eta}
  $W_{\Eta^{\mathbf 1}} = \frac{1-X^{-1}T}{1-T} W_\Eta(X,X^{-1}T)$
  and $W_{\Eta^{\mathbf 0}} = W_{\Eta}$.
\end{prop}

\begin{proof}[New proof of Theorem~\ref{thm:cmt}]
  We show that for each cograph $\Gamma$ on $n$ vertices, there exists a
  hypergraph $\Eta$ with $n$ vertices and $n-1$ hyperedges such that $W^-_\Gamma = W_\Eta$.
  We proceed by structural induction.
  Recall that beginning with a single isolated vertex, cographs are
  constructed by repeatedly taking disjoint unions and joins of smaller
  cographs.
  
  For the base case, if $\Gamma$ consists of a single vertex, then
  clearly $W^-_\Gamma = 1/(1-XT) = W_\Eta$, where $\Eta$ is the
  hypergraph on a single vertex and without hyperedges.

  Let $\Gamma_1$ and $\Gamma_2$ be (co)graphs on $n_1$ and $n_2$ vertices,
  respectively.
  Suppose that $\Eta_1$ and $\Eta_2$ are hypergraphs such that $\Eta_i$ has
  $n_i$ vertices and $n_i-1$ hyperedges and such that $W^-_{\Gamma_i} =
  W_{\Eta_i}$.
  Then Proposition~\ref{prop:Hadamard_Eta} and Proposition~\ref{prop:01_Eta} yield
  \[
    W_{(\Eta_1 \oplus\Eta_2)^{\mathbf 0}} =
    W_{\Eta_1 \oplus\Eta_2} =
    W_{\Eta_1} *_T W_{\Eta_2} =
    W^-_{\Gamma_1} *_T W^-_{\Gamma_2} =
    W^-_{\Gamma_1 \oplus \Gamma_2}.
  \]
  Noting that ${(\Eta_1 \oplus\Eta_2)^{\mathbf 0}}$ has $n_1 + n_2$ vertices and
  $n_1 + n_2 - 1$ hyperedges,
  our claim thus holds for $\Gamma_1 \oplus \Gamma_2$.
  To show that it also holds for $\Gamma_1 \join \Gamma_2$,
  let $\Eta = (\Eta_1 \circledast \Eta_2)^{\mathbf 1}$.
  This hypergraph too has $n_1 + n_2$ vertices and $n_1 + n_2 - 1$ hyperedges.
  Let $n = n_1 + n_2$.
  By Proposition~\ref{prop:freep_Eta} (with $z_i = X^{1-n_i}$), we have
  \begin{align*}
    W_{\Eta_1\circledast\Eta_2} & =
                                  \Bigl(X^{2 - n} T - 1 \\
    & \quad
      + W_{\Eta_1}(X, X^{1-n_2} T)(1- X^{1-n_2} T)(1 - X^{2-n_2} T) \\
                                & \quad
      + W_{\Eta_2}(X, X^{1-n_1} T)(1- X^{1-n_1} T)(1 - X^{2-n_1} T) 
                                  \Bigr)/\bigl((1-T)(1-X^2T)\bigr).
  \end{align*}
  Using Proposition~\ref{prop:01_Eta} and Theorem~\ref{thm:join}, we thus find
  that
  $W_\Eta = W^-_{\Gamma_1 \join \Gamma_2}$ whence our claim holds for
  $\Gamma_1 \join \Gamma_2$.
\end{proof}

\section{Fundamental properties of $W^\sharp_\Gamma$}
\label{s:Wsharp}

In this final section, we derive an explicit graph-theoretic formula for
$W^\sharp_\Gamma$ (Proposition~\ref{prop:Wsharp_formula}) and derive analytic consequences
(Proposition~\ref{prop:Wsharp_poles}).
We also show that the rational functions $W^\sharp_\Gamma$ are well behaved
under joins of graphs (Proposition~\ref{prop:Wsharp_join}).
Finally, we collect formulae for $W^\sharp_\Gamma$ for all graphs on at most
four vertices (Table~\ref{tab:graphs4}).
Recall that $\hat\Gamma$ denotes the reflexive closure of $\Gamma$.
Since $W^\sharp_\Gamma = W^\sharp_{\hat\Gamma}$, we may assume that $\Gamma$
is loopless.

\paragraph{An explicit formula for $W^\sharp_\Gamma$.}
Let $\Gamma = (V,E)$ be a (loopless) graph.
The following notation matches that from \cite{csp}.
For $U\subset V$, let $\Nbh_\Gamma[U]\subset V$
consist of all vertices from $U$ as well as all vertices adjacent to some
vertex from $U$.
Let $\nbhcnt_\Gamma(U) = \card{\Nbh_\Gamma[U]\setminus U}$, the
number of vertices in $V\setminus U$ with a neighbour in $U$.
Recall that $\WOhat(V)$ denotes the poset of flags of subsets of $V$.

\begin{prop}
  \label{prop:Wsharp_formula}
  Let $\Gamma = (V,E)$ be a (loopless) graph.
  Then
    \begin{equation}
      \label{eq:Wsharp_master}
          W^\sharp_\Gamma(X^{-1},T) = \sum_{y \in \WOhat(V)}(1-X)^{\card{\sup(y)}}
          \prod_{U\in y} \frac{X^{\nbhcnt_\Gamma(U)}T}{1-X^{\nbhcnt_\Gamma(U)}T}.
    \end{equation}
\end{prop}
\begin{proof}
  Let $\Eta = \Adj(\hat\Gamma)$ so that $W^\sharp_\Gamma = W_\Eta$.
  Using the notation from Theorem~\ref{thm:hypergraph_master},
  as $\hat\Gamma$ is reflexive, for each $U\subset V$, we have
  $\pth{U} = \Nbh_\Gamma[U]$.
  Now apply Theorem~\ref{thm:hypergraph_master}.
\end{proof}

\begin{rem}
  In \cite{csp}, the cardinalities of the set $\Nbh_\Gamma[U]$ featured
  crucially in an explicit formula \cite[Cor.~B]{csp} for the coefficient of
  $T$ in $W^-_\Gamma(X,T)$.
  At present, no explicit combinatorial formula for $W^-_\Gamma$ akin to
  \eqref{eq:Wsharp_master} is known; see \cite[Question~1.8(iii)]{cico}.
\end{rem}



\paragraph{Local poles.}
Let $\Eta$ be a hypergraph.
Theorem~\ref{thm:hypergraph_master} shows that
$W_\Eta$ can be written in the form
\begin{equation}
  \label{eq:local_pole_factorisation}
  W^\sharp_\Gamma = \frac{f(X,T)}{\prod_{i=1}^N (1-X^{a_i} T)}
\end{equation}
for $f(X,T) \in \ZZ[X^{\pm 1}, T]$ and $a_1,\dotsc,a_N\in \ZZ$.
For any graph $\Gamma$, the same conclusion holds for $W^\sharp_\Gamma$.
By Theorem~\ref{thm:cmt}, it also holds for $W^-_\Gamma$ if $\Gamma$ is a cograph.
In any case, we may assume that $f(X,X^{-a_i}) \not= 0$ for $i=1,\dotsc,N$.
It is then easy to see that the representation in
\eqref{eq:local_pole_factorisation} is unique up to the order of the $a_i$.
We refer to the integers $a_1,\dotsc,a_N$ as the \emph{local poles} 
of~$W_\Eta$; multiplicities of local poles are understood in the evident way.
The local poles of $W_\Eta$ are precisely the real parts of the poles
of the meromorphic function $W_\Eta(q,q^{-s})$, where $q > 1$
is arbitrary.

Even for cographs~$\Gamma$, the local poles of $W^-_\Gamma$
remain mysterious.
In particular, positive and negative local poles can arise and
no single number appears as a universal local pole of all $W^-_\Gamma$;
see \cite[Table~2]{cico}.
In contrast, the $W^\sharp_\Gamma$ are much better behaved.

\begin{prop}
  \label{prop:Wsharp_poles}
  Let $\Gamma = (V,E)$ be a (loopless) graph.
  Then:
  \begin{enumerate}[(a)]
  \item
    \label{prop:Wsharp_poles1}
    Each local pole of $W^\sharp_\Gamma$ is nonpositive.
  \item
    \label{prop:Wsharp_poles2}
    Let $\Gamma$ have $c$ connected components.
    Then $W^\sharp_\Gamma$ has a pole of order $c+1$ at $T = 1$.
    That is, $(1-T)^{c+1}W^\sharp_\Gamma$ is regular at $T = 1$
    and $(1-T)^{c+1}W^\sharp_\Gamma \Biggm\vert_{T \gets 1} \not= 0$.
  \end{enumerate}
\end{prop}
\begin{proof}
  The first part follows from Proposition~\ref{prop:Wsharp_formula}
  since $\nbhcnt_\Gamma(U) \ge 0$ for each $U\subset V$.
  It is easy to see that $\nbhcnt_\Gamma(U) = 0$ is equivalent to $U$ being a disjoint
  union of (zero or more) connected components of $\Gamma$.
  In particular, if $U_0 \subset \dotsb \subset U_r \subset V$ with
  $\nbhcnt_\Gamma(U_i) = 0$ for $i=0,\dotsc,r$, then $r \le c$.
  This proves that $(1-T)^{c+1} W^\sharp_\Gamma$ is regular at $T = 1$.
  Let $\mathcal F\subset \WOhat(V)$ consist of those flags $y = (U_0 \subsetneq
  U_1\subsetneq \dotsb \subsetneq U_r)$ of subsets of $V$ such that precisely
  $c+1$ of the $U_i$ satisfy $\nbhcnt_\Gamma(U_i) = 0$.
  Such a flag $y$ necessarily satisfies $U_0 = \emptyset$ and $\sup(y) = U_r = V$,
  in addition to $r\ge c$.
  The elements of $\mathcal F$ are precisely the flags that contribute nonzero
  summands to $(1-T)^{c+1} W^\sharp_\Gamma \Biggm\vert_{T\gets 1}$.
  It remains to rule out possible cancellations.
  Write $n = \card{V}$.
  Evaluating at $X = 1/2$,
  we obtain
  \begin{equation}
    \label{eq:pole_c+1}
    (1-T)^{c+1} W^\sharp_\Gamma(2,T) \Biggm\vert_{T \gets 1}
    = \sum_{y\in \mathcal F}
    2^{-n}
    \prod_{\substack{U\in y\\ \nbhcnt_\Gamma(U) > 0}} \frac{2^{-\nbhcnt_\Gamma(U)}T}{1-2^{-\nbhcnt_\Gamma(U)}T}.
  \end{equation}
  Given $y \in \mathcal F$, let $f(y) = \#\{ U \in y :
  \nbhcnt_\Gamma(U) > 0\} \ge 0$.
  Viewed as a power series in $T$, the coefficient of $T^{f(y)}$ of the
  summand corresponding to $y$ on the right-hand side of \eqref{eq:pole_c+1}
  is positive and all other coefficients are nonnegative.
  We conclude that the right-hand side of \eqref{eq:pole_c+1} is nonzero.
\end{proof}

\begin{rem}
The description of the rational number given in \eqref{eq:pole_c+1} is
reminiscent of the definition of the constant $c_d$
in \cite[(6.1)]{SVV}. The latter constant occurs as a special value of the
reduced and topological subgroup zeta functions of the free class-$2$-nilpotent
groups of rank~$d$; cf.\ \cite[Thms.~6.8 and~6.11]{SVV}. In both
contexts, we lack a conceptual interpretation of these rational
numbers: a group-theoretic one in the case of $c_d$, a graph-theoretic one in
the current case.
\end{rem}

\paragraph{Joins and disjoint unions.}
\begin{proof}[Proof of Proposition~\ref{prop:Wsharp_join}]
  Recall that $\hat\Gamma$ denotes the reflexive closure of a graph $\Gamma$.
  Clearly, $\widehat{\Gamma_1 \oplus \Gamma_2} = \hat \Gamma_1 \oplus \hat
  \Gamma_2$ and $\widehat{\Gamma_1 \join \Gamma_2} = \hat \Gamma_1 \join \hat
  \Gamma_2$.
  Using the notation from \cite[\S 3.1]{cico}, we thus have
  $\Adj(\widehat{\Gamma_1 \oplus \Gamma_2}) = \Adj(\hat\Gamma_1) \oplus
  \Adj(\hat\Gamma_2)$
  and
  $\Adj(\widehat{\Gamma_1 \join \Gamma_2}) = \Adj(\hat\Gamma_1) \circledast
  \Adj(\hat\Gamma_2)$ (complete union).
  The first claim follows from
  Proposition~\ref{prop:Hadamard_Eta}
  and the second from \cite[Cor.\ 5.18]{cico} with $m_i = n_i$.
\end{proof}

\paragraph{Formulae for small graphs.}
Table~\ref{tab:graphs4} lists the rational functions $W^\sharp_\Gamma$ for
all loopless graphs on at most four vertices.
Table~\ref{tab:graphs4} is a sequel to \cite[Table~1]{cico} which lists
$W^+_\Gamma$ and~$W^-_\Gamma$ for the same class of graphs.
Formulae for $W^\sharp_\Gamma$ for all $1252$ graphs on at most seven vertices
are available on the first author's home page.
These functions were computed using Proposition~\ref{prop:Wsharp_formula}.
For graphs on at most six vertices, they agree with computations performed
using \textsf{Zeta}~\cite{Zeta}.
(For some graphs on seven vertices, using the algorithms from~\cite{cico}
implemented in \textsf{Zeta} to compute $W^\sharp_\Gamma$ requires a significant
amount of memory which renders these computations impractical on a typical
desktop computer.)

\paragraph{Some infinite families of graphs.}
Some of the rational functions in Table~\ref{tab:graphs4} were previously
known in the sense that they follow from existing results in the literature.
First, let $\CG_n$ and $\DG_n$ denote the complete and edgeless graph on $n$
vertices, respectively.
By \cite[Prop.\ 1.5]{ask}, we have $W^\sharp_{\CG_d} = \frac{1-X^{-d} T}{(1-T)^2}$.
Using a result due to Brenti~\cite[Thm~3.4]{Bre94}, \cite[Cor.\ 5.17]{ask}
provides an explicit formula for $W^\sharp_{\DG_n}$ in terms of permutation
statistics on the hyperoctahedral group
$\mathrm B_n = \{ \pm 1\} \wr \mathrm S_n$.
More generally, \cite[Cor.\ 5.11]{CMR24b} provides an explicit formula for
$W^\sharp_{\CG_{d_1} \oplus \dotsb \oplus \CG_{d_n}}$ in terms of permutations
statistics on $(n+1)$-coloured permutations on $n$ letters.
This includes the aforementioned known formulae for $W^\sharp_{\CG_d}$ ($n=1$,
$d_1 = d$) and $W^\sharp_{\DG_n}$ ($d_1 = \dotsb = d_n = 1$) as special cases.
In this way, $10$ of the $18$ formulae in Table~\ref{tab:graphs4} are
in fact explained by \cite{CMR24b}.

\phantomsection
\addcontentsline{toc}{section}{Acknowledgements}
\subsection*{Acknowledgements}

This work has emanated from research
funded by the Deutsche Forschungsgemeinschaft
(DFG, German Research Foundation) — SFB-TRR 358/1 2023 — 491392403 and
Taighde Éireann – Research Ireland under grant number
22/FFP-P/11449.
We also acknowledge support by the
International Centre for Theoretical Sciences (ICTS) through the
programme \itemph{Combinatorial Methods in Enumerative Algebra} (code:
ICTS/cmea2024/12).

{
  \def\emph{\itemph}
  \bibliographystyle{abbrv}
  \bibliography{plumbing}
  \def\emph{\bfempf}
}

\vspace*{2em}
\noindent
{\footnotesize
\begin{minipage}[t]{0.60\textwidth}
  Tobias Rossmann\\
  School of Mathematical and Statistical Sciences \\
  University of Galway \\
  Galway \\
  Ireland \\
  \quad\\
  E-mail: \href{mailto:tobias.rossmann@universityofgalway.ie}{tobias.rossmann@universityofgalway.ie}
\end{minipage}
\hfill
\begin{minipage}[t]{0.38\textwidth}
  Christopher Voll\\
  Fakult\"at f\"ur Mathematik\\
  Universit\"at Bielefeld\\
  D-33501 Bielefeld\\
  Germany\\
  \quad\\
  E-mail: \href{mailto:C.Voll.98@cantab.net}{C.Voll.98@cantab.net}
\end{minipage}
}

\begin{landscape}
  \begin{table}
    \centering
    \small
    \begin{tabular}{m{3cm} | l}
      $\Gamma$ & $W^\sharp_\Gamma(X,T)$ \\
      \hline
      \begin{tikzpicture}[scale=0.2]
        \tikzstyle{Black Vertex}=[draw=black, fill=black, shape=circle, scale=0.2]
        \tikzstyle{Solid Edge}=[-]
        \node [style=Black Vertex] (0) at (0, 0) {0};
      \end{tikzpicture}
      & $\frac{1 - X^{-1} T}{(1 - T)^2}$
      \\
      
      \begin{tikzpicture}[scale=0.2]
        \tikzstyle{Black Vertex}=[draw=black, fill=black, shape=circle, scale=0.2]
        \tikzstyle{Solid Edge}=[-]
        \node [style=Black Vertex] (0) at (0, 0) {0};
        \node [style=Black Vertex] (1) at (3, 0) {1};
      \end{tikzpicture}
      & $\frac{1 + X^{-2} T - 4X^{-1} T + T  + X^{-2}T^2}{(1-T)^3}$
      \\

      \begin{tikzpicture}[scale=0.2]
        \tikzstyle{Black Vertex}=[draw=black, fill=black, shape=circle, scale=0.2]
        \tikzstyle{Solid Edge}=[-]
        \node [style=Black Vertex] (0) at (0, 0) {0};
        \node [style=Black Vertex] (1) at (3, 0) {1};
        \draw (0) to (1);
      \end{tikzpicture}
      & $\frac{1-X^{-2}T}{(1-T)^2}$
      \\
      
      \begin{tikzpicture}[scale=0.2]
        \tikzstyle{Black Vertex}=[draw=black, fill=black, shape=circle, scale=0.2]
        \tikzstyle{Solid Edge}=[-]
        \node [style=Black Vertex] (0) at (0, 0) {0};
        \node [style=Black Vertex] (1) at (3, 0) {1};
        \node [style=Black Vertex] (2) at (6, 0) {2};
      \end{tikzpicture}
      & $\frac{1 - X^{-3}T + 6 X^{-2}T - 12X^{-1}T + 4 T + T^2 - 4X^{-3}T^2 + 12 X^{-2} T^2 - 6 X^{-1}T^2 - X^{-3}T^3}{(1 - T)^4}$
      \\

      \begin{tikzpicture}[scale=0.2]
        \tikzstyle{Black Vertex}=[draw=black, fill=black, shape=circle, scale=0.2]
        \tikzstyle{Solid Edge}=[-]
        \node [style=Black Vertex] (0) at (0, 0) {0};
        \node [style=Black Vertex] (1) at (3, 0) {1};
        \node [style=Black Vertex] (2) at (6, 0) {2};
        \draw (0) to (1);
      \end{tikzpicture}
      & $\frac{1 + X^{-3}T - 2X^{-2}T - 2X^{-1}T + T  + X^{-3} T^2 }{(1-T)^3}$
      \\

      \begin{tikzpicture}[scale=0.2]
        \tikzstyle{Black Vertex}=[draw=black, fill=black, shape=circle, scale=0.2]
        \tikzstyle{Solid Edge}=[-]
        \node [style=Black Vertex] (0) at (0, 0) {0};
        \node [style=Black Vertex] (1) at (3, 0) {1};
        \node [style=Black Vertex] (2) at (6, 0) {2};
        \draw (0) to (1);
        \draw (1) to (2);
      \end{tikzpicture}
      & $\frac{1 + X^{-3}T - 4X^{-2}T + X^{-1}T + X^{-4}T^2}{(1 - X^{-1}T)(1 - T)^2}$
      \\

      \begin{tikzpicture}[scale=0.2]
        \tikzstyle{Black Vertex}=[draw=black, fill=black, shape=circle, scale=0.2]
        \tikzstyle{Solid Edge}=[-]
        \node [style=Black Vertex] (0) at (0, 0) {0};
        \node [style=Black Vertex] (1) at (-3, -3) {2};
        \node [style=Black Vertex] (2) at (3, -3) {3};
        \draw (0) to (1);
        \draw (0) to (2);
        \draw (1) to (2);
      \end{tikzpicture}
      & $\frac{1-X^{-3}T}{(1-T)^2}$
      \\
      
      \begin{tikzpicture}[scale=0.2]
        \tikzstyle{Black Vertex}=[draw=black, fill=black, shape=circle, scale=0.2]
        \tikzstyle{Solid Edge}=[-]
        \node [style=Black Vertex] (0) at (0, 0) {0};
        \node [style=Black Vertex] (1) at (3, 0) {1};
        \node [style=Black Vertex] (2) at (6, 0) {2};
        \node [style=Black Vertex] (3) at (9, 0) {3};
      \end{tikzpicture}
      & $\bigl(1 + X^{-4} T - 8 X^{-3} T + 24 X^{-2} T - 32 X^{-1} T + 11 T + 11 X^{-4} T^2 - 56 X^{-3} T^2 + 96 X^{-2} T^2 - 56 X^{-1} T^2$
      \\
      & \,\, $+ 11 T^2 + 11 X^{-4} T^3 - 32 X^{-3} T^3 + 24 X^{-2} T^3 - 8 X^{-1} T^3 + T^3+ X^{-4} T^4\bigr)/(1- T)^5$
      \\

      \begin{tikzpicture}[scale=0.2]
        \tikzstyle{Black Vertex}=[draw=black, fill=black, shape=circle, scale=0.2]
        \tikzstyle{Solid Edge}=[-]
        \node [style=Black Vertex] (0) at (0, 0) {0};
        \node [style=Black Vertex] (1) at (3, 0) {1};
        \node [style=Black Vertex] (2) at (6, 0) {2};
        \node [style=Black Vertex] (3) at (9, 0) {3};
        \draw (0) to (1);
      \end{tikzpicture}
      & $\frac{1 - X^{-4} T + 4 X^{-3} T - 2X^{-2} T - 8 X^{-1} T + 4 T - 4X^{-4}T^2 + 8X^{-3} T^2 + 2 X^{-2} T^2 - 4 X^{-1} T^2 + T^2 - X^{-4} T^3}{(1 - T)^4}$
      \\

      \begin{tikzpicture}[scale=0.2]
        \tikzstyle{Black Vertex}=[draw=black, fill=black, shape=circle, scale=0.2]
        \tikzstyle{Solid Edge}=[-]
        \node [style=Black Vertex] (0) at (0, 0) {0};
        \node [style=Black Vertex] (1) at (3, 0) {1};
        \node [style=Black Vertex] (2) at (6, 0) {2};
        \node [style=Black Vertex] (3) at (9, 0) {3};
        \draw (2) to (3);
        \draw (0) to (1);
      \end{tikzpicture}
      & $\frac{1 + X^{-4}T - 4X^{-2}T + T  + X^{-4}T^2}{(1 - T)^3}$
      \\

      \begin{tikzpicture}[scale=0.2]
        \tikzstyle{Black Vertex}=[draw=black, fill=black, shape=circle, scale=0.2]
        \tikzstyle{Solid Edge}=[-]
        \node [style=Black Vertex] (0) at (0, 0) {0};
        \node [style=Black Vertex] (1) at (3, 0) {1};
        \node [style=Black Vertex] (2) at (6, 0) {2};
        \node [style=Black Vertex] (3) at (9, 0) {3};
        \draw (0) to (1);
        \draw (1) to (2);
      \end{tikzpicture}
      & $\frac{1 - X^{-4} T + 6 X^{-3} T - 10 X^{-2} T + T - 2 X^{-5} T^2 + X^{-4} T^2 + 8 X^{-3} T^2 + X^{-2} T^2 - 2 X^{-1} T^2 + X^{-6} T^3 - 10 X^{-4} T^3 + 6 X^{-3} T^3 - X^{-2} T^3 + X^{-6} T^4}{(1 - X^{-1}T)^2 (1 - T)^3}$
      \\
      
      \begin{tikzpicture}[scale=0.2]
        \tikzstyle{Black Vertex}=[draw=black, fill=black, shape=circle, scale=0.2]
        \tikzstyle{Solid Edge}=[-]
        \node [style=Black Vertex] (0) at (0, 0) {0};
        \node [style=Black Vertex] (1) at (3, 0) {1};
        \node [style=Black Vertex] (2) at (6, 0) {2};
        \node [style=Black Vertex] (3) at (9, 0) {3};
        \draw (0) to (1);
        \draw (1) to (2);
        \draw (2) to (3);
      \end{tikzpicture}
      & $\frac{1 + X^{-4} T - 6 X^{-2} T + 2 X^{-1} T - 2 X^{-5} T^2 + 6 X^{-4} T^2 - X^{-2} T^2 - X^{-6} T^3}{(1 - X^{-1}T)^2(1 - T)^2}$
      \\
      
      \begin{tikzpicture}[scale=0.2]
        \tikzstyle{Black Vertex}=[draw=black, fill=black, shape=circle, scale=0.2]
        \tikzstyle{Solid Edge}=[-]
        \node [style=Black Vertex] (0) at (0, -2) {0};
        \node [style=Black Vertex] (1) at (4, -2) {1};
        \node [style=Black Vertex] (2) at (4, 2) {2};
        \node [style=Black Vertex] (3) at (0, 2) {3};
        \draw (0) to (1);
        \draw (1) to (2);
        \draw (2) to (3);
        \draw (3) to (0);
      \end{tikzpicture}
      & $\frac{1 + 3 X^{-4} T  - 8 X^{-3} T + 3 X^{-2} T + X^{-6} T^2}{(1 - X^{-2} T)(1 - T)^2}$
      \\
      
      \begin{tikzpicture}[scale=0.2]
        
        \tikzstyle{Black Vertex}=[draw=black, fill=black, shape=circle, scale=0.2]
        \tikzstyle{Solid Edge}=[-]
        \node [style=Black Vertex] (0) at (0, 0) {0};
        \node [style=Black Vertex] (1) at (4, 0) {1};
        \node [style=Black Vertex] (2) at (4, 4) {2};
        \node [style=Black Vertex] (3) at (0, 4) {3};
        \draw (0) to (1);
        \draw (0) to (2);
        \draw (1) to (2);
        \draw (2) to (3);
        \draw (3) to (0);
      \end{tikzpicture}
      & $\frac{1 + X^{-4} T - 4 X^{-3} T + X^{-2} T  + X^{-6}T^2}{(1 - X^{-2} T) (1 - T)^2}$
      \\
      
      \begin{tikzpicture}[scale=0.2]
        \tikzstyle{Black Vertex}=[draw=black, fill=black, shape=circle, scale=0.2]
        \tikzstyle{Solid Edge}=[-]
        \node [style=Black Vertex] (0) at (0, 0) {0};
        \node [style=Black Vertex] (1) at (0, 6) {1};
        \node [style=Black Vertex] (2) at (4, 3) {2};
        \node [style=Black Vertex] (3) at (8, 3) {3};
        \draw (0) to (1);
        \draw (0) to (2);
        \draw (1) to (2);
      \end{tikzpicture}
      & $\frac{1 + X^{-4} T - 2 X^{-3} T - 2 X^{-1} T + T + X^{-4}T^2}{(1 - T)^3}$
      \\
      
      \begin{tikzpicture}[scale=0.2]
        \tikzstyle{Black Vertex}=[draw=black, fill=black, shape=circle, scale=0.2]
        \tikzstyle{Solid Edge}=[-]
        \node [style=Black Vertex] (0) at (0, 0) {0};
        \node [style=Black Vertex] (1) at (0, 6) {1};
        \node [style=Black Vertex] (2) at (4, 3) {2};
        \node [style=Black Vertex] (3) at (8, 3) {3};
        \draw (0) to (1);
        \draw (0) to (2);
        \draw (1) to (2);
        \draw (2) to (3);
      \end{tikzpicture}
      & $\frac{1 + X^{-4}T - 2X^{-3}T - 2 X^{-2}T + X^{-1} T + X^{-5}T^2}{(1 - X^{-1}T)(1 - T)^2}$
      \\
      
      \begin{tikzpicture}[scale=0.2]
        \tikzstyle{Black Vertex}=[draw=black, fill=black, shape=circle, scale=0.2]
        \tikzstyle{Solid Edge}=[-]
        \node [style=Black Vertex] (0) at (0, 0) {0};
        \node [style=Black Vertex] (1) at (0, 4) {1};
        \node [style=Black Vertex] (2) at (-3, -3) {2};
        \node [style=Black Vertex] (3) at (3, -3) {3};
        \draw (0) to (1);
        \draw (0) to (2);
        \draw (0) to (3);
      \end{tikzpicture}
               &
                 $\frac{
                 1
                 - X^{-4}T
                 + 6X^{-3}T
                 - 12 X^{-2}T
                 + 4 X^{-1} T
                 - 4X^{-5}T^2
                 + 12X^{-4}T^2
                 - 6X^{-3}T^2
                 + X^{-2} T^2 
                 - X^{-6}T^3
                 }{(1 - X^{-1}T)^2(1 - T)^2}$
      \\
      
      \begin{tikzpicture}[scale=0.2]
        \tikzstyle{Black Vertex}=[draw=black, fill=black, shape=circle, scale=0.2]
        \tikzstyle{Solid Edge}=[-]
        \node [style=Black Vertex] (0) at (-1, 0) {0};
        \node [style=Black Vertex] (1) at (-1, 5) {1};
        \node [style=Black Vertex] (2) at (-5, -4) {2};
        \node [style=Black Vertex] (3) at (3, -4) {3};
        \draw (0) to (1);
        \draw (0) to (2);
        \draw (0) to (3);
        \draw (1) to (2);
        \draw (1) to (3);
        \draw (2) to (3);
      \end{tikzpicture}
      & $\frac{1 - X^{-4} T}{(1-T)^2}$
    \end{tabular}
    \caption{$W^\sharp_\Gamma$ for all graphs on at most four vertices}
    \label{tab:graphs4}
  \end{table}
\end{landscape}

\end{document}